\documentclass[a4paper,11pt]{amsart}
\usepackage[T1]{fontenc}
\usepackage{lmodern}
\usepackage[utf8]{inputenc}
\usepackage{amssymb}
\usepackage{tikz-cd}

\numberwithin{equation}{section}

\newtheorem{thm}{Theorem}[section]
\newtheorem{prop}[thm]{Proposition}
\newtheorem{lem}[thm]{Lemma}
\newtheorem{cor}[thm]{Corollary}

\theoremstyle{definition}
\newtheorem{defn}[thm]{Definition}

\theoremstyle{remark}
\newtheorem{rem}[thm]{Remark}
\newtheorem*{claim}{Claim}

\newcommand{\Z}{\mathbb{Z}}

\newcommand{\R}{\mathbb{R}}

\DeclareMathOperator{\Hom}{Hom}
\DeclareMathOperator{\Ext}{Ext}
\DeclareMathOperator{\End}{End}
\DeclareMathOperator{\Ker}{Ker}
\DeclareMathOperator{\Coker}{Cok}

\newcommand{\trivrep}{\mathbf{1}}
\newcommand{\aff}{\mathrm{aff}}
\newcommand{\St}{\mathrm{St}}
\DeclareMathOperator{\Stab}{Stab}
\newcommand{\id}{\mathrm{id}}
\newcommand*{\opposite}[1]{#1^{\mathrm{op}}}
\newcommand{\opp}{\mathrm{op}}
\DeclareMathOperator{\val}{val}
\DeclareMathOperator{\Aut}{Aut}
\newcommand{\Refs}[1]{\mathrm{Ref}(#1)}
\newcommand{\coroot}[1]{#1^\vee}
\newcommand{\red}{\mathrm{red}}

\title{Parabolic inductions for pro-$p$-Iwahori Hecke algebras}
\author{Noriyuki Abe}
\address{Department of Mathematics, Hokkaido University, Kita 10, Nishi 8, Kita-Ku, Sapporo, Hokkaido, 060-0810, Japan}
\email{abenori@math.sci.hokudai.ac.jp}

\begin{document}
\begin{abstract}
We give some properties of parabolic inductions and their adjoint functors for pro-$p$-Iwahori Hecke algebras.
\end{abstract}
\maketitle

\section{Introduction}
Let $F$ be a non-archimedean local field with residue characteristic $p$, $G$ a connected reductive group over $F$ and $C$ a commutative ring.
Fix a pro-$p$-Iwahori subgroup $I(1)$ of $G$ and let $\mathcal{H}$ be the space of $I(1)$-biinvariant $C$-valued functions with compact support.
Then via the convolution product, $\mathcal{H}$ has the structure of an algebra, called \emph{pro-$p$-Iwahori Hecke algebra}.

We are mainly interested in ``mod $p$ case'', namely when $C$ is an algebraically closed field of characteristic $p$.
In this case, this algebra has an important role to study modulo $p$ representations of $G$.
For example, in the proof of the classification theorem \cite{arXiv:1412.0737_accepted}, it has a crucial role.
Motivated by this, we study the representation theory of $\mathcal{H}$.

As usual in the theory of reductive groups, the parabolic induction is one of the most important tool in the study of $\mathcal{H}$-modules.
This functor is studied in \cite{MR2728487,arXiv:1406.1003_accepted,MR3437789}.
In particular, using parabolic induction, the simple $\mathcal{H}$-modules are classified in terms of supersingular modules~\cite[Theorem~1.1]{arXiv:1406.1003_accepted}. (Note that the supersingular modules are classified by Ollivier~\cite{MR3263136} and Vign\'eras~\cite{Vigneras-prop-III}.)
This is an analogous of the main theorem of \cite{arXiv:1412.0737_accepted}.

In this paper, we study more details of the parabolic induction.

\subsection{Results}
Let $P$ be a parabolic subgroup and $\mathcal{H}_P$ the pro-$p$-Iwahori Hceke algebra of its Levi part.
Then the parabolic induction $I_P$ is defined by
\[
I_P(\sigma) = \Hom_{(\mathcal{H}_P^-,j_P^{-*})}(\mathcal{H},\sigma)
\]
using ``negative subalgebra'' $\mathcal{H}_P^-$ of $\mathcal{H}_P$ and a certain algebra homomorphism $j_P^{-*}\colon \mathcal{H}_P^-\to \mathcal{H}$.
Replacing ``negative subalgebra'' with ``positive subalgebra'' $\mathcal{H}_P^{+}$, we also have another notion of ``inductions''.
Moreover, we also have other maps $j_P^{\pm}\colon \mathcal{H}_P^{\pm}\to \mathcal{H}$.
Therefore, we have four ``inductions''.
We compare these inductions and give relations.
It turns out that we only have two inductions $I_P$ and $I'_P = \Hom_{(\mathcal{H}_P^-,j_P^-)}(\mathcal{H},\cdot)$.
The other inductions are described by $I_P$ or $I'_P$ (Proposition~\ref{prop:comparison with two more inductions}).
As a corollary, we get  a new description of $I_P$ (Corollary~\ref{cor:parabolic induction via tensor product 2}).

The functor $I_P$ has the left and right adjoint functors which are denoted by $L_P$ and $R_P$, respectively.
We calculate the image of simple modules of $\mathcal{H}$ by these functors.
As proved in \cite{arXiv:1406.1003_accepted}, any simple modules are got in the following way: starting with a simple supersingular module, take the generalized Steinberg module and apply the parabolic induction.
So to calculate the image of simple modules of $\mathcal{H}$ by these functors, it is sufficient to calculate the images of parabolic induction, generalized Steinberg modules and supersingular modules.
We give such descriptions (Proposition~\ref{prop:parabolic induction and its adjoint}, Corollary~\ref{cor:left adjoint and parabolic induction}, Proposition~\ref{prop:left adj of St}, Proposition~\ref{prop:right adj of St}, Proposition~\ref{prop:vanishing of adjoint+supsersingular}).
In particular, the image of simple modules by these functors is zero or again irreducible (Theorem~\ref{thm:adjoint and simple modules}).
The same calculations for the representations of the group is given in \cite{AHHV2}.

\subsection{Applications}
The main application is the calculation of the extension groups between simple $\mathcal{H}$-modules.
Since the left adjoint functor $L_P$ is exact, we have $\Ext^i(\pi,I_P(\sigma))\simeq \Ext^i(L_P(\pi),\sigma)$ for $\mathcal{H}$-module $\pi$ and $\mathcal{H}_P$-module $\sigma$.
We also have a similar formula for the functor $R_P$.
Since any simple $\mathcal{H}$-module is given by the parabolic induction from the generalized Steinberg modules, these formulas and results in this paper enable us to deduce the calculation of the extension groups to that of the extension groups between generalized Steinberg modules.

For the extension groups between the generalized Steinberg modules, it turns out that we need to study the functor $I'_P$.
Especially, we need the relations between the functors $I_P$ and $I'_P$.
Such analysis and the calculation of extension groups will be appeared in sequel in which we use results in this paper.

\subsection{Organization of the paper}
In the next section we give notation and recall some results mainly from \cite{MR3484112}.
In Section~\ref{sec:A ftilration on parabolic inductions}, we define a certain filtration on $I_P(\sigma)$.
This filtration is analogous of the filtration of the parabolic induction of the group representation coming from the Bruhat cell.
Some properties of the filtration are proved in \cite{arXiv:1412.0737_accepted,AHHV2} in the group case and we give analogous of them.
These properties is used in Section~\ref{sec:Adjoint functors}.
The study of four ``inductions'' is done in Section~\ref{sec:Inductions}.
In Section~\ref{sec:Adjoint functors}, we calculate the image of parabolic induction, generalized Steinberg modules and supersingular modules by $L_P$ and $R_P$.

\subsection*{Acknowledgments}
I thank Guy Henniart and Marie-France Vign\'eras for helpful discussions.
Most of this work was done during my pleasant stay at Institut de mathématiques de Jussieu.
The work is supported by JSPS KAKENHI Grant Number 26707001.

\section{Preliminaries}\label{sec:Preliminaries}
\subsection{Pro-$p$-Iwahori Hecke algebra}\label{subsec:Prop-p-Iwahori Hecke algebra}
Let $\mathcal{H}$ be a pro-$p$-Iwahori Hecke algebra over a commutative ring $C$~\cite{MR3484112}.
We study modules over $\mathcal{H}$ in this paper.
\emph{In this paper, a module means a right module.}
The algebra $\mathcal{H}$ is defined with a combinatorial data $(W_\aff,S_\aff,\Omega,W,W(1),Z_\kappa)$ and a parameter $(q,c)$.

We recall the definitions.
The data satisfy the following.
\begin{itemize}
\item $(W_\aff,S_\aff)$ is a Coxeter system.
\item $\Omega$ acts on $(W_\aff,S_\aff)$.
\item $W = W_\aff\rtimes \Omega$.
\item $Z_\kappa$ is a finite commutative group.
\item The group $W(1)$ is an extension of $W$ by $Z_\kappa$, namely we have an exact sequence $1\to Z_\kappa\to W(1)\to W\to 1$.
\end{itemize}
The subgroup $Z_\kappa$ is normal in $W(1)$.
Hence the conjugate action of $w\in W(1)$ induces an automorphism of $Z_\kappa$, hence of the group ring $C[Z_\kappa]$.
We denote it by $c\mapsto w\cdot c$.

Let $\Refs{W_\aff}$ be the set of reflections in $W_\aff$ and $\Refs{W_\aff(1)}$ the inverse image of $\Refs{W_\aff}$ in $W(1)$.
The parameter $(q,c)$ is maps $q\colon S_\aff\to C$ and $c\colon \Refs{W_\aff(1)}\to C[Z_\kappa]$ with the following conditions. (Here the image of $s$ by $q$ (resp.\ $c$) is denoted by $q_s$ (resp.\ $c_s$).)
\begin{itemize}
\item For $w\in W$ and $s\in S_\aff$, if $wsw^{-1}\in S_\aff$ then $q_{wsw^{-1}} = q_s$.
\item For $w\in W(1)$ and $s\in \Refs{W_\aff(1)}$, $c_{wsw^{-1}} = w\cdot c_s$.
\item For $s\in \Refs{W_\aff(1)}$ and $t\in Z_\kappa$, we have $c_{ts} = tc_s$.
\end{itemize}
Let $S_\aff(1)$ be the inverse image of $S_\aff$ in $W(1)$.
For $s\in S_\aff(1)$, we write $q_s$ for $q_{\bar{s}}$ where $\bar{s}\in S_\aff$ is the image of $s$.
The length function on $W_\aff$ is denoted by $\ell$ and its inflation to $W$ and $W(1)$ is also denoted by $\ell$.

The $C$-algebra $\mathcal{H}$ is a free $C$-module and has a basis $\{T_w\}_{w\in W(1)}$.
The multiplication is given by
\begin{itemize}
\item (Quadratic relations) $T_{s}^2 = q_sT_{s^2} + c_sT_s$ for $s\in S_\aff(1)$.
\item (Braid relations) $T_{vw} = T_vT_w$ if $\ell(vw) = \ell(v) + \ell(w)$.
\end{itemize}
We extend $q\colon S_\aff\to C$ to $q\colon W\to C$ as follows.
For $w\in W$, take $s_1,\dots,s_l$ and $u\in\Omega$ such that $w = s_1\dotsm s_l u$.
Then put $q_w = q_{s_1}\dotsm q_{s_l}$.
From the definition, we have $q_{w^{-1}} = q_w$.
We also put $q_w = q_{\overline{w}}$ for $w\in W(1)$ with the image $\overline{w}$ in $W$.

\subsection{The data from a group}
Let $F$ be a non-archimedean local field, $\kappa$ its residue field, $p$ its residue characteristic and $G$ a connected reductive group over $F$.
We can get the data in the previous subsection from $G$ as follows.
See \cite{MR3484112}, especially 3.9 and 4.2 for the details.

Fix a maximal split torus $S$ and denote the centralizer of $S$ by $Z$.
Let $Z^0$ be the unique parahoric subgroup of $Z$ and $Z(1)$ its pro-$p$ radical.
Then the group $W(1)$ (resp.\ $W$) is defined by $W(1) = N_G(Z)/Z(1)$ ($W = N_G(Z)/Z^0$) where $N_G(Z)$ is the normalizer of $Z$ in $G$.
We also have $Z_\kappa = Z^0/Z(1)$.
Let $G'$ be the group generated by the unipotent radical of parabolic subgroups \cite[II.1]{arXiv:1412.0737_accepted} and $W_\aff$ the image of $G'\cap N_G(Z)$ in $W$.
Then this is a Coxeter group.
Fix a set of simple reflections $S_\aff$.
The group $W$ has the natural length function and let $\Omega$ be the set of length zero elements in $W$.
Then we get the data $(W_\aff,S_\aff,\Omega,W,W(1),Z_\kappa)$.

Consider the apartment attached to $S$ and an alcove surrounded by the hyperplanes fixed by $S_\aff$.
Let $I(1)$ be the pro-$p$-Iwahori subgroup attached to this alcove.
Then with $q_s = \#(I(1)\widetilde{s}I(1)/I(1))$ for $s\in S_\aff$ with a lift $\widetilde{s}\in N_G(Z)$ and suitable $c_s$, the algebra $\mathcal{H}$ is isomorphic to the Hecke algebra attached to $(G,I(1))$ \cite[Proposition~4.4]{MR3484112}.

In this paper, \emph{the data $(W_\aff,S_\aff,\Omega,W,W(1),Z_\kappa)$ and the parameters $(q,c)$ come from $G$ in this way.}
Let $W_\aff(1)$ be the image of $G'\cap N_G(Z)$ in $W(1)$ and put $\mathcal{H}_\aff = \bigoplus_{w\in W_\aff(1)}CT_w$.
This is a subalgebra of $\mathcal{H}$.

\begin{prop}[{\cite[Proposition~4.4]{MR3484112}}]\label{prop:expansion of c}
Let $s\in \Refs{W(1)}$.
Then we have $c_{s} = \sum_{t\in Z_\kappa}c_{s}(t)T_t$ for some $c_{s}(t)\in \Z$ such that $\sum_{t\in Z_\kappa}c_{s}(t) = q_s - 1$.
\end{prop}

\subsection{Assumptions on $C$}
In this paper, $C$ is assumed to be a commutative ring.
Sometimes we add the following assumptions.
\begin{itemize}
\item $p = 0$ in $C$.
\item $C$ is an algebraically closed field of characteristic $p$
\end{itemize}
We declare this assumption at the top of the subsection or the statement of a theorem/proposition/lemma etc.
Otherwise we do not assume anything on $C$.

\subsection{The algebra $\mathcal{H}[q_s]$ and $\mathcal{H}[q_s^{\pm 1}]$}\label{subsec:The algebra H[q_s] and H[q_s^pm]}
For each $s\in S_\aff$, let $\mathbf{q}_s$ be an indeterminate such that if $wsw^{-1}\in S_\aff$ for $w\in W$, we have $\mathbf{q}_{wsw^{-1}} = \mathbf{q}_s$.
Let $C[\mathbf{q}_s]$ be a polynomial ring with these indeterminate.
Then with the parameter $s\mapsto \mathbf{q}_s$ and the other data coming from $G$, we have the algebra.
This algebra is denoted by $\mathcal{H}[\mathbf{q}_s]$ and we put $\mathcal{H}[\mathbf{q}_s^{\pm 1}] = \mathcal{H}[\mathbf{q}_s]\otimes_{C[\mathbf{q}_s]}C[\mathbf{q}_s^{\pm 1}]$.
Under $\mathbf{q}_s\mapsto \#(I(1)\widetilde{s}I(1)/I(1))\in C$ where $\widetilde{s}\in N_G(Z)$ is a lift of $s$, we have $\mathcal{H}[\mathbf{q}_s]\otimes_{C[\mathbf{q}_s]}C \simeq \mathcal{H}$.
As an abbreviation, we denote $\mathbf{q}_s$ by just $q_s$.
Consequently we denote by $\mathcal{H}[q_s]$ (resp.\ $\mathcal{H}[q_s^{\pm 1}]$).

Since $q_s$ is invertible in $\mathcal{H}[q_s^{\pm 1}]$, we can do some calculations in $\mathcal{H}[q_s^{\pm 1}]$ with $q_s^{-1}$.
If the result can be stated in $\mathcal{H}[q_s]$, then this is an equality in $\mathcal{H}[q_s]$ since $\mathcal{H}[q_s]$ is a subalgebra of $\mathcal{H}[q_s^{\pm 1}]$ and by specializing, we can get some equality in $\mathcal{H}$.
See \cite[4.5]{MR3484112} for more details.

\subsection{The root system and the Weyl groups}
Let $W_0 = N_G(Z)/Z$ be the finite Weyl group.
Then this is a quotient of $W$.
Recall that we have the alcove defining $I(1)$.
Fix a special point $\boldsymbol{x}_0$ from the border of this alcove.
Then $W_0\simeq \Stab_W\boldsymbol{x}_0$ and the inclusion $\Stab_W\boldsymbol{x}_0\hookrightarrow W$ is a splitting of the canonical projection $W\to W_0$.
Throughout this paper, we fix this special point and regard $W_0$ as a subgroup of $W$.
Set $S_0 = S_\aff\cap W_0\subset W$.
This is a set of simple reflections in $W_0$.
For each $w\in W_0$, we fix a representative $n_w\in W(1)$ such that $n_{w_1w_2} = n_{w_1}n_{w_2}$ if $\ell(w_1w_2) = \ell(w_1) + \ell(w_2)$.

The group $W_0$ is the Weyl group of the root system $\Sigma$ attached to $(G,S)$.
Our fixed alcove and special point give a positive system of $\Sigma$, denoted by $\Sigma^+$.
The set of simple roots is denoted by $\Delta$.
As usual, for $\alpha\in\Delta$, let $s_\alpha\in S_0$ be a simple reflection for $\alpha$.

The kernel of $W(1)\to W_0$ (resp.\ $W\to W_0$) is denoted by $\Lambda(1)$ (resp.\ $\Lambda$).
Then $Z_\kappa\subset \Lambda(1)$ and we have $\Lambda = \Lambda(1)/Z_\kappa$.
The group $\Lambda$ (resp.\ $\Lambda(1)$) is isomorphic to $Z/Z^0$ (resp.\ $Z/Z(1)$).
Any element in $W(1)$ can be uniquely written as $n_w \lambda$ where $w\in W_0$ and $\lambda\in \Lambda(1)$.
We have $W = W_0\ltimes \Lambda$.

\subsection{The map $\nu$}
The group $W$ acts on the apartment attached to $S$ and the action of $\Lambda$ is by the translation.
Since the group of translations of the apartment is $X_*(S)\otimes_{\Z}\R$, we have a group homomorphism $\nu\colon \Lambda\to X_*(S)\otimes_{\Z}\R$.
The compositions $\Lambda(1)\to \Lambda\to X_*(S)\otimes_{\Z}\R$ and $Z\to \Lambda\to X_*(S)\otimes_{\Z}\R$ are also denoted by $\nu$.
The homomorphism $\nu\colon Z\to X_*(S)\otimes_{\Z}\R\simeq \Hom_\Z(X^*(S),\R)$ is characterized by the following: For $t\in S$ and $\chi\in X^*(S)$, we have $\nu(t)(\chi) = -\val(\chi(t))$ where $\val$ is the normalized valuation of $F$.
The kernel of $\nu\colon Z\to X_*(S)\otimes_{\Z}\R$ is equal to the maximal compact subgroup $\widetilde{Z}$ of $Z$.
In particular, $\Ker(\Lambda(1)\xrightarrow{\nu} X_*(S)\otimes_{\Z}\R) = \widetilde{Z}/Z(1)$ is a finite group.

We call $\lambda\in \Lambda(1)$ dominant (resp.\ anti-dominant) if $\nu(\lambda)$ is dominant (resp.\ anti-dominant).

Since the group $W_\aff$ is a Coxeter system, it has the Bruhat order denoted by $\le$.
For $w_1,w_2\in W_\aff$, we write $w_1 < w_2$ if there exists $u\in \Omega$ such that $w_1u,w_2u\in W_\aff$ and $w_1u < w_2u$.
Moreover, for $w_1,w_2\in W(1)$, we write $w_1 < w_2$ if $w_1\in W_{\aff}(1)w_2$ and $\overline{w}_1 < \overline{w}_2$ where $\overline{w}_1,\overline{w}_2$ are the image of $w_1,w_2$ in $W$, respectively.
We write $w_1\le w_2$ if $w_1 < w_2$ or $w_1 = w_2$.

\subsection{Other basis}
For $w\in W(1)$, take $s_1,\dotsm,s_l\in S_\aff(1)$ and $u\in W(1)$ such that $l = \ell(w)$, $\ell(u) = 0$ and $w = s_1\dotsm s_l u$.
Set $T_w^* = (T_{s_1} - c_{s_1})\dotsm (T_{s_l} - c_{s_l})T_u$.
Then this does not depend on the choice and we have $T_w^* \in T_w + \sum_{v <  w}C T_v$.
In particular, $\{T_w^*\}_{w\in W(1)}$ is a basis of $\mathcal{H}$.
We also have the following other definition of $T_w^*$.
In $\mathcal{H}[q_s^{\pm 1}]$, $T_w$ is invertible and $q_wT_{w^{-1}}^{-1}\in \mathcal{H}[q_s]$.
The element $T_w^*$ is the image of $q_wT_{w^{-1}}^{-1}$ in $\mathcal{H}$.

For simplicity, we always assume that our commutative ring $C$ contains a square root of $q_s$ which is denoted by $q_s^{1/2}$ for $s\in S_\aff$.
For $w = s_1\dotsm s_lu$ where $\ell(w) = l$ and $\ell(u) = 0$, $q_w^{1/2} = q_{s_1}^{1/2}\dotsm q_{s_l}^{1/2}$ is a square root of $q_w$.
For a spherical orientation $o$, there is a basis $\{E_o(w)\}_{w\in W(1)}$ of $\mathcal{H}$ introduced in \cite[5]{MR3484112}.
We have
\[
E_o(w)\in T_w + \sum_{v < w}CT_v.
\]
This satisfies the following product formula~\cite[Theorem~5.25]{MR3484112}.
\begin{equation}\label{eq:product formula}
E_o(w_1)E_{o\cdot w_1}(w_2) = q_{w_1w_2}^{-1/2}q_{w_1}^{1/2}q_{w_2}^{1/2}E_o(w_1w_2).
\end{equation}
\begin{rem}
Since we do not assume that $q_s$ is invertible in $C$, $q_{w_1w_2}^{-1/2}q_{w_1}^{1/2}q_{w_2}^{1/2}$ does not make sense in a usual way.
The meaning of this term is the following.
Consider the indeterminate $\mathbf{q}_s$ as in subsection~\ref{subsec:The algebra H[q_s] and H[q_s^pm]}.
We also add $\mathbf{q}_s^{\pm 1/2}$.
Then we have $\mathbf{q}_{w_1w_2}^{-1/2}\mathbf{q}_{w_1}^{1/2}\mathbf{q}_{w_2}^{1/2}$.
Then this belongs to $C[\mathbf{q}_s]$~\cite[Lemma~4.19]{MR3484112}.
Hence by specializing $\mathbf{q}_s\mapsto q_s$, we get $q_{w_1w_2}^{-1/2}q_{w_1}^{1/2}q_{w_2}^{1/2}\in C$.
Such calculations will be used several times in this paper implicitly.
Note that since $\mathbf{q}_{w_1w_2}^{-1/2}\mathbf{q}_{w_1}^{1/2}\mathbf{q}_{w_2}^{1/2}$ is in $C[\mathbf{q}_s]$, rather than $C[\mathbf{q}_s^{1/2}]$, the factor $q_{w_1w_2}^{-1/2}q_{w_1}^{1/2}q_{w_2}^{1/2}$ does not depend on a choice of the square roots.
\end{rem}
Applying \eqref{eq:product formula} to $w_1 = w^{-1}$ and $w_2 = w$, if $q_w$ is invertible then $E_{o\cdot w}(w)$ is invertible and we have
\begin{equation}\label{eq:inverse of E_o}
E_o(w^{-1}) = q_wE_{o\cdot w}(w)^{-1}.
\end{equation}

There is a one-to-one correspondence between spherical orientations and closed Weyl chambers.
Let $o$ be a spherical orientation and $\mathcal{D}$ the corresponding closed Weyl chamber.
Then for $\lambda\in \Lambda(1)$, we have \cite[Example~5.30]{MR3484112}
\begin{equation}\label{eq:E_o for lambda}
E_o(\lambda) =
\begin{cases}
T_\lambda & (\nu(\lambda)\in \mathcal{D}),\\
T_\lambda^* & (-\nu(\lambda)\in \mathcal{D}).
\end{cases}
\end{equation}
Let $o_+$ (resp.\ $o_-$) be the dominant (resp.\ anti-dominant) orientation.
The orientation $o_+$ (resp.\ $o_-$) corresponds to the dominant (resp.\ anti-dominant) chamber.
Then for $s\in S_0$ and $v\in W_0$ \cite[Example~5.33]{MR3484112},
\begin{equation}\label{eq:E_o for W_0}
E_{o_+\cdot v}(n_s) =
\begin{cases}
T_{n_s} & (vs < v),\\
T_{n_s}^* & (vs > v),
\end{cases}
\quad
E_{o_-\cdot v}(n_s) =
\begin{cases}
T_{n_s} & (vs > v),\\
T_{n_s}^* & (vs < v).
\end{cases}
\end{equation}

By the definition of spherical orientations, $o\cdot \lambda = o$ for any spherical orientation $o$ and $\lambda\in \Lambda(1)$.
Hence by \eqref{eq:product formula}, the subspace $\bigoplus_{\lambda\in \Lambda(1)}C E_o(\lambda)$ is a subalgebra of $\mathcal{H}$.
We denote this subalgebra by $\mathcal{A}_o$.

Finally we introduce one more basis defined by
\[
E_-(n_w \lambda) = q_{n_w \lambda}^{1/2}q_{n_w}^{-1/2}q_\lambda^{-1/2}T_{n_w}^* E_{o_-}(\lambda)
\]
for $w\in W_0$ and $\lambda\in\Lambda(1)$.
By \cite[Lemma~4.2]{arXiv:1406.1003_accepted}, $\{E_-(w)\mid w\in W(1)\}$ is a $C$-basis of $\mathcal{H}$.

\subsection{Levi subalgebra}
Since we have a positive system $\Sigma^+$, we have a minimal parabolic subgroup $B$ with a Levi part $Z$.
In this paper, \emph{parabolic subgroups are always standard, namely containing $B$}.
Note that such parabolic subgroups correspond to subsets of $\Delta$.

Let $P$ be a parabolic subgroup.
Attached to the Levi part of $P$ containing $Z$, we have a data $(W_{\aff,P},S_{\aff,P},\Omega_P,W_P,W_P(1),Z_\kappa)$ and parameters $(q_P,c_P)$.
Hence we have the algebra $\mathcal{H}$.
The parameter $c_P$ is given by the restriction of $c$, hence we denote it just by $c$.
The parameter $q_P$ is defined as in \cite[4.1]{arXiv:1406.1003_accepted}.

For the objects attached to this data, we add a suffix $P$.
We have the set of simple roots $\Delta_P$, the root system $\Sigma_P$ and its positive system $\Sigma_P^+$, the finite Weyl group $W_{0,P}$, the set of simple reflections $S_{0,P}\subset W_{0,P}$, the length function $\ell_P$ and the base $\{T^P_w\}_{w\in W_P(1)}$, $\{T^{P*}_w\}_{w\in W_P(1)}$ and $\{E^P_o(w)\}_{w\in W_P(1)}$ of $\mathcal{H}_P$.
Note that we have no $\Lambda_P$, $\Lambda_P(1)$ and $Z_{\kappa,P}$ since they are equal to $\Lambda$, $\Lambda(1)$ and $Z_\kappa$.
We have the following lemma.
This follows from \cite[3.8, 4.2]{MR3484112}.
See also \cite[Remark~4.6]{inverse-Satake}.

\begin{lem}\label{lem:on c, about aff,Levi}
Let $s\in S_\aff$, $\widetilde{s}\in S_\aff(1)$ its lift, $\alpha\in\Sigma$ such that the image of $s$ in $W_0$ is $s_\alpha$ and $P$ a parabolic subgroup such that $\alpha\in\Sigma_P$.
\begin{enumerate}
\item We can take $\widetilde{s}$ from $W_{\aff,P}(1)$.
\item If $\widetilde{s}\in W_{\aff,P}(1)$, then $c_{\widetilde{s}}\in C[Z_\kappa\cap W_{\aff,P}(1)]$.
\end{enumerate}
\end{lem}
In particular, for $s = s_\alpha\in S_0$ where $\alpha\in\Delta$, we can take (and do) $n_s\in W_{\aff,P_\alpha}(1)$ where $P_\alpha$ is a parabolic subgroup such that $\Delta_{P_\alpha} = \{\alpha\}$.

An element $w = n_v\lambda\in W_P(1)$ where $v\in W_{0,P}$ and $\lambda\in\Lambda(1)$ is called $P$-positive (resp.\ $P$-negative) if $\langle \alpha,\nu(\lambda)\rangle\le 0$ (resp.\ $\langle \alpha,\nu(\lambda)\rangle\ge 0$) for any $\alpha\in\Sigma^+\setminus\Sigma_P^+$.
Let $W_P^+(1)$ (resp.\ $W_P^-(1)$) be the set of $P$-positive (resp.\ $P$-negative) elements and put $\mathcal{H}_P^{\pm} = \bigoplus_{w\in W_P^{\pm}(1)}CT_w^P$.
These are subalgebras of $\mathcal{H}_P$ \cite[Lemma~4.1]{arXiv:1406.1003_accepted}.
In general, for a group $\Gamma$, we denote its center by $Z(\Gamma)$.
\begin{lem}
There exists $\lambda_P^+$ (resp.\ $\lambda_P^-$) in the center of $W_P(1)$ such that $\langle \alpha,\nu(\lambda_P^+)\rangle < 0$ (resp.\ $\langle \alpha,\nu(\lambda_P^-)\rangle > 0$) for all $\alpha\in\Sigma^+\setminus\Sigma_P^+$.
\end{lem}
\begin{proof}
We prove the existence of $\lambda_P^+$.
Take $\lambda_0\in Z(\Lambda(1))$ such that the stabilizers of $\nu(\lambda_0)$ in $W_0$ is $W_{0,P}$ and $\langle \alpha,\nu(\lambda_0)\rangle < 0$ for any $\alpha\in\Sigma^+\setminus\Sigma^+_P$.
Fix $w\in W_{0,P}$.
Then for each $k\in\Z_{\ge 0}$, $\nu(n_w\lambda_0^kn_w^{-1}\lambda_0^{-k}) = 0$.
Namely $n_w\lambda_0^kn_w^{-1}\lambda_0^{-k}\in \Ker\nu$.
Since $\Ker\nu$ is finite, there exists $k_1 > k_2$ such that $n_w\lambda_0^{k_1}n_w^{-1}\lambda_0^{-{k_1}} = n_w\lambda_0^{k_2}n_w^{-1}\lambda_0^{-{k_2}}$.
Hence $n_w\lambda_0^{k_1 - k_2}n_w^{-1}\lambda_0^{-(k_1 - k_2)} = 1$.
Therefore, for each $w\in W_{0,P}$, there exists $k_w\in\Z_{>0}$ such that $n_w\lambda_0^{k_w}n_w^{-1}\lambda_0^{-k_w} = 1$.
Hence $n_w\cdot \lambda_0^{k_w} = \lambda_0^{k_w}$.
Set $k = \prod_{w\in W_{0,P}}k_w$ and $\lambda = \lambda_0^k$.
Then for $w\in W_{0,P}$, we have $n_w\cdot \lambda = \lambda$.
Since we took $\lambda_0$ from the center of $\Lambda(1)$, $\lambda$ commutes with the elements in $\Lambda(1)$.
The group $W_P(1)$ is generated by $\{n_w\mid w\in W_{0,P}\}$ and $\Lambda(1)$.
Hence $\lambda\in Z(W_P(1))$.
For $\alpha\in\Sigma^+\setminus\Sigma^+_P$, we have $\langle \alpha,\nu(\lambda)\rangle = k\langle \alpha,\nu(\lambda_0)\rangle < 0$.
Therefore $\lambda_P^+ = \lambda$ satisfies the condition of the lemma.
The element $\lambda_P^- = (\lambda_P^+)^{-1}$ satisfies the condition of the lemma.
\end{proof}

\begin{prop}[{\cite[Theorem~1.4]{MR3437789}}]\label{prop:localization as Levi subalgebra}
Let $\lambda_P^+$ (resp.\ $\lambda_P^-$) be in the center of $W_P(1)$ such that $\langle \alpha,\nu(\lambda_P^+)\rangle < 0$ (resp.\ $\langle \alpha,\nu(\lambda_P^-)\rangle > 0$) for all $\alpha\in\Sigma^+\setminus\Sigma_P^+$.
Then $T^P_{\lambda_P^+} = T^{P*}_{\lambda_P^+} = E^P_{o_{-,P}}(\lambda_P^+)$ (resp.\ $T^P_{\lambda_P^-} = T^{P*}_{\lambda_P^-} = E^P_{o_{-,P}}(\lambda_P^-)$) is in the center of $\mathcal{H}_P$ and we have $\mathcal{H}_P = \mathcal{H}_P^+E^P_{o_{-,P}}(\lambda_P^+)^{-1}$ (resp.\ $\mathcal{H}_P = \mathcal{H}_P^-E^P_{o_{-,P}}(\lambda_P^-)^{-1}$).
\end{prop}

We define $j_P^{\pm}\colon \mathcal{H}_P^\pm\to \mathcal{H}$ and $j_P^{\pm *}\colon \mathcal{H}_P^{\pm}\to \mathcal{H}$ by $j_P^{\pm}(T_w^P) = T_w$ and $j_P^{\pm *}(T_w^{P*}) = T_w^*$ for $w\in W_P^\pm(1)$.
Then these are algebra homomorphisms.
For $j_P^+$ and $j_P^{-*}$, it is \cite[Lemma~4.6]{arXiv:1406.1003_accepted} and the same argument can apply for $j_P^{+*}$ and $j_P^-$.

Some special cases of the following lemma is proved in \cite{MR3366919,arXiv:1406.1003_accepted}.
\begin{lem}\label{lem:image of E by j}
Let $w\in W_P(1)$.
\begin{enumerate}
\item If the image of $w$ in $W_P$ is in $W_{0,P}$, then $w$ is both $P$-positive and $P$-negative.
Moreover we have $j_P^{\pm}(T^{P*}_w) = T_w^*$ and $j_P^{\pm *}(T^P_w) = T_w$.
\item For $x\in W_{0,P}$, we have
\begin{align*}
j_P^+(E^P_{o_{-,P}\cdot x}(w)) & = E_{o_-\cdot x}(w) \quad (w\in W_P^+(1)),\\
j_P^{+*}(E^P_{o_{+,P}\cdot x}(w)) & = E_{o_+\cdot x}(w) \quad (w\in W_P^+(1)),\\
j_P^-(E^P_{o_{+,P}\cdot x}(w)) & = E_{o_+\cdot x}(w) \quad (w\in W_P^-(1)),\\
j_P^{-*}(E^P_{o_{-,P}\cdot x}(w)) & = E_{o_-\cdot x}(w) \quad (w\in W_P^-(1)).
\end{align*}
\end{enumerate}
\end{lem}
\begin{proof}
(1)
Let $v$ be the image of $w$ in $W_P$.
Then it is in $W_{0,P}$ by the assumption.
Set $t = n_v^{-1}w$.
Then $t\in \Ker(W_P(1)\to W_P) = Z_\kappa$.
In particular $t\in \Lambda(1)$ and we have $w = n_vt$.
Since $\nu(t) = 0$, $w$ is $P$-positive and $P$-negative.

For the proof of the second statement, first assume that $w = t\in Z_\kappa$.
Since $\ell(t) = 0$, we have $T_t^P = T_t^{P*}$ and $T_t = T_t^*$.
Therefore we get the statement.
In particular, for $s\in S_\aff(1)$, we have $j_P^{\pm}(c_s) = j_P^{\pm*}(c_s) = c_s$.

To prove the second statement, by induction on $\ell_P(w)$, we may assume that the image of $w$ in $W_{0,P}$ is a simple reflection.
Set $s = w$.
Then $j_P^{\pm}(T_s^{P*}) = j_P^{\pm}(T_s^P - c_s) = T_s - c_s$.
Since the image of $s$ in $W_{P}$ is a finite simple reflection, the image of $s$ in $W$ is also a finite simple reflection.
Therefore $T_s^* = T_s - c_s$.
Hence $j_P^{\pm}(T_s^{P*}) = T_s^*$.
Similarly we have $j_P^{\pm *}(T_s^P) = j_P^{\pm *}(T_s^{P*} + c_s) = T_s^* + c_s = T_s$.

(2)
It is sufficient to prove the lemma in $\mathcal{H}[q_s^{\pm 1}]$.
First we prove the following: Let $w_1,w_2,w_3$ be $P$-negative or $P$-positive elements such that $w_3 = w_1w_2$.
Then if the lemma holds for $w = w_1,w_2$ (resp.\ $w = w_2,w_3$) then it also holds for $w = w_3$ (resp.\ $w = w_1$).
We prove this claim only for $j_P^+$.
For the other cases, the same proof apply.

For $j_P^+$, we assume that $w_1,w_2,w_3$ are $P$-positive.
By \eqref{eq:product formula}, we have
\[
E^P_{o_{-,P}\cdot x}(w_1)E^P_{o_{-,P}\cdot n_xw_1}(w_2) = q_{P,w_1}^{1/2}q_{P,w_2}^{1/2}q_{P,w_3}^{-1/2}E^P_{o_{-,P}\cdot x}(w_3).
\]
Since $j_P^+$ is an algebra homomorphism, we have
\[
j_P^+(E^P_{o_{-,P}\cdot x}(w_1))j_P^+(E^P_{o_{-,P}\cdot n_xw_1}(w_2)) = q_{P,w_1}^{1/2}q_{P,w_2}^{1/2}q_{P,w_3}^{-1/2}j_P^+(E^P_{o_{-,P}\cdot x}(w_3)).
\]
By \cite[Lemma~4.5]{arXiv:1406.1003_accepted}, we have $q_{P,w_1}^{1/2}q_{P,w_2}^{1/2}q_{P,w_3}^{-1/2} = q_{w_1}^{1/2}q_{w_2}^{1/2}q_{w_3}^{-1/2}$.
Hence
\[
j_P^+(E^P_{o_{-,P}\cdot x}(w_1))j_P^+(E^P_{o_{-,P}\cdot n_xw_1}(w_2)) = q_{w_1}^{1/2}q_{w_2}^{1/2}q_{w_3}^{-1/2}j_P^+(E^P_{o_{-,P}\cdot x}(w_3)).
\]
If the lemma holds for $w = w_1$ and $w_2$, then 
\[
E_{o_{-}\cdot x}(w_1)E_{o_{-}\cdot n_xw_1}(w_2) = q_{w_1}^{1/2}q_{w_2}^{1/2}q_{w_3}^{-1/2}j_P^+(E^P_{o_{-,P}\cdot x}(w_3)).
\]
By \eqref{eq:product formula}, we get $j_P^+(E^P_{o_{-,P}\cdot x}(w_3)) = E_{o_{-}\cdot x}(w_3)$.
If the lemma holds for $w = w_2$ and $w_3$, then
\[
j_P^+(E^P_{o_{-,P}\cdot x}(w_1))E_{o_{-}\cdot n_xw_1}(w_2) = q_{w_1}^{1/2}q_{w_2}^{1/2}q_{w_3}^{-1/2}E_{o_{-}\cdot x}(w_3).
\]
By \eqref{eq:inverse of E_o}, $E_{o_{-,P}\cdot n_xw_1}(w_2)$ is invertible in $\mathcal{H}[q_s^{\pm 1}]$.
Hence by \eqref{eq:product formula}, we get $j_P^+(E^P_{o_{-,P}\cdot x}(w_1)) = E_{o_-\cdot x}(w_1)$.
The same proof can be applicable for the other cases.

We assume that $w = n_v$ for $v\in W_{0,P}$.
By the above argument, we may assume that $v\in S_{0,P}$.
We write $s$ for $v$.
By \eqref{eq:E_o for W_0}, we have
\[
E^P_{o_{+,P}\cdot x}(n_s) =
\begin{cases}
T^P_{n_s} & (xs < x),\\
T^{P*}_{n_s} & (xs > x),
\end{cases}
\quad
E^P_{o_{-,P}\cdot x}(n_s) =
\begin{cases}
T^P_{n_s} & (xs > x),\\
T^{P*}_{n_s} & (xs < x).
\end{cases}
\]
Hence by (1), for $j = j_P^{+*}$ or $j_P^-$ and $j' = j_P^+$ or $j_P^{-*}$, we have 
\[
j(E^P_{o_{+,P}\cdot x}(n_s)) =
\begin{cases}
T_{n_s} & (xs < x),\\
T^{*}_{n_s} & (xs > x),
\end{cases}
\quad
j'(E^P_{o_{-,P}\cdot x}(n_s)) =
\begin{cases}
T_{n_s} & (xs > x),\\
T^{*}_{n_s} & (xs < x).
\end{cases}
\]
On the other hand, we have
\[
E_{o_+\cdot x}(n_s) =
\begin{cases}
T_{n_s} & (xs < x),\\
T^{*}_{n_s} & (xs > x),
\end{cases}
\quad
E_{o_-\cdot x}(n_s) =
\begin{cases}
T_{n_s} & (xs > x),\\
T^{*}_{n_s} & (xs < x).
\end{cases}
\]
We get the lemma in this case.

Next we assume that $w = \lambda\in \Lambda(1)$ and it is in a chamber corresponding to the spherical orientation.
We deal with 4 cases separately.
Note that since $x(\Sigma^+\setminus\Sigma_P^+) = \Sigma^+\setminus\Sigma_P^+$, $n_x\cdot \lambda$ is $P$-positive (resp.\ $P$-negative) if and only if $\lambda$ is $P$-positive (reps.\ $P$-negative).
\begin{itemize}
\item[$j_P^+$:] Assume that $n_x\cdot \lambda$ is anti-dominant.
Then $n_x\cdot \lambda$ is $P$-positive.
Hence $\lambda$ is also $P$-positive.
We have $E^P_{o_{-,P}\cdot x}(\lambda) = T^P_\lambda$ and $E_{o_-\cdot x}(\lambda) = T_\lambda$ by \eqref{eq:E_o for lambda}.
Hence $j_P^+(E^P_{o_{-,P}\cdot x}(\lambda)) = E_{o_-\cdot x}(\lambda)$ in this case.
\item[$j_P^{+*}$:] Assume that $n_x\cdot \lambda$ is anti-dominant.
Then $n_x\cdot \lambda$ is $P$-positive.
Hence $\lambda$ is also $P$-positive.
We have $E^P_{o_{+,P}\cdot x}(\lambda) = T^{P*}_\lambda$ and $E_{o_+\cdot x}(\lambda) = T^*_\lambda$ by \eqref{eq:E_o for lambda}.
Hence $j_P^{+*}(E^P_{o_{+,P}\cdot x}(\lambda)) = E_{o_+\cdot x}(\lambda)$ in this case.
\item[$j_P^-$:] Assume that $n_x\cdot \lambda$ is dominant.
Then $n_x\cdot \lambda$ is $P$-negative.
Hence $\lambda$ is also $P$-negative.
We have $E^P_{o_{+,P}\cdot x}(\lambda) = T^P_\lambda$ and $E_{o_+\cdot x}(\lambda) = T_\lambda$ by \eqref{eq:E_o for lambda}.
Hence $j_P^-(E^P_{o_{+,P}\cdot x}(\lambda)) = E_{o_+\cdot x}(\lambda)$ in this case.
\item[$j_P^{-*}$:] Assume that $n_x\cdot \lambda$ is dominant.
Then $n_x\cdot \lambda$ is $P$-negative.
Hence $\lambda$ is also $P$-negative.
We have $E^P_{o_{-,P}\cdot x}(\lambda) = T^{P*}_\lambda$ and $E_{o_-\cdot x}(\lambda) = T^*_\lambda$ by \eqref{eq:E_o for lambda}.
Hence $j_P^{-*}(E^P_{o_{-,P}\cdot x}(\lambda)) = E_{o_-\cdot x}(\lambda)$ in this case.
\end{itemize}

We prove the general case for $j_P^+$.
The same argument implies the other cases.
Let $w\in W_P^+(1)$ and take $v\in W_0$, $\lambda_1,\lambda_2\in \Lambda(1)\cap W_P^+(1)$ such that $w = n_v\lambda_1\lambda_2^{-1}$, $n_x\cdot \lambda_1$ and $n_x\cdot \lambda_2$ are anti-dominant.
Then the lemma holds for $n_v,\lambda_1$ and $\lambda_2$.
Hence the argument in the beginning of the proof of (2), we get the lemma for $w$.
\end{proof}

\begin{cor}\label{cor:image by j, positie and negative case}
If $w\in W_P(1)$ is both $P$-positive and $P$-negative (in particular, if $w = n_v$ for some $v\in W_{0,P}$), then we have $j_P^{\pm}(T_w^{P*}) = T_w^*$ and $j_P^{\pm *}(T_w^{P}) = T_w$.
\end{cor}
\begin{proof}
We only prove $j_P^+(T_w^{P*}) = T_w^*$.
The same argument apply to other cases.

Take $c_v\in C$ such that $T_w^{P*} = \sum_v c_vE^P_{o_{-,P}}(v)$.
If $c_v\ne 0$ then $v\le w$ in $W_P(1)$.
Hence $v$ is also $P$-positive and $P$-negative by \cite[Lemma 4.1]{arXiv:1406.1003_accepted}.
We have
\begin{align*}
j_P^+(T_w^{P*}) & = \sum_v c_vj_P^+(E^P_{o_{-,P}}(v))\\
& = \sum_v c_vE_{o_-}(v)\\
& = \sum_v c_vj_P^{-*}(E^P_{o_{-,P}}(v))\\
& = j_P^{-*}(T_w^{P*}) = T_w^{*}.
\end{align*}
We get the corollary.
\end{proof}

We also use the following relative setting.
Let $Q$ be a parabolic subgroup containing $P$ and let $W_P^{Q+}(1)$ (resp.\ $W_P^{Q-}(1)$) be the set of $n_w\lambda$ where $\langle \alpha,\nu(\lambda)\rangle \le 0$ (resp.\ $\langle \alpha,\nu(\lambda)\rangle \ge 0$) for any $\alpha\in\Sigma_Q^+\setminus\Sigma_P^+$ and $w\in W_{0,P}$.
Put $\mathcal{H}_P^{Q\pm } = \bigoplus_{w\in W_P^{Q\pm}(1)}C T_w^P\subset \mathcal{H}_P$.
Then we have homomorphisms $j_P^{Q\pm },j_P^{Q\pm *}\colon \mathcal{H}_P^{Q\pm}\to \mathcal{H}_Q$ defined by a similar way.
\subsection{Parabolic induction}\label{subsec:Preliminaries, Parabolic induction}
Let $P$ be a parabolic subgroup and $\sigma$ an $\mathcal{H}_P$-module. (This is a right module as in subsection~\ref{subsec:Prop-p-Iwahori Hecke algebra}.)
Then we define an $\mathcal{H}$-module $I_P(\sigma)$ by
\[
I_P(\sigma) = \Hom_{(\mathcal{H}_P^-,j_P^{-*})}(\mathcal{H},\sigma).
\]
We call $I_P$ the parabolic induction.
For $P\subset P_1$, we write
\[
I_P^{P_1}(\sigma) = \Hom_{(\mathcal{H}_P^{P_1-},j_P^{P_1-*})}(\mathcal{H}_{P_1},\sigma).
\]
\begin{rem}
Since we have two algebras $\mathcal{H}_P^{\pm}$ and four homomorphisms $j_P^{\pm},j_P^{\pm *}$, we can define four ``inductions''.
The other inductions are studied in Section~\ref{sec:Inductions}.
\end{rem}

We recall some properties from \cite{arXiv:1406.1003_accepted} and \cite{MR3437789}.
Let $P$ be a parabolic subgroup.
Set $W_0^P = \{w\in W_0\mid w(\Delta_P) \subset\Sigma^+\}$.
Then the multiplication map $W_0^P\times W_{0,P}\to W_0$ is bijective and for $w_1\in W_0^P$ and $w_2\in W_{0,P}$, we have $\ell(w_1w_2) = \ell(w_1) + \ell(w_2)$.
We also put ${}^PW_0 = \{w\in W_0\mid w^{-1}(\Delta_P) \subset\Sigma^+\}$.
Then the multiplication map $W_{0,P}\times {}^PW_0 \to W_0$ is bijective and for $w_1\in W_{0,P}$ and $w_2\in {}^PW_0$, we have $\ell(w_1w_2) = \ell(w_1) + \ell(w_2)$.

\begin{prop}\label{prop:decomposition of I_P}
Let $P$ be a parabolic subgroup and $\sigma$ an $\mathcal{H}_P$-module.
\begin{enumerate}
\item The map $I_P(\sigma)\ni\varphi\mapsto (\varphi(T_{n_w}))_{w\in W_0^P}\in \bigoplus_{w\in W_0^P}\sigma$ is bijective.
\item The map $I_P(\sigma)\ni\varphi\mapsto (\varphi(T_{n_w}^*))_{w\in W_0^P}\in \bigoplus_{w\in W_0^P}\sigma$ is bijective.
\end{enumerate}
\end{prop}
\begin{proof}
The first part is \cite[Lemma~3.9]{MR3437789}.
The second part follows from the first part and $T_{n_w}^* \in T_{n_w} + \sum_{v < w}T_{n_v}C[Z_\kappa]$ with a usual triangular argument.
\end{proof}

Let $P\subset Q$ be parabolic subgroups.
We have an $\mathcal{H}_Q$-module $I_P^Q(\sigma)$ for an $\mathcal{H}_P$-module $\sigma$.
The parabolic inductions are transitive.
\begin{prop}[{\cite[Proposition~4.10]{MR3437789}}]\label{prop:transitivity of inductions}
The map $\varphi\mapsto (X\mapsto \varphi(X)(1))$ gives an isomorphism $I_Q(I_P^Q(\sigma))\simeq I_P(\sigma)$.
\end{prop}

\subsection{Length}
We prove some lemmas on the length.
These lemmas follow from the formula on the length~\cite{MR3484112}.

We recall the formula.
Let $\Sigma'\subset \Hom_\R(X_*(S)\otimes_\Z\R,\R)$ be the subset such that the action of $\Refs{W_\aff}$ on $X_*(S)\otimes_{\Z}\R$ is the set of reflections with respect to the hyperplanes $\{\{v\in X_*(S)\otimes_{\Z}\R\mid \alpha(v) + k = 0\}\mid \alpha\in\Sigma',k\in\Z\}$.
Then for any $\alpha\in \Sigma'$, there exists only one element in $\R_{>0}\alpha\cap \Sigma_\red$ and the map $\alpha$ sends to this unique element gives a bijection $\Sigma'\simeq \Sigma_\red$ where $\Sigma_\red$ is the set of reduced roots in $\Sigma$.
Set $\Sigma'^+ = \{\alpha\in\Sigma'\mid \Sigma^+\cap \R_{>0}\alpha\ne\emptyset\}$.
Then $\Sigma'$ is a root system with the Weyl group $W_0$ and $\Sigma'^+$ is a positive system of $\Sigma'$.
For $v\in W_0$ and $\lambda\in \Lambda(1)$, we have \cite[Corollary~5.10]{MR3484112}
\begin{align*}
\ell(\lambda n_v) & = \sum_{\alpha\in\Sigma'^+\cap v(\Sigma'^+)}\lvert\langle\alpha,\nu(\lambda)\rangle\rvert + \sum_{\alpha\in\Sigma'^+\cap v(\Sigma'^-)}\lvert\langle\alpha,\nu(\lambda)\rangle - 1\rvert,\\
\ell(n_v\lambda) & = \sum_{\alpha\in\Sigma'^+\cap v^{-1}(\Sigma'^+)}\lvert\langle\alpha,\nu(\lambda)\rangle\rvert + \sum_{\alpha\in\Sigma'^+\cap v^{-1}(\Sigma'^-)}\lvert\langle\alpha,\nu(\lambda)\rangle + 1\rvert.
\end{align*}

The map $W\to \Aut(X_*(S)\otimes_\Z\R)$ is injective on $W_\aff$ and the image of $W_\aff$ is the Weyl group of the affine root system $\{\alpha + k\mid \alpha\in\Sigma',k\in \Z\}$.
Hence we have $W_\aff\simeq W_0\ltimes \Z\coroot{\Sigma'}$.
For $\lambda\in \Lambda(1)\cap W_\aff(1)$, $\nu(\lambda)$ is in the image of $\Lambda\cap W_\aff$.
It is $\Z\coroot{\Sigma'}$.
Hence $\nu(\lambda)\in \Z\coroot{\Sigma'}\subset\R\coroot{\Sigma}$.

\begin{lem}\label{lem:length, on lambda}
Let $\lambda_1,\lambda_2\in \Lambda(1)$.
We have $\ell(\lambda_1\lambda_2) = \ell(\lambda_1) + \ell(\lambda_2)$ if and only if the vectors $\nu(\lambda_1),\nu(\lambda_2)$ are in the same closed Weyl chamber, namely $\langle \alpha,\nu(\lambda_1)\rangle\langle\alpha,\nu(\lambda_2)\rangle\ge 0$ for any $\alpha\in\Sigma^+$.
\end{lem}
\begin{proof}
Obvious from the above formula.
\end{proof}

\begin{lem}\label{lem:length zero}
Let $\lambda\in \Lambda(1)$.
Then we have $\ell(\lambda) = 0$ if and only if $\langle \alpha,\nu(\lambda)\rangle = 0$ for any $\alpha\in\Sigma$.
\end{lem}
\begin{proof}
Obvious from the above formula.
\end{proof}

\begin{lem}
If $w\in W(1)$ is in the center of $W(1)$, then $w\in \Lambda(1)$ and $\ell(w) = 0$.
\end{lem}
\begin{proof}
If $w$ is in the center, then $w\in \Lambda(1)$ by \cite[Lemma~1.1]{MR3271250}.
For $\alpha\in \Sigma$, we have $n_{s_\alpha}w n_{s_\alpha}^{-1} = w$.
Applying $\nu$, we get $s_\alpha(\nu(w)) = \nu(w)$.
Hence $\langle \alpha,\nu(w)\rangle = 0$.
\end{proof}

\begin{lem}\label{lem:length of lambda_aff is even}
If $\lambda\in \Lambda(1)\cap W_{\aff}(1)$, then $\ell(\lambda)$ is even.
\end{lem}
\begin{proof}
We have $\ell(\lambda) = \sum_{\alpha\in\Sigma'^+}\lvert\langle\alpha,\nu(\lambda)\rangle\rvert \equiv \sum_{\alpha\in\Sigma'^+}\langle \alpha,\nu(\lambda)\rangle = 2\langle\rho,\nu(\lambda)\rangle\pmod{2}$ where $\rho = (1/2)\sum_{\alpha\in\Sigma'^+}\alpha$.
For each $\alpha\in\Sigma'$ which is simple, $\langle \rho,\coroot{\alpha}\rangle = 1$.
Hence for any $v\in \Z\coroot{\Sigma'}$, we have $\langle \rho,v\rangle\in\Z$.
Hence $2\langle\rho,\nu(\lambda)\rangle$ is even.
\end{proof}

\begin{lem}\label{lem:conjugate and length}
For $w\in W(1)$ and $\lambda\in \Lambda(1)$, we have $\ell(w\cdot \lambda) = \ell(\lambda)$.
\end{lem}
\begin{proof}
Let $x\in W_0$ be the image of $w$.
Then we have $\nu(w\cdot \lambda) = x(\nu(\lambda))$.
Hence $\ell(w\cdot \lambda) = \sum_{\alpha\in\Sigma'^+}\lvert\langle \alpha,x(\nu(\lambda))\rangle\rvert = \sum_{\alpha\in x^{-1}(\Sigma'^+)}\lvert\langle \alpha,\nu(\lambda)\rangle\rvert$.
We have $x^{-1}(\Sigma'^+) = (\Sigma'^+\cap x^{-1}(\Sigma'^+))\sqcup (\Sigma'^-\cap x^{-1}(\Sigma'^+)) = (\Sigma'^+\cap x^{-1}(\Sigma'^+))\sqcup (-(\Sigma'^+\cap x^{-1}(\Sigma'^-)))$.
Since 
\[
\sum_{\alpha\in -(\Sigma'^+\cap x^{-1}(\Sigma'^-))}\lvert\langle \alpha,\nu(\lambda)\rangle\vert = \sum_{\alpha\in (\Sigma'^+\cap x^{-1}(\Sigma'^-))}\lvert\langle \alpha,\nu(\lambda)\rangle\vert,
\]
we have
\begin{align*}
\ell(w\cdot \lambda) & = \sum_{\alpha\in(\Sigma'^+\cap x^{-1}(\Sigma'^+))}\lvert\langle \alpha,\nu(\lambda)\rangle\vert + \sum_{\alpha\in-(\Sigma'^+\cap x^{-1}(\Sigma'^-))}\lvert\langle \alpha,\nu(\lambda)\rangle\vert\\
& = \sum_{\alpha\in(\Sigma'^+\cap x^{-1}(\Sigma'^+))}\lvert\langle \alpha,\nu(\lambda)\rangle\vert + \sum_{\alpha\in(\Sigma'^+\cap x^{-1}(\Sigma'^-))}\lvert\langle \alpha,\nu(\lambda)\rangle\vert\\
&= \sum_{\alpha\in\Sigma'^+}\lvert\langle\alpha,\nu(\lambda)\rangle\rvert = \ell(\lambda).
\end{align*}
We finish the proof.
\end{proof}

\begin{lem}\label{lem:ell(v lambda) as ell(v)+ell(lambda)}
Let $v\in W_0$ and $\lambda\in \Lambda(1)$.
Then we have
\begin{align*}
\ell(\lambda n_v) & = \ell(\lambda) + \ell(v) - 2\#\{\alpha\in\Sigma_\red^+\mid v^{-1}(\alpha) < 0,\ \langle \alpha,\nu(\lambda)\rangle > 0\}\\
& = \ell(\lambda) - \ell(v) + 2\#\{\alpha\in\Sigma_\red^+\mid v^{-1}(\alpha) < 0,\ \langle \alpha,\nu(\lambda)\rangle \le 0\}
\end{align*}
and
\begin{align*}
\ell(n_v \lambda) & = \ell(\lambda) + \ell(v) - 2\#\{\alpha\in\Sigma_\red^+\mid v(\alpha) < 0,\ \langle \alpha,\nu(\lambda)\rangle < 0\}\\
& = \ell(\lambda) - \ell(v) + 2\#\{\alpha\in\Sigma_\red^+\mid v(\alpha) < 0,\ \langle \alpha,\nu(\lambda)\rangle \ge 0\}.
\end{align*}
\end{lem}
\begin{proof}
The formula for $\ell(n_v\lambda)$ follows from the formula for $\ell(\lambda n_v)$ and $\ell(n_v \lambda) = \ell((n_v\lambda)^{-1}) = \ell(\lambda^{-1}n_v^{-1}) = \ell(\lambda^{-1}n_{v^{-1}})$, here, at the last point, we use the fact that $n_{v^{-1}}n_v\in \Ker(W(1)\to W)$ has the length zero.

Since
\[
\lvert\langle \alpha,\nu(\lambda)\rangle - 1\rvert - 1
=
\begin{cases}
\lvert\langle\alpha,\nu(\lambda)\rangle\rvert - 2& (\langle\alpha,\nu(\lambda)\rangle > 0),\\
\lvert\langle\alpha,\nu(\lambda)\rangle\rvert & (\langle\alpha,\nu(\lambda)\rangle \le 0),
\end{cases}
\]
we have
\begin{multline*}
\ell(\lambda n_v) - \#(\Sigma'^+\cap v(\Sigma'^-))\\
= \sum_{\alpha\in\Sigma'^+}\lvert\langle\alpha,\nu(\lambda)\rangle\rvert - 2\#\{\alpha\in\Sigma'^+\cap v(\Sigma'^-)\mid \langle \alpha,\nu(\lambda)\rangle > 0\}.
\end{multline*}
We have $\#(\Sigma'^+\cap v(\Sigma'^-)) = \ell(v)$ and $\sum_{\alpha\in\Sigma'^+}\lvert\langle\alpha,\nu(\lambda)\rangle\rvert = \ell(\lambda)$.
Hence $\ell(\lambda n_v) = \ell(\lambda) + \ell(v) -  2\#\{\alpha\in\Sigma'^+\cap v(\Sigma'^-)\mid \langle \alpha,\nu(\lambda)\rangle > 0\}$.
For any $\alpha\in\Sigma'$ there exists a unique $r_\alpha > 0$ such that $r_\alpha \alpha\in\Sigma_\red$.
We have $\alpha\in\Sigma'^+\cap v(\Sigma'^-)$ if and only if $r_\alpha \alpha\in \Sigma_\red^+\cap v(\Sigma_\red^-)$ and $\langle \alpha,\nu(\lambda)\rangle > 0$ if and only if $\langle r_\alpha\alpha,\nu(\lambda)\rangle > 0$.
Since $\alpha\mapsto r_\alpha \alpha$ gives a bijection $\Sigma'\simeq \Sigma_\red$, we get 
\[
\#\{\alpha\in\Sigma'^+\cap v(\Sigma'^-)\mid \langle \alpha,\nu(\lambda)\rangle > 0\}
=
\#\{\alpha\in\Sigma_\red^+\cap v(\Sigma_\red^-)\mid \langle \alpha,\nu(\lambda)\rangle > 0\}.
\]
We get the first formula.
Since
\begin{align*}
\ell(v) & = \#\{\alpha\in\Sigma_\red^+\mid v^{-1}(\alpha) < 0\}\\
& = \#\{\alpha\in\Sigma_\red^+\mid v^{-1}(\alpha) < 0,\ \langle \alpha,\nu(\lambda)\rangle \le 0\}\\
& \quad + \#\{\alpha\in\Sigma_\red^+\mid v^{-1}(\alpha) < 0,\ \langle \alpha,\nu(\lambda)\rangle > 0\},
\end{align*}
we get the second formula.
\end{proof}
\begin{lem}\label{lem:condition for length is additive}
Let $v\in W_0$ and $\lambda\in\Lambda(1)$.
We have:
\begin{itemize}
\item $\ell(\lambda n_v) = \ell(\lambda) + \ell(v)$ if and only if for any $\alpha\in\Sigma^+$ such that $v^{-1}(\alpha) < 0$, $\langle \alpha,\nu(\lambda)\rangle \le 0$.
\item $\ell(\lambda n_v) = \ell(\lambda) - \ell(v)$ if and only if for any $\alpha\in\Sigma^+$ such that $v^{-1}(\alpha) < 0$, $\langle \alpha,\nu(\lambda)\rangle > 0$.
\item $\ell(n_v \lambda) = \ell(\lambda) + \ell(v)$ if and only if for any $\alpha\in\Sigma^+$ such that $v(\alpha) < 0$, $\langle \alpha,\nu(\lambda)\rangle \ge 0$.
\item $\ell(n_v \lambda) = \ell(\lambda) - \ell(v)$ if and only if for any $\alpha\in\Sigma^+$ such that $v(\alpha) < 0$, $\langle \alpha,\nu(\lambda)\rangle < 0$.
\end{itemize}
In particular, for $\alpha\in\Delta$ and $\lambda\in\Lambda(1)$, we have
\begin{itemize}
\item $\ell(\lambda n_{s_\alpha}) = \ell(\lambda) + 1$ if and only if $\langle \alpha,\nu(\lambda)\rangle \le 0$.
\item $\ell(n_{s_\alpha}\lambda) = \ell(\lambda) + 1$ if and only if $\langle \alpha,\nu(\lambda)\rangle \ge 0$.
\end{itemize}
\end{lem}
\begin{proof}
Obvious from the previous lemma.
\end{proof}

\begin{lem}\label{lem:length, positive/negative elements}
Let $P$ be a parabolic subgroup, $v\in W_0$, $w\in W_P(1)$ and $\lambda_0 \in Z(W_P(1))$.
\begin{enumerate}
\item If $v\in W_0^P$, $\lambda_0$ is dominant and $w$ is $P$-negative, then we have $\ell(n_v \lambda_0w) = \ell(n_v\lambda_0) + \ell(w) = \ell(v) + \ell(\lambda_0) + \ell(w)$.\label{enum:negative:lem:length, positive/negative elements}
\item If $v\in {}^PW_0$, $\lambda_0$ is anti-dominant and $w$ is $P$-positive, then we have $\ell(w \lambda_0 n_v) = \ell(w) + \ell(\lambda_0 n_v) = \ell(w) + \ell(\lambda_0) + \ell(v)$.\label{enum:positive:lem:length, positive/negative elements}
\end{enumerate}
\end{lem}
\begin{proof}
(\ref{enum:positive:lem:length, positive/negative elements}) follows from (\ref{enum:negative:lem:length, positive/negative elements}) by taking the inverse.

We prove (1).
Take $w_1\in W_{0,P}$ and $\lambda\in \Lambda(1)$ such that $w = n_{w_1}\lambda$.
We remark that $\nu(\lambda_0)$ and $\nu(\lambda)$ is in the same closed Weyl chamber.
In fact, if $\alpha\in\Sigma^+\setminus\Sigma^+_P$, then $\langle \alpha,\nu(\lambda_0)\rangle\ge 0$ and $\langle \alpha,\nu(\lambda)\rangle\ge 0$ by the assumption.
Hence $\langle \alpha,\nu(\lambda_0)\rangle\langle \alpha,\nu(\lambda)\rangle\ge 0$.
If $\alpha\in\Sigma_P^+$, then $\langle \alpha,\nu(\lambda_0)\rangle = 0$ since $\lambda_0\in Z(W_P(1))$.
Hence $\langle \alpha,\nu(\lambda)\rangle\langle \alpha,\nu(\lambda_0)\rangle\ge 0$ for any $\alpha\in\Sigma^+_P$.
Therefore $\nu(\lambda)$ and $\nu(\lambda_0)$ are in the same closed Weyl chamber.
In particular, $\ell(\lambda_0\lambda) = \ell(\lambda_0) + \ell(\lambda)$ by Lemma~\ref{lem:length, on lambda}.

Since $\lambda_0$ is in the center of $W_P(1)$, we have
\begin{align*}
\ell(n_v\lambda_0 w) & = \ell(n_vn_{w_1}\lambda_0\lambda)\\
& = \ell(\lambda_0 \lambda) + \ell(vw_1) - 2\#\{\alpha\in\Sigma^+_\red\mid vw_1(\alpha) < 0,\ \langle \alpha,\nu(\lambda_0\lambda)\rangle < 0\}
\end{align*}
by Lemma~\ref{lem:ell(v lambda) as ell(v)+ell(lambda)}.

Let $\alpha\in\Sigma^+_\red$ such that $vw_1(\alpha) < 0$.
Since $v\in W^P_0$, $\ell(vw_1) = \ell(v) + \ell(w_1)$.
Hence we have $w_1(\alpha) < 0$ or $v(\beta) < 0$ where $\beta = w_1(\alpha) > 0$.
Assume that $v(\beta) < 0$ where $\beta = w_1(\alpha) > 0$.
Since $v\in W_0^P$, $\beta\in\Sigma^+\setminus\Sigma_P^+$.
Therefore $\alpha = w_1^{-1}(\beta)\in\Sigma^+\setminus\Sigma^+_P$.
Hence $\langle \alpha,\nu(\lambda)\rangle\ge 0$ since $\lambda$ is $P$-negative.
From the assumption, $\lambda_0$ is dominant.
Therefore $\langle\alpha,\nu(\lambda_0)\rangle\ge 0$.
Hence we get $\langle\alpha,\nu(\lambda_0\lambda)\rangle\ge 0$.
Therefore we have
\[
\ell(n_v\lambda_0 w) = \ell(\lambda_0 \lambda) + \ell(vw_1) - 2\#\{\alpha\in\Sigma^+_\red\mid w_1(\alpha) < 0,\ \langle \alpha,\nu(\lambda_0\lambda)\rangle < 0\}.
\]
Since $w_1\in W_{0,P}$, $w_1(\alpha) < 0$ implies $\alpha\in\Sigma_P^+$.
Hence $\langle\alpha,\nu(\lambda_0)\rangle = 0$.
We have
\[
\ell(n_v\lambda_0 w) = \ell(\lambda_0 \lambda) + \ell(vw_1) - 2\#\{\alpha\in\Sigma^+_\red\mid w_1(\alpha) < 0,\ \langle \alpha,\nu(\lambda)\rangle < 0\}.
\]
Recall that we have $\ell(vw_1) = \ell(v) + \ell(w_1)$ and $\ell(\lambda_0 \lambda) = \ell(\lambda_0) + \ell(\lambda)$.
Hence 
\begin{align*}
& \ell(n_v\lambda_0 w)\\
& = \ell(\lambda_0) + \ell(v) + \ell(\lambda) + \ell(w_1) - 2\#\{\alpha\in\Sigma^+_\red\mid w_1(\alpha) < 0,\ \langle \alpha,\nu(\lambda)\rangle < 0\}\\
& = \ell(\lambda_0) + \ell(v) + \ell(n_{w_1}\lambda) = \ell(\lambda_0) + \ell(v) + \ell(w)
\end{align*}
by Lemma~\ref{lem:ell(v lambda) as ell(v)+ell(lambda)}.
Put $w = 1$.
Then we have $\ell(n_v \lambda_0) = \ell(\lambda_0) + \ell(v)$.
\end{proof}

\begin{lem}\label{lem:length, positive/negative elements, strictly}
Let $w\in W_P(1)$, $v\in W_0$ and $\lambda_0\in Z(W_P(1))$.
\begin{enumerate}
\item If $w$ is $P$-positive, $\lambda_0 = \lambda_P^+$ as in Proposition~\ref{prop:localization as Levi subalgebra} and $v\in W_0^P$, then $\ell(n_v\lambda_0 w) = \ell(n_v\lambda_0) + \ell(w) = \ell(\lambda_0) - \ell(v) + \ell(w)$.\label{enum:positive:lem:length, positive/negative elements, strictly}
\item If $w$ is $P$-negative, $\lambda_0 = \lambda_P^-$ as in Proposition~\ref{prop:localization as Levi subalgebra} and $v\in {}^PW_0$, then $\ell(w\lambda_0 n_v) = \ell(w) + \ell(\lambda_0 n_v) = \ell(w) + \ell(\lambda_0) - \ell(v)$.\label{enum:negative:lem:length, positive/negative elements, strictly}
\end{enumerate}
\end{lem}
\begin{proof}
(\ref{enum:negative:lem:length, positive/negative elements, strictly}) follows from (\ref{enum:positive:lem:length, positive/negative elements, strictly}) by taking the inverse.

Take $w_1\in W_{0,P}$ and $\lambda\in \Lambda(1)$ such that $w = n_{w_1}\lambda$.
Then we have
\begin{align*}
& \ell(n_v\lambda_0 w)\\
& = \ell(n_vn_{w_1} \lambda_0 \lambda)\\
& = \ell(\lambda_0\lambda) - \ell(vw_1) + 2\#\{\alpha\in\Sigma_\red^+\mid (vw_1)(\alpha) < 0,\ \langle \alpha,\nu(\lambda_0\lambda)\rangle\ge 0\}
\end{align*}
by Lemma~\ref{lem:ell(v lambda) as ell(v)+ell(lambda)}.
Let $\alpha\in\Sigma^+$ such that $(vw_1)(\alpha) < 0$.
Then $w_1(\alpha) < 0$ or $v(\beta) < 0$ where $\beta = w_1(\alpha) > 0$.
If $v(\beta) < 0$, $\beta = w_1(\alpha) > 0$, then $\beta\in\Sigma^+\setminus\Sigma_P^+$ since $v\in W_0^P$.
Since $w_1\in W_{0,P}$, we have $\alpha = w_1^{-1}(\beta)\in \Sigma^+\setminus\Sigma_P^+$.
Hence $\langle \alpha,\nu(\lambda_0)\rangle < 0$ by the condition on $\lambda_P^+$, $\langle \alpha,\nu(\lambda)\rangle \le 0$ since $\lambda$ is $P$-positive.
Therefore we have $\langle \alpha,\nu(\lambda_0\lambda)\rangle < 0$.
We get
\[
\ell(n_v\lambda_0 w) = \ell(\lambda_0\lambda) - \ell(vw_1) + 2\#\{\alpha\in\Sigma_\red^+\mid w_1(\alpha) < 0,\ \langle \alpha,\nu(\lambda_0\lambda)\rangle\ge 0\}.
\]
If $w_1(\alpha) < 0$, then since $w_1\in W_{0,P}$, we have $\alpha\in\Sigma_P$.
Hence $\langle\alpha,\nu(\lambda_0)\rangle = 0$.
Therefore we have
\[
\ell(n_v\lambda_0 w) = \ell(\lambda_0\lambda) - \ell(vw_1) + 2\#\{\alpha\in\Sigma_\red^+\mid w_1(\alpha) < 0,\ \langle \alpha,\nu(\lambda)\rangle\ge 0\}.
\]
We have:
\begin{itemize}
\item $\ell(vw_1) = \ell(v) + \ell(w_1)$ since $v\in W_0^P$ and $w_1\in W_{0,P}$.
\item $\ell(\lambda_0 \lambda) = \ell(\lambda_0) + \ell(\lambda)$.
Indeed, $\nu(\lambda_0)$ and $\nu(\lambda)$ are in the same closed Weyl chamber.
If $\alpha\in\Sigma^+\setminus\Sigma_P^+$, then $\langle\alpha,\nu(\lambda_0)\rangle$ and $\langle\alpha,\nu(\lambda)\rangle$ are both not positive since $\lambda_0,\lambda$ are both $P$-positive.
If $\alpha\in\Sigma_P^+$, then $\langle\alpha,\nu(\lambda_0)\rangle = 0$, hence $\langle\alpha,\nu(\lambda_0)\rangle\langle\alpha,\nu(\lambda)\rangle = 0\ge 0$.
\end{itemize}
Therefore we get
\begin{align*}
& \ell(n_v \lambda_0 w)\\
& = \ell(\lambda_0) - \ell(v) + \ell(\lambda) - \ell(w_1) + 2\#\{\alpha\in\Sigma_\red^+\mid w_1(\alpha) < 0,\ \langle \alpha,\nu(\lambda)\rangle\ge 0\}\\
& = \ell(\lambda_0) - \ell(v) + \ell(\lambda n_{w_1}) = \ell(\lambda_0) - \ell(v) + \ell(w)
\end{align*}
by Lemma~\ref{lem:ell(v lambda) as ell(v)+ell(lambda)}.
Applying this to $w = 1$, we get $\ell(n_v \lambda_0) = \ell(\lambda_0) - \ell(v)$.
\end{proof}

\begin{lem}\label{lem:condition for length is additive, P-negative}
Let $w\in W_P(1)$ and $\lambda_P^-\in \Lambda(1)$ as in Proposition~\ref{prop:localization as Levi subalgebra}.
Then $\ell(w \lambda_P^-) = \ell(w) + \ell(\lambda_P^-)$ if and only if $w$ is $P$-negative.
\end{lem}
\begin{proof}
The ``if part'' follows from Lemma~\ref{lem:length, positive/negative elements, strictly} (2).
Assume that $\ell(w \lambda_P^-) = \ell(w) + \ell(\lambda_P^-)$.
Take $v\in W_{0,P}$ and $\mu\in \Lambda(1)$ such that $w = n_v\mu$.
Then we have
\[
\ell(w\lambda_P^-) = \sum_{\alpha\in\Sigma'^+\cap v^{-1}(\Sigma'^+)}\lvert\langle \alpha,\nu(\mu \lambda_P^-)\rangle\rvert + \sum_{\alpha\in\Sigma'^+\cap v^{-1}(\Sigma'^-)}\lvert\langle \alpha,\nu(\mu \lambda_P^-)\rangle + 1\rvert
\]
and
\begin{align*}
\ell(w) & = \sum_{\alpha\in\Sigma'^+\cap v^{-1}(\Sigma'^+)}\lvert\langle \alpha,\nu(\mu)\rangle\rvert + \sum_{\alpha\in\Sigma'^+\cap v^{-1}(\Sigma'^-)}\lvert\langle \alpha,\nu(\mu)\rangle + 1\rvert,\\
\ell(\lambda_P^-) & = \sum_{\alpha\in\Sigma'^+\cap v^{-1}(\Sigma'^+)}\lvert\langle \alpha,\nu(\lambda_P^-)\rangle\rvert + \sum_{\alpha\in\Sigma'^+\cap v^{-1}(\Sigma'^-)}\lvert\langle \alpha,\nu(\lambda_P^-)\rangle\rvert
\end{align*}
by the length formula.
By the triangle inequality and the assumption  $\ell(w \lambda_P^-) = \ell(w) + \ell(\lambda_P^-)$, we have
\[
\lvert\langle \alpha,\nu(\mu \lambda_P^-)\rangle + \varepsilon\rvert
=
\lvert\langle \alpha,\nu(\mu)\rangle + \varepsilon\rvert
+
\lvert\langle \alpha,\nu(\lambda_P^-)\rangle\rvert.
\]
where $\varepsilon = 1$ if $\alpha\in \Sigma'^+\cap v^{-1}(\Sigma'^-)$ and $\varepsilon = 0$ if $\alpha\in\Sigma'^+\cap v^{-1}(\Sigma'^+)$.
If $\alpha\in \Sigma^+\setminus \Sigma_P^+$, we have $v(\alpha) > 0$ since $v\in W_{0,P}$.
Hence $\varepsilon = 0$.
Therefore we get 
\[
\lvert\langle \alpha,\nu(\mu \lambda_P^-)\rangle\rvert
=
\lvert\langle \alpha,\nu(\mu)\rangle\rvert
+
\lvert\langle \alpha,\nu(\lambda_P^-)\rangle\rvert,
\]
So we get $\langle \alpha,\nu(\mu)\rangle\langle \alpha,\nu(\lambda_P^-)\rangle\ge 0$.
We have $\langle \alpha,\nu(\lambda_P^-)\rangle > 0$ by the condition on $\lambda_P^-$.
Hence $\langle \alpha,\nu(\mu)\rangle\ge 0$.
Therefore $w = n_v\mu$ is $P$-negative.
\end{proof}

\subsection{Twist by $n_{w_Gw_P}$}
For a parabolic subgroup $P$, let $w_P$ be the longest element in $W_{0,P}$.
In particular, $w_G$ is the longest element in $W_0$.
Let $P'$ be a parabolic subgroup corresponding to $-w_G(\Delta_P)$, in other words, $P' = n_{w_Gw_P}\opposite{P}n_{w_Gw_P}^{-1}$ where $\opposite{P}$ is the opposite parabolic subgroup of $P$ with respect to the Levi part of $P$ containing $Z$.
Set $n = n_{w_Gw_P}$.
Then the map $\opposite{P}\to P'$ defined by $p\mapsto npn^{-1}$ is an isomorphism which preserves the data used to define the pro-$p$-Iwahori Hecke algebras.
Hence $T^P_w\mapsto T^{P'}_{nwn^{-1}}$ gives an isomorphism $\mathcal{H}_P\to \mathcal{H}_{P'}$.
This sends $T_w^{P*}$ to $T_{nwn^{-1}}^{P'*}$ and $E^P_{o_{+,P}\cdot v}(w)$ to $E^{P'}_{o_{+,P'}\cdot nvn^{-1}}(nwn^{-1})$ where $v\in W_{0,P}$.

Let $\sigma$ be an $\mathcal{H}_P$-module.
Then we define an $\mathcal{H}_{P'}$-module $n_{w_Gw_P}\sigma$ via the pull-back of the above isomorphism.
Namely, we define $(n_{w_Gw_P}\sigma)(T^{P'}_w) = \sigma(T^P_{n_{w_Gw_P}^{-1}wn_{w_Gw_P}})$.

Using this twist, we have another description of $I_P$.
\begin{prop}[{\cite[Theorem~1.8]{MR3437789}}]\label{prop:tensor description of I_P}
Let $P$ be a parabolic subgroup and $\sigma$ an $\mathcal{H}_P$-module.
Set $P' = n_{w_Gw_P}\opposite{P}n_{w_Gw_P}^{-1}$.
Then the map $\varphi\mapsto \varphi(T_{n_w})$ gives an isomorphism 
\[
\{\varphi\in I_P(\sigma)\mid \text{$\varphi(T_{n_w}) = 0$ for any $w\in W^P_0\setminus\{w_Gw_P\}$}\}\simeq n_{w_Gw_P}\sigma|_{\mathcal{H}_{P'}^+}
\]
as $j_{P'}^+(\mathcal{H}_{P'}^+)$-modules.
Let $x\mapsto \varphi_x$ be the inverse of the above isomorphism.
Then the induced homomorphism 
\[
n_{w_Gw_P}\sigma\otimes_{(\mathcal{H}_{P'}^+,j_{P'}^+)}\mathcal{H}\to I_P(\sigma)
\]
given by $x\otimes X\mapsto \varphi_xX$ is an isomorphism.
\end{prop}
\begin{lem}\label{lem:description of isomo between two inductions}
\begin{enumerate}
\item The map $w\mapsto w_Gw_Pw^{-1}$ gives a bijection $W^P_0\simeq {}^{P'}W_0$ which reverse the Bruhat order.
\item The isomorphism $I_P(\sigma)\simeq n_{w_Gw_P}\sigma\otimes_{(\mathcal{H}_{P'}^+,j_{P'}^+)}\mathcal{H}$ is given by $I_P(\sigma)\ni \varphi\mapsto \sum_{w\in W^P_0}\varphi(T_{n_w})\otimes T_{n_{w_Gw_Pw^{-1}}}^*$.
\end{enumerate}\end{lem}
\begin{proof}
The first part is proved in the proof of \cite[Proposition~4.15]{arXiv:1406.1003_accepted}.
For (2), we prove the following.
Let $x\in \sigma$ and $\varphi\in I_P(\sigma)$ such that $\varphi(T_{n_{w_Gw_P}}) = x$ and $\varphi(T_{n_w}) = 0$ for any $w\in W_0^P\setminus\{w_Gw_P\}$.
Then for $v\in {}^PW_0$ and $w\in W_0^P$, we have
\[
(\varphi T_{n_v}^*)(T_{n_w}) = 
\begin{cases}
x & v = w_Gw_Pw^{-1},\\
0 & \text{otherwise}.
\end{cases}
\]
This means the following diagram is commutative.
\[
\begin{tikzcd}
I_P(\sigma)\arrow{d}{\wr}& n_{w_Gw_P}\sigma\otimes_{(\mathcal{H}_{P'}^+,j_{P'}^+)}\mathcal{H}\arrow{l}[swap]{\sim}\arrow[-]{d}{\wr}\\
\displaystyle\bigoplus_{w\in W^P_0}\sigma \arrow{r} & \displaystyle\bigoplus_{v\in {}^{P'}W_0}\sigma\otimes T_{n_v}^*\\[-1em]
\rotatebox{90}{$\in$} & \rotatebox{90}{$\in$}\\[-2.2em]
(x_w) \arrow[mapsto]{r} & (x_{v^{-1}w_Gw_P}\otimes T_{n_v}^*).
\end{tikzcd}
\]
The commutativity of this diagram implies the lemma.

Assume that $(\varphi T_{n_v}^*)(T_{n_w})\ne 0$.
We have
\[
T_{n_v}^*T_{n_w} = E_{v^{-1}\cdot o_-}(n_v)E_{o_-}(n_w) \in CE_{v^{-1}\cdot o_-}(n_vn_w)
\subset \sum_{a\le vw}T_{n_a}C[Z_\kappa].
\]
For $a\in W_0$ and $t\in Z_\kappa$, we have $\varphi(T_{n_a}T_t) = \varphi(T_{n_a})T^P_t$.
Hence $\varphi(T_{n_a}) \ne 0$ for some $a\le vw$.

Decompose $a = a_1a_2$ where $a_1\in W^P_0$ and $a_2\in W_{0,P}$.
Then we have $\varphi(T_{n_{a}}) = \varphi(T_{n_{a_1}})T^P_{{n_{a_2}}}$.
Since this is non-zero, we have $a_1 = w_Gw_P$.
Namely we have $a\in w_GW_{0,P}$.
Take $b\in W_{0,P}$ such that $a = w_Gb$.
By (1), we can take $v_1\in W_0^P$ such that $v = w_Gw_Pv_1^{-1}$.
Then $a\le vw$ implies $b\ge w_Pv_1^{-1}w$.
Since $b\in W_{0,P}$, we also have $w_Pv_1^{-1}w\in W_{0,P}$.
Hence $v_1^{-1}w\in W_{0,P}$.
Therefore we have $w\in v_1W_{0,P}$.
Since $v_1,w\in W_0^P$, we have $v_1 = w$.
Hence $v = w_Gw_Pw^{-1}$.

If $v = w_Gw_Pw^{-1}$, then $\ell(v) = \ell(w_G) - \ell(w_Pw^{-1}) = \ell(w_G) - \ell(w_P) - \ell(w)$ since $w\in W_0^P$.
Hence $\ell(v) + \ell(w) = \ell(w_G) - \ell(w_P) = \ell(w_Gw_P) = \ell(vw)$.
Therefore we have 
\begin{align*}
T_{n_v}^*T_{n_w} & = E_{o_-\cdot v^{-1}}(n_v)E_{o_-}(n_w)\\
&  = E_{o_-\cdot v^{-1}}(n_vn_w)\in T_{n_{w_Gw_P}} + \sum_{a < w_Gw_P}T_{n_a}C[Z_\kappa].
\end{align*}
By \cite[Lemma~4.13]{arXiv:1406.1003_accepted}, if $a < w_Gw_P$ then $a\notin w_Gw_PW_{0,P}$.
Hence we have $\varphi(T_{n_a}C[Z_\kappa]) = 0$.
Therefore we have $\varphi(T_{n_v}^*T_{n_w}) = \varphi(T_{n_{w_Gw_P}}) = x$.
\end{proof}

\subsection{The extension and the generalized Steinberg modiles}
Let $P$ be a parabolic subgroup and $\sigma$ an $\mathcal{H}_P$-module.
For $\alpha\in\Delta$, let $P_\alpha$ be a parabolic subgroup corresponding to $\{\alpha\}$.
Then we define $\Delta(\sigma)\subset\Delta$ by
\begin{align*}
&\Delta(\sigma)\\& = \{\alpha\in\Delta\mid \langle \Delta_P,\coroot{\alpha}\rangle = 0,\ \text{$\sigma(T^P_\lambda) = 1$ for any $\lambda\in W_{\aff,P_\alpha}(1)\cap \Lambda(1)$}\}\cup \Delta_P.
\end{align*}
Let $P(\sigma)$ be a parabolic subgroup corresponding to $\Delta(\sigma)$.
\begin{prop}[{\cite{AHHV2}}]
Let $\sigma$ be an $\mathcal{H}_P$-module and $Q$ a parabolic subgroup between $P$ and $P(\sigma)$.
Denote the parabolic subgroup corresponding to $\Delta_Q\setminus\Delta_P$ by $P_2$.
Then there exist a unique $\mathcal{H}_Q$-module $e_Q(\sigma)$ acting on the same space as $\sigma$ such that
\begin{itemize}
\item $e_Q(\sigma)(T_w^{Q*}) = \sigma(T_w^{P*})$ for any $w\in W_P(1)$.
\item $e_Q(\sigma)(T_w^{Q*}) = 1$ for any $w\in W_{P_2,\aff}(1)$.
\end{itemize}
Moreover, one of the following condition gives a characterization of $e_Q(\sigma)$.
\begin{enumerate}
\item For any $w\in W_P^{Q-}(1)$, $e_Q(\sigma)(T_w^{Q*}) = \sigma(T_w^{P*})$ (namely, $e_Q(\sigma) \simeq \sigma$ as $(\mathcal{H}_P^{Q-},j_P^{Q-*})$-modules) and for any $w\in W_{\aff,P_2}(1)$, $e_Q(\sigma)(T_w^{Q*}) = 1$.
\item For any $w\in W_P^{Q+}(1)$, $e_Q(\sigma)(T_w^{Q*}) = \sigma(T_w^{P*})$ and for any $w\in W_{\aff,P_2}(1)$, $e_Q(\sigma)(T_w^{Q*}) = 1$.
\end{enumerate}
\end{prop}

We call $e_Q(\sigma)$ the extension of $\sigma$ to $\mathcal{H}_Q$.
A typical example of the extension is the trivial representation $\trivrep = \trivrep_G$.
This is a one-dimensional $\mathcal{H}$-module defined by $\trivrep(T_w) = q_w$, or equivalently $\trivrep(T_w^*) = 1$.
We have $\Delta(\trivrep_P) = \{\alpha\in\Delta\mid \langle \Delta_P,\coroot{\alpha}\rangle = 0\}\cup \Delta_P$ and, if $Q$ is a parabolic subgroup between $P$ and $P(\trivrep_P)$, we have $e_Q(\trivrep_P) = \trivrep_Q$

\begin{rem}
Assume that $p = 0$ in $C$.
The condition ``for any $w\in W_{\aff,P_2}(1)$, $e_Q(\sigma)(T_w^{Q*}) = 1$'' is equivalent to the following two conditions.
\begin{itemize}
\item $e_Q(\sigma)(T_s^Q) = 0$ for any $s\in S_{\aff,P_2}(1)\cap W_{\aff,P_2}(1)$.
\item $e_Q(\sigma)(T_t^Q) = 1$ for any $t\in Z_\kappa\cap W_{\aff,P_2}(1)$.
\end{itemize}
Indeed, assume that $e_Q(\sigma)(T_w^{Q*}) = 1$ for any $w\in W_{\aff,P_2}(1)$.
Then for $t\in Z_\kappa\cap W_{\aff,P_2}(1)$, we have $T_t^{Q*} = T_t^Q$.
Hence $e_Q(\sigma)(T_t^Q) = e_Q(\sigma)(T_t^{Q*}) = 1$.
For $s\in S_{\aff,P_2}(1)\cap W_{\aff,P_2}(1)$, take $c_s(t)\in \Z$ as in Proposition~\ref{prop:expansion of c}.
By Lemma~\ref{lem:on c, about aff,Levi}, we have $e_Q(\sigma)(c_s) = \sum_{t\in Z_\kappa\cap W_{\aff,P_2}(1)}c_s(t)e_Q(\sigma)(T^Q_t) = \sum_{t\in Z_\kappa\cap W_{\aff,P_2}(1)}c_s(t) = q_{s,P_2} - 1 = -1$.
Hence $e_Q(\sigma)(T^Q_s) = e_Q(\sigma)(T_s^{Q*}) + e_Q(\sigma)(c_s) = 1 - 1 = 0$.

On the other hand, assume that the two conditions hold.
Then by the above argument, from the second condition, we have $e_Q(\sigma)(c_s) = -1$.
Hence $e_Q(\sigma)(T_s^{Q*}) = e_Q(\sigma)(T_s^Q) - e_Q(\sigma)(c_s) = 1$.
By taking a reduced expression of $w\in W_{\aff,P_2}(1)$, we get $e_Q(\sigma)(T_w^{Q*}) = 1$.
The conditions are appeared in \cite[4.4]{arXiv:1406.1003_accepted}.
\end{rem}

\begin{rem}\label{rem:value at Lambda_aff,P_2}
For each $\alpha\in\Delta$, let $P_\alpha$ be a parabolic subgroup corresponding to $\{\alpha\}$.
By \cite[Lemma~2.5]{arXiv:1406.1003_accepted}, $\Lambda(1)\cap W_{\aff,P_2}(1)$ is generated by $\bigcup_{\alpha\in\Delta(\sigma)\setminus\Delta_P}(\Lambda(1)\cap W_{\aff,P_\alpha}(1))$.
Hence for each $\lambda\in \Lambda(1)\cap W_{\aff,P_2}(1)$, we can write $\lambda = \mu_1\dotsm \mu_r$ where $\mu_i\in W_{\aff,P_\alpha}(1)\cap \Lambda(1)$ for some $\alpha\in\Delta(\sigma)\setminus\Delta_P$.
Since $\alpha\in\Delta(\sigma)\setminus\Delta_P$ is orthogonal to $\Delta_P$, $\ell_P(\mu_i) = 0$ for each $i$.
Therefore $T^P_\lambda = T^P_{\mu_1}\dotsm T^P_{\mu_r}$.
Since $\sigma(T_{\mu_i}^P) = 1$, we have $\sigma(T^P_\lambda) = 1$.
\end{rem}

Let $P(\sigma)\supset P_0\supset Q_1\supset Q\supset P$.
Then as in \cite[4.5]{arXiv:1406.1003_accepted}, we have $I^{P_0}_{Q_1}(e_{Q_1}(\sigma))\subset I^{P_0}_Q(e_Q(\sigma))$.
Define
\[
\St_Q^{P_0}(\sigma) = \Coker\left(\bigoplus_{Q_1\supsetneq Q}I_{Q_1}^{P_0}(e_{Q_1}(\sigma))\to I^{P_0}_Q(e_Q(\sigma))\right).
\]
When $P_0 = G$, we write $\St_Q(\sigma)$.

In the rest of this subsection, we assume that $P(\sigma) = G$.
Since $\Delta\setminus\Delta_P = \Delta(\sigma)\setminus \Delta_P$ is orthogonal to $\Delta_P$, we have $w_Gw_P\in W_{0,P_2}$ where $P_2$ corresponds to $\Delta\setminus\Delta_P$ and $n_{w_Gw_P}\opposite{P}n_{w_Gw_P}^{-1} = P$.
Hence $n_{w_Gw_P}\sigma$ is also an $\mathcal{H}_P$-module.
\begin{lem}\label{lem:twist by longest element stabilize sigma}
$n_{w_Gw_P}\sigma =\sigma$.
\end{lem}
\begin{proof}
Let $w\in W_P(1)$.
Put $n = n_{w_Gw_P}$.
Since $n \in W_{\aff,P_2}(1)$ and $W_{\aff,P_2}(1)$ is a normal subgroup of $W(1)$~\cite[Lemma~4.17]{arXiv:1406.1003_accepted}, we have $n^{-1}wnw^{-1} = n^{-1}(wnw^{-1})\in W_{\aff,P_2}(1)$.
The image of $n$ (resp.\ $w$) by $W(1)\to W\to W_0$ is in $W_{0,P_2}$ (resp.\ $W_{0,P}$) and by the assumption $P(\sigma) = G$, $W_{0,P_2}$ and $W_{0,P}$ commute with each other.
Hence the image of $n^{-1}wnw^{-1}$ in $W_0$ is trivial.
Therefore $n^{-1}wnw^{-1}\in \Lambda(1)\cap W_{\aff,P_2}(1)$.
In particular, the length as an element in $W_P(1)$ is zero by Lemma~\ref{lem:length zero} and $\sigma(T^P_{n^{-1}wnw^{-1}}) = 1$ by Remark~\ref{rem:value at Lambda_aff,P_2}.
Hence $n\sigma(T_w^P) = \sigma(T_{n^{-1}wn}^P) = \sigma(T_{n^{-1}wnw^{-1}}^PT_w^P) = \sigma(T_{n^{-1}wnw^{-1}}^P)\sigma(T_w^P) = \sigma(T_w^P)$.
\end{proof}

\begin{lem}\label{lem:extension and twist}
Let $Q$ be a parabolic subgroup containing $P$ and set $Q' = n_{w_Gw_Q}\opposite{Q}n_{w_Gw_Q}^{-1}$.
Then we have $n_{w_Gw_Q}e_Q(\sigma)\simeq e_{Q'}(\sigma)$.
\end{lem}
\begin{proof}
By the above lemma, we have
\[
\sigma|_{\mathcal{H}_P^-} = n_{w_Gw_P}\sigma|_{\mathcal{H}_P^-} = n_{w_Gw_Q}e_Q(n_{w_Qw_P}\sigma)|_{\mathcal{H}_P^-} = n_{w_Gw_Q}e_Q(\sigma)|_{\mathcal{H}_P^-}.
\]

Let $Q_2$ (resp.\ $Q'_2$) be the subgroup corresponding to $\Delta_Q\setminus\Delta_P$ (resp.\ $\Delta_{Q'}\setminus \Delta_P$).
We have $w_Gw_Q\in W_{0,P_2}$.
Hence $n_{w_Gw_Q}$ preserves $\Sigma_{P_2}$.
Moreover we have $w_Gw_Q(\Sigma_{Q_2}) = \Sigma_{Q'_2}$.
For $\alpha\in\Sigma_{Q_2}$, the root subgroup for $\alpha$ is sent to that of $w_Gw_Q(\alpha)$ by $n_{w_Gw_Q}$.
Denote the Levi part of $Q_2$ (resp.\ $Q_2'$) containing $Z$ by $M_{Q_2}$ (resp.\ $M_{Q_2'}$).
Then the above argument implies $n_{w_Gw_Q}M'_{Q_2}n_{w_Gw_Q}^{-1} = M'_{Q'_2}$.
Hence $n_{w_Gw_Q}W_{\aff,Q_2}(1)n_{w_Gw_Q}^{-1} = W_{\aff,Q'_2}(1)$.
Therefore, for $w\in W_{\aff,Q'_2}(1)$, we have 
\[
(n_{w_Gw_Q}e_Q(\sigma))(T_{w}^{Q'*}) = e_Q(\sigma)(T_{n_{w_Gw_Q}^{-1}wn_{w_Gw_Q}}^{Q*}) = 1
\]
from the definition of the extension.
We get the lemma by the characterization of the extension.
\end{proof}

\subsection{Supersingular modules}\label{subsec:supersingulars}
Assume that $p = 0$ in $C$.
Let $\mathcal{O}$ be a conjugacy class in $W(1)$ which is contained in $\Lambda(1)$.
For a spherical orientation $o$, set $z_\mathcal{O} = \sum_{\lambda\in \mathcal{O}}E_o(\lambda)$.
Then this does not depend on $o$ and gives an element of the center of $\mathcal{H}$ \cite[Theorem~5.1]{Vigneras-prop-III}.
The length of $\lambda\in \mathcal{O}$ does not depend on $\lambda$.
We denote it by $\ell(\mathcal{O})$.
\begin{defn}
Let $\pi$ be an $\mathcal{H}$-module.
We call $\pi$ supersingular if there exists $n\in \Z_{>0}$ such that $\pi z_\mathcal{O}^n = 0$ for any $\mathcal{O}$ such that $\ell(\mathcal{O}) > 0$.
\end{defn}

\subsection{Simple modules}
Assume that $C$ is an algebraically closed field of characteristic $p$.
We consider the following triple $(P,\sigma,Q)$.
\begin{itemize}
\item $P$ is a parabolic subgroup.
\item $\sigma$ is an simple supersingular $\mathcal{H}_P$-module.
\item $Q$ is a parabolic subgroup between $P$ and $P(\sigma)$.
\end{itemize}
Define
\[
I(P,\sigma,Q) = I_{P(\sigma)}(\St_Q^{P(\sigma)}(\sigma)).
\]
\begin{thm}[{\cite[Theorem~1.1]{arXiv:1406.1003_accepted}}]
The module $I(P,\sigma,Q)$ is simple and any simple module has this form.
Moreover, $(P,\sigma,Q)$ is unique up to isomorphism.
\end{thm}
The simple supersingular modules are classified in \cite{MR3263136,Vigneras-prop-III}.
We do not recall the classification since we do not need it in this paper.

\section{A ftilration on parabolic inductions}\label{sec:A ftilration on parabolic inductions}
\subsection{A filtration}\label{subsec:A filtration}
Let $P$ be a parabolic subgroup and $A$ a subset of $W_0^P$.
For an $\mathcal{H}_P$-module $\sigma$, put 
\[
I_P(\sigma)_A = \{\varphi\in I_P(\sigma)\mid \varphi(T_{n_w}) = 0\ (w\in W_0^P\setminus A) \}.
\]
We call $A\subset W_0^P$ open if $v\in A,w\ge v$ implies $w\in A$.
Assume that $A$ is open and fix a minimal element $w\in A$.
Set $A' = A\setminus \{w\}$.
Then $A'$ is also open.
By Proposition~\ref{prop:decomposition of I_P}, the map $I_P(\sigma)_A/I_P(\sigma)_{A'}\to \sigma$ given by $\varphi\mapsto \varphi(T_{n_w})$ is a bijection.
In this section, we give a description of the action of $E_{o_-}(\lambda)$ on $I_P(\sigma)_A/I_P(\sigma)_{A'}$.
We start with the following lemma.
\begin{lem}
Let $w\in W_P(1)$ and $\lambda_0 = \lambda_P^-\in \Lambda(1)$ as in Proposition~\ref{prop:localization as Levi subalgebra} such that $w\lambda_0$ is $P$-negative.
Then $q_w^{1/2}q_{\lambda_0}^{1/2}q_{w\lambda_0}^{-1/2}$ does not depend on a choice of $\lambda_0$.
\end{lem}
\begin{proof}
Let $\lambda_0'$ be another choice and put $\lambda_1 = \lambda_0\lambda_0'$.
Since $\nu(\lambda_0)$ and $\nu(\lambda_0')$ belong to the same closed Weyl chamber, we have $\ell(\lambda_0\lambda_0') = \ell(\lambda_0) + \ell(\lambda_0')$ by Lemma~\ref{lem:length, on lambda}.
Hence $q_{\lambda_1} = q_{\lambda_0}q_{\lambda_0'}$.
By Lemma~\ref{lem:length, positive/negative elements, strictly} (2), we have $\ell(w\lambda_1) = \ell(w\lambda_0) + \ell(\lambda_0')$.
Hence $q_{w\lambda_1} = q_{w\lambda_0}q_{\lambda_0'}$.
Therefore we get $q_w^{1/2}q_{\lambda_0}^{1/2}q_{w\lambda_0}^{-1/2} = q_w^{1/2}q_{\lambda_1}^{1/2}q_{w\lambda_1}^{-1/2}$.
Replacing $\lambda_0$ with $\lambda_0'$, we also have $q_w^{1/2}q_{\lambda_0'}^{1/2}q_{w\lambda_0'}^{-1/2} = q_w^{1/2}q_{\lambda_1}^{1/2}q_{w\lambda_1}^{-1/2}$.
We get the lemma.
\end{proof}
We denote $q_w^{1/2}q_{\lambda_0}^{1/2}q_{w\lambda_0}^{-1/2}$ by $q(P,w)$.
By Lemma~\ref{lem:condition for length is additive, P-negative}, we have $q(P,w) = 1$ if and only if $w$ is $P$-negative.

\begin{prop}\label{prop:Bruhat filtration and action of A}
The subspace $I_P(\sigma)_A$ is $\mathcal{A}_{o_-}$-stable and the action of $E_{o_-}(\lambda)$ on $I_P(\sigma)_A/I_P(\sigma)_{A'}\simeq \sigma$ is given by $q(P,n_w^{-1}\cdot \lambda)E_{o_{-,P}}^P(n_w^{-1}\cdot \lambda)$.
\end{prop}
We need the following lemma.
Recall that we have another basis $\{E_-(w)\mid w\in W(1)\}$ defined by $E_-(n_v \mu) = q_{n_v}^{-1/2}q_{\mu}^{-1/2}q_{n_v\mu}^{1/2}T_{n_v}^* E_{o_-}(\mu)$ for $v\in W_0$ and $\mu\in \Lambda(1)$.
From the definition, we have
\[
E_-(w)E_{o_-}(\lambda) = q_w^{1/2}q_{\lambda}^{1/2}q_{w\lambda}^{-1/2}E_-(w\lambda).
\]
\begin{lem}\label{lem:E_- at right of parabolic induction}
Let $X\in \mathcal{H}$, $\varphi \in I_P(\sigma)$ and $w\in W_P(1)$.
Then we have $\varphi(XE_-(w)) = q(P,w)\varphi(X)\sigma(E_-^P(w))$.
\end{lem}
\begin{proof}
Replacing $\varphi$ with $\varphi X$, we may assume $X = 1$.
If $w$ is $P$-negative, then this follows from $q(P,w) = 1$ and $j_P^{-*}(E^P_-(w)) = E_-(w)$~\cite[Lemma~4.6]{arXiv:1406.1003_accepted}.
In general, let $\lambda_P^-\in \Lambda(1)$ as in Proposition~\ref{prop:localization as Levi subalgebra} such that $w\lambda_P^-$ is $P$-negative.
Then we have
\[
E_-(w)E_{o_-}(\lambda_P^-) = q(P,w)E_-(w\lambda_P^-).
\]
Hence we have
\begin{align*}
\varphi(E_-(w)) & = \varphi(E_-(w)E_{o_-}(\lambda_P^-))\sigma(E^P_{o_{-,P}}(\lambda_P^-)^{-1})\\
& = q(P,w)\varphi(E_-(w\lambda_P^-))\sigma(E^P_{o_{-,P}}(\lambda_P^-)^{-1})\\
& = q(P,w)\varphi(1)\sigma(E^P_-(w\lambda_P^-)E^P_{o_{-,P}}(\lambda_P^-)^{-1})\\
& = q(P,w)\varphi(1)\sigma(E^P_-(w))
\end{align*}
We get the lemma.
\end{proof}

We also use:
\begin{lem}
Let $v\in W_0^P$ and $\varphi\in I_P(\sigma)$.
Assume that $\varphi(T_{n_v}^*) = 0$.
Then we have $\varphi(E_-(n_vw)) = 0$ for any $w\in W_P(1)$.
\end{lem}
\begin{proof}
Take $\lambda_P^-$ as in Proposition~\ref{prop:localization as Levi subalgebra} such that $w\lambda_P^-$ is $P$-negative.
Then we have 
\begin{align*}
\varphi(E_-(n_vw)) & = \varphi(E_-(n_vw)E_{o_-}(\lambda_P^-))\sigma(E^P_{o_{-,P}}(\lambda_P^-)^{-1})\\
& \in C[q_s]\varphi(E_-(n_vw\lambda_P^-))\sigma(E^P_{o_{-,P}}(\lambda_P^-)^{-1}).
\end{align*}
Hence it is sufficient to prove $\varphi(E_-(n_vw\lambda_P^-)) = 0$.
Namely we may assume $w$ is $P$-negative.

If $w$ is $P$-negative, by Lemma~\ref{lem:length, positive/negative elements}, we have $\ell(n_vw) = \ell(n_v) + \ell(w)$.
Hence by the definition of $E_-$, we have $E_-(n_vw) = T_{n_v}^*E_-(w)$.
Therefore we have $\varphi(E_-(n_vw)) = \varphi(T_{n_v}^*)\sigma(E_-^P(w)) = 0$.
\end{proof}

\begin{proof}[Proof of Proposition~\ref{prop:Bruhat filtration and action of A}]
Let $v\notin A'$.
By the Bernstein relations~\cite[Corollary~5.43]{MR3484112}, in $\mathcal{H}[q_s^{\pm 1}]$, we have
\[
E_{o_-}(\lambda)T_{n_v}\in T_{n_v}E_{o_-}(n_v^{-1}\cdot\lambda) + \sum_{v_1 < v,\mu \in \Lambda(1)}C[q_s^{\pm 1}]T_{n_{v_1}}E(\mu).
\]
Since $T_{n_{v_1}}\in \sum_{v_2 \le v_1}T_{n_{v_2}}^*C[Z_\kappa]$, we have
\begin{align*}
E_{o_-}(\lambda)T_{n_v} & \in T_{n_v}E_{o_-}(n_v^{-1}\cdot\lambda) + \sum_{v_2 < v,\mu \in \Lambda(1)}C[q_s^{\pm 1}]T_{n_{v_2}}^*E(\mu)\\
& = T_{n_v}E_{o_-}(n_v^{-1}\cdot\lambda) + \sum_{v_2<v,\mu\in\Lambda(1)}C[q_s^{\pm 1}]E_-(n_{v_2}\mu).
\end{align*}
For $v_2 < v$, take $v_3\in W^P_0$ and $v_3'\in W_{0,P}$ such that $v_2 = v_3 v_3'$.
Then $E_-(n_{v_2}\mu) = E_{-}(n_{v_3}n_{v_3'}\mu)$ and $n_{v_3'}\mu\in W_P(1)$.
We have $v_3\le v_2 < v$.
Hence
\begin{align*}
E_{o_-}(\lambda)T_{n_v}
& \in T_{n_v}E_{o_-}(n_v^{-1}\cdot\lambda) + \left(\sum_{v_3<v,v_3\in W^P_0,x\in W_P(1)}C[q_s^{\pm 1}]E_-(n_{v_3}x)\cap \mathcal{H}\right)\\
& = T_{n_v}E_{o_-}(n_v^{-1}\cdot\lambda) + \sum_{v_3<v,v_3\in W^P_0,x\in W_P(1)}C[q_s]E_-(n_{v_3}x).
\end{align*}

Let $\varphi \in I_P(\sigma)_A$ and we prove $\varphi(E_-(n_{v_3}x)) = 0$ for $v_3 < v,v_3\in W_0^P,x\in W_P(1)$ by applying the above lemma.
We check $\varphi(T_{n_{v_3}}^*) = 0$.
We have $T_{n_{v_3}}^* \in \sum_{v_4 \le v_3}T_{n_{v_4}}C[Z_\kappa]$.
Since $v_4\le v_3 < v$ and $v\notin A'$, we have $v_4\notin A$.
Hence $\varphi(T_{n_{v_4}}) = 0$.
Therefore we get $\varphi(T_{n_{v_3}}^*) = 0$.

Therefore we have $(\varphi E_{o_-}(\lambda))(T_{n_{v}}) = \varphi(T_{n_v}E_{o_-}(n_v^{-1}\cdot \lambda))$.
By Lemma~\ref{lem:E_- at right of parabolic induction}, we have $\varphi(T_{n_v}E_{o_-}(n_v^{-1}\cdot \lambda)) = q(P,n_v^{-1}\cdot \lambda)\varphi(T_{n_v})\sigma(E^P_{o_-}(n_v^{-1}\cdot \lambda))$.
This is zero if $v\ne w$.
Hence $\varphi E_{o_-}(\lambda) \in I_P(\sigma)_A$.
If $v = w$, we get $\varphi(E_{o_-}(\lambda)T_{n_{w}}) = q(P,n_w^{-1}\cdot \lambda)\varphi(T_{n_w})\sigma(E^P_{o_-}(n_w^{-1}\cdot \lambda))$.
This gives the lemma.
\end{proof}

Finally, we describe the filtration in terms of a tensor product.
Recall that we have an isomorphism (Proposition~\ref{prop:tensor description of I_P})
\[
I_P(\sigma)\simeq n_{w_Gw_P}\sigma\otimes_{(\mathcal{H}_{P'}^+,j_{P'}^+)}\mathcal{H}
\]
where $P' = n_{w_Gw_P}\opposite{P}n_{w_Gw_P}^{-1}$
Let $A\subset {}^{P'}W_0$ be a closed subset, namely a subset which satisfies that $v\in A,w\le v$ implies $w\in A$.
Set $A_0 = \{w^{-1}w_Gw_P\mid w\in A\}$.
Then $A_0\subset W_0^P$ is an open subset by Lemma~\ref{lem:description of isomo between two inductions}.
By Lemma~\ref{lem:description of isomo between two inductions}, $I_P(\sigma)_{A_0}$ corresponds to
\[
\sum_{v\in A}n_{w_Gw_P}\sigma\otimes T^*_{n_v}.
\]
Let $w\in A$ be a maximal element and put $A' = A\setminus\{w\}$.
\begin{lem}\label{lem:Bruhat filtration via tensor product}
The quotient
\[
\left(\sum_{v\in A}n_{w_Gw_P}\sigma\otimes T^*_{n_v}\right)/\left(\sum_{v\in A'}n_{w_Gw_P}\sigma\otimes T^*_{n_v}\right)
\]
is isomorphic to $\sigma$ as a vector space and the action of $E_{o_-}(\lambda)$ is given by
\[
q(P,n_{w^{-1}w_Gw_P}^{-1}\cdot \lambda)\sigma(E^P_{o_{-,P}}(n_{w^{-1}w_Gw_P}^{-1}\cdot \lambda)).
\]
\end{lem}
\begin{rem}
Since $ \ell(w) + \ell(w^{-1}w_Gw_P) = \ell(w_Gw_P)$ (see the last part of the proof of Lemma~\ref{lem:description of isomo between two inductions}), we have $n_{w^{-1}w_Gw_P} = n_w^{-1}n_{w_Gw_P}$.
Hence we have
\begin{align*}
\sigma(E^P_{o_{-,P}}(n_{w^{-1}w_Gw_P}^{-1}\cdot \lambda)) & = \sigma(E^P_{o_{-,P}}(n_{w_Gw_P}^{-1}n_w\cdot \lambda))\\
& = (n_{w_Gw_P}\sigma)(E^{P'}_{o_{-,{P'}}}(n_w\cdot \lambda)).
\end{align*}

\begin{rem}\label{rem:q(P,*)=0 iff P'-positive}
For any $\mu$, $q(P,n_{w_Gw_P}^{-1}\cdot \mu) = 1$ if and only if $n_{w_Gw_P}^{-1}\cdot \mu$ is $P$-negative.
Set $P' = n_{w_Gw_P}\opposite{P}n_{w_Gw_P}^{-1}$.
Then we have $(w_Gw_P)(\Sigma^+\setminus\Sigma_P^+) = \Sigma^-\setminus\Sigma_{P'}^-$.
Hence $q(P,n_{w_Gw_P}^{-1}\cdot \mu) = 1$ if and only if $\mu$ is $P'$-positive.
Therefore $q(P,n_{w^{-1}w_Gw_P}^{-1}\cdot \lambda) = 1$ if and only if $n_w\cdot \lambda$ is $P'$-positive.
\end{rem}\end{rem}

\subsection{Sum and intersections}
In this subsection, assume that $P$ is a parabolic subgroup and $\sigma$ an $\mathcal{H}_P$-module which has the extension to $\mathcal{H}$.
Let $Q$ be a parabolic subgroup containing $P$.
Let $A\subset W_0^Q$ be an open subset.
Then we have $I_Q(\sigma)_A\subset I_Q(\sigma)$.
In this subsection we prove the following lemma using an argument in \cite{AHHV2}.
\begin{lem}\label{lem:intersection and sum, Bruhat cel}
Let $\mathcal{P}\subset \{Q_1\supset Q\}$ be a subset.
Then we have 
\[
I_Q(e_Q(\sigma))_A\cap \sum_{Q_1\in \mathcal{P}}(I_{Q_1}(e_{Q_1}(\sigma)))
=
\sum_{Q_1\in \mathcal{P}}(I_Q(e_Q(\sigma))_A\cap I_{Q_1}(e_{Q_1}(\sigma))).
\]
\end{lem}

\begin{rem}\label{rem:exactness on each Bruhat cell}
The above lemma is equivalent to the following.
Let $\mathcal{P}\subset \{Q_1\supset Q\}$ be a subset and set $\pi = I_Q(e_Q(\sigma))/\sum_{Q_1\in \mathcal{P}}I_{Q_1}(e_{Q_1}(\sigma))$.
Put $I_{Q_1,A} = I_{Q_1}(e_{Q_1}(\sigma))\cap I_Q(e_Q(\sigma))_A$ and let $\pi_A$ be the image of $I_Q(e_Q(\sigma))_A$.
Then the sequence
\[
\bigoplus_{Q_1\in \mathcal{P}}I_{Q_1,A}\to I_{Q,A}\to \pi_{A}\to 0
\]
is exact.
\end{rem}
Take a minimal element $w\in A$ and set $A' = A\setminus\{w\}$.
\begin{lem}\label{lem:successive quotient of Bruhat filtration, small induction}
Let $Q_1\supset Q$.
The injective map $I_{Q_1,A}/I_{Q_1,A'}\hookrightarrow I_{Q,A}/I_{Q,A'}$ is surjective if $w\in W_0^{Q_1}$ and $0$ otherwise.
\end{lem}
\begin{proof}
Recall that $\varphi\mapsto \varphi(T_{n_w})$ gives an isomorphism $I_{Q,A}/I_{Q,A'}\simeq \sigma$.
Assume that $w\in W_0^{Q_1}$ and let $\varphi \in I_{Q,A}$.
Set $x = \varphi(T_{n_w})$ and take $\psi\in I_{Q_1}(e_{Q_1}(\sigma))$ such that $\psi(T_{n_w}) = x$ and $\psi(T_{n_v}) = 0$ for $v\in W_0^{Q_1}\setminus\{w\}$.
Let $v\in W_0^Q$ and take $v_1\in W_0^{Q_1}$ and $v_2\in W_{Q_1,0}$ such that $v = v_1v_2$.
Then we have $\psi(T_{n_v}) = \psi(T_{n_{v_1}})e_{Q_1}(\sigma)(T^{Q_1}_{n_{v_2}})$ since $j_{Q_1}^{-*}(T_{n_{v_2}}^{Q_1}) = T_{n_{v_2}}$ by Lemma~\ref{cor:image by j, positie and negative case}.
Hence if $\psi(T_{n_v})\ne 0$, then $v_1 = w$.
Therefore $v\in wW_{Q_1,0}$.
Since $w\in W_0^{Q_1}$, any element in $wW_{Q_1,0}$ is greater than or equal to $w$.
Hence $v\in A$.
Namely $\psi\in I_{Q_1,A}$ and we proved the surjectivity of the map in the lemma.

Assume that $w\notin W_0^{Q_1}$ and take $w_1\in W_0^{Q_1}$ and $w_2\in W_{Q_1,0}$ such that $w = w_1w_2$.
Then $w_1 < w$.
Since $w$ is minimal in $A$, we have $w_1\notin A$.
Hence for $\psi\in I_{Q_1,A}$, we have $\psi(T_{n_{w_1}}) = 0$.
Therefore we have $\psi(T_{n_w}) = \psi(T_{n_{w_1}})e_{Q_1}(\sigma)(T^{Q_1}_{n_{w_2}}) = 0$.
Hence we have $\psi\in I_{Q_1,A'}$.
We get $I_{Q_1,A} = I_{Q_1,A'}$.
\end{proof}

\begin{proof}[Proof of Lemma~\ref{lem:intersection and sum, Bruhat cel}]
Obviously the left hand side contains the right hand side.
We prove that the right hand side contains the left hand side by backward induction on $\#A$.
We assume that the lemma is true for $A$ and prove the lemma for $A'$.
By inductive hypothesis, we  have
\begin{align*}
I_{Q,A'}\cap \sum_{Q_1\in \mathcal{P}}I_{Q_1}(e_{Q_1}(\sigma))
&=  I_{Q,A'}\cap I_{Q,A} \cap \sum_{Q_1\in \mathcal{P}}I_{Q_1}(e_{Q_1}(\sigma))\\
& =
I_{Q,A'}\cap \sum_{Q_1\in \mathcal{P}}I_{Q_1,A}.
\end{align*}

First assume that $w\notin W_0^{Q_1}$ for any $Q_1\in \mathcal{P}$.
Then by the above lemma, we have $I_{Q_1,A} = I_{Q_1,A'}$ for any $Q_1\in \mathcal{P}$.
Hence
\[
I_{Q,A'}\cap \sum_{Q_1\in \mathcal{P}}I_{Q_1,A}
= I_{Q,A'}\cap \sum_{Q_1\in \mathcal{P}}I_{Q_1,A'}
= \sum_{Q_1\in \mathcal{P}}I_{Q_1,A'}.
\]

Now assume that there exists $Q_1\in \mathcal{P}$ such that $w\in W_0^{Q_1}$.
Take $\varphi_{Q_1}\in I_{Q_1,A}$ such that $\sum_{Q_1\in \mathcal{P}}\varphi_{Q_1}\in I_{Q,A'}$.
Set $\mathcal{P}_1 = \{Q_1\in \mathcal{P}\mid w\in W_0^{Q_1}\}$.
Let $Q_0$ be a parabolic subgroup corresponding to $\bigcup_{Q_1\in \mathcal{P}_1}\Delta_{Q_1}$.
By the above lemma, for each $Q_1\in \mathcal{P}_1$, we have $I_{Q_0,A}/I_{Q_0,A'}\simeq I_{Q_1,A}/I_{Q_1,A'}$.
Therefore for each $Q_1\in \mathcal{P}_1$ there exists $\varphi'_{Q_1}\in I_{Q_0,A}$ such that $\varphi_{Q_1} - \varphi'_{Q_1}\in I_{Q_1,A'}$.
Then we have
\begin{align*}
\sum_{Q_1\in \mathcal{P}}\varphi_{Q_1} & = \sum_{Q_1\in \mathcal{P}\setminus\mathcal{P}_1}\varphi_{Q_1} + \sum_{Q_1\in \mathcal{P}_1}(\varphi_{Q_1} - \varphi'_{Q_1}) + \sum_{Q_1\in \mathcal{P}_1}\varphi'_{Q_1}\\
& \in \sum_{Q_1\in \mathcal{P}\setminus\mathcal{P}_1}I_{Q_1,A} + \sum_{Q_1\in \mathcal{P}}I_{Q_1,A'} + \sum_{Q_1\in \mathcal{P}_1}\varphi'_{Q_1}.
\end{align*}
By the above lemma, for $Q_1\in \mathcal{P}\setminus \mathcal{P}_1$, we have $I_{Q_1,A} = I_{Q_1,A'}$.
Hence
\[
\sum_{Q_1\in \mathcal{P}}\varphi_{Q_1} \in \sum_{Q_1\in \mathcal{P}\setminus\mathcal{P}_1}I_{Q_1,A'} + \sum_{Q_1\in \mathcal{P}_1}I_{Q_1,A'} + \sum_{Q_1\in \mathcal{P}_1}\varphi'_{Q_1}.
\]
In particular, $\sum_{Q_1\in \mathcal{P}_1}\varphi'_{Q_1}\in I_{Q,A'}\cap I_{Q_0}(e_{Q_0}(\sigma)) = I_{Q_0,A'}$ since $\sum_{Q_1\in \mathcal{P}}\varphi_{Q_1}\in I_{Q,A'}$.
For $Q_1\in \mathcal{P}_1$, we have $I_{Q_0,A'}\subset I_{Q_1,A'}$.
Hence $I_{Q_0,A'}\subset \sum_{Q_1\in \mathcal{P}_1}I_{Q_1,A'}$.
Therefore
\[
\sum_{Q_1\in \mathcal{P}}\varphi_{Q_1} \in \sum_{Q_1\in \mathcal{P}\setminus\mathcal{P}_1}I_{Q_1,A'} + \sum_{Q_1\in \mathcal{P}_1}I_{Q_1,A'} + \sum_{Q_1\in \mathcal{P}_1}I_{Q_1,A'} = \sum_{Q_1\in \mathcal{P}}I_{Q_1,A'}.
\]
We get the lemma.
\end{proof}

\subsection{A filtration on generalized Steinberg modules}
As in the previous section, let $P$ be a parabolic subgroup and $\sigma$ an $\mathcal{H}_P$-module which has the extension to $\mathcal{H}$.
Let $Q$ be a parabolic subgroup containing $P$.
As in Remark~\ref{rem:exactness on each Bruhat cell}, for each open subset $A\subset W^Q_0$, set $I_{Q_1,A} = I_{Q_1}(e_{Q_1}(\sigma))\cap I_Q(e_Q(\sigma))_A$ and let $\St_{Q,A}$ be the image of $I_Q(e_Q(\sigma))_A$.
Let $w\in A$ be a minimal element and put $A' = A\setminus\{w\}$.
Then we have a commutative diagram.
\[
\begin{tikzcd}
0\arrow{d} & 0\arrow{d} & 0\arrow{d} & \\
\bigoplus_{Q_1\supsetneq Q}I_{Q_1,A'}\arrow{r}\arrow{d} & I_{Q,A'}\arrow{r}\arrow{d} & \St_{Q,A'}\arrow{r}\arrow{d} & 0\\
\bigoplus_{Q_1\supsetneq Q}I_{Q_1,A }\arrow{r}\arrow{d} & I_{Q,A }\arrow{r}\arrow{d} & \St_{Q,A }\arrow{r}\arrow{d} & 0\\
\bigoplus_{Q_1\supsetneq Q}I_{Q_1,A}/I_{Q_1,A'}\arrow{r}\arrow{d} & I_{Q,A}/I_{Q,A'}\arrow{r}\arrow{d} & \St_{Q,A}/\St_{Q,A'}\arrow{r}\arrow{d} & 0\\
0 & 0 & 0.
\end{tikzcd}
\]
Since the first two rows are exact by Remark~\ref{rem:exactness on each Bruhat cell}, the third row is also exact.
If $w\notin W^{Q_1}_0$ for any $Q_1\supsetneq Q$, we have $I_{Q_1,A}/I_{Q_1,A'} = 0$ for any $Q_1\supsetneq Q$ by Lemma~\ref{lem:successive quotient of Bruhat filtration, small induction}.
Hence $I_{Q,A}/I_{Q,A'}\xrightarrow{\sim}\St_{Q,A}/\St_{Q,A'}$.
If $w\in W^{Q_1}_0$ for some $Q_1\supsetneq Q$, we have $I_{Q_1,A}/I_{Q_1,A'}\xrightarrow{\sim}I_{Q,A}/I_{Q,A'}$ by Lemma~\ref{lem:successive quotient of Bruhat filtration, small induction}.
Hence $\bigoplus_{Q_1\supsetneq Q}I_{Q_1,A}/I_{Q_1,A'}\to I_{Q,A}/I_{Q,A'}$ is surjective.
Therefore we have $\St_{Q,A}/\St_{Q,A'} = 0$.
Summarizing this argument, we get the following lemma.

\begin{lem}\label{lem:successive quotient of filtration on St}
If $w\in W^{Q_1}_0$ for some $Q_1\supsetneq Q$ then $\St_{Q,A}/\St_{Q,A'} = 0$.
If $w\notin W^{Q_1}_0$ for any $Q_1\supsetneq Q$ then $I_{Q,A}/I_{Q,A'}\xrightarrow{\sim}\St_{Q,A}/\St_{Q,A'}$.
\end{lem}

Using this, we give a description of $\St_P(\sigma)$.
Let $P_2$ be a parabolic subgroup corresponding to $\Delta\setminus\Delta_P$.
Note that, since $\sigma$ is assumed to be have the extension to $\mathcal{H}$, $\Delta_P$ and $\Delta\setminus\Delta_{P} = \Delta_{P_2}$ are orthogonal to each others.
Hence $W_0^P = W_{0,P_2}$.
\begin{prop}\label{prop:characterization of Steinberg}
The representation $\pi = \St_P(\sigma)$ is isomorphic to $\sigma$ as an $(\mathcal{H}_P^+,j_P^+)$-module and for $w\in W_{\aff,P_2}(1)$, $\pi(T_{w}) = (-1)^{\ell(w)}$.
\end{prop}
\begin{proof}
Let $w\in W_0^P$ and assume that $w\notin W_0^{Q_1}$ for any $Q_1\supsetneq P$.
Let $\alpha\in\Delta\setminus\Delta_P$ and consider the parabolic subgroup $Q_1$ corresponding to $\Delta_P\cup \{\alpha\}$.
Then by the assumption $w(\Delta_{Q_1})\not\subset \Sigma^+$.
Since $w(\Delta_P) \subset\Sigma^+$, we have $w(\alpha) < 0$.
Therefore $w\in W_0^P = W_{0,P_2}$ satisfies that $w(\alpha) < 0$ for any $\alpha\in\Delta_{P_2}$.
Hence $w = w_{P_2} = w_Gw_P$.
Combining the above lemma, we get the following.
Note that $\{w_Gw_P\}$ is open.
\begin{itemize}
\item $\St_{P}(\sigma)/\St_{P,\{w_Gw_P\}}$ has a filtration with zero successive quotients. Hence $\St_P(\sigma) = \St_{P,\{w_Gw_P\}}$.
\item $\St_{P,\{w_Gw_P\}}\simeq I_P(\sigma)_{\{w_Gw_P\}}$.
\end{itemize}
Set $\sigma' = I_P(\sigma)_{\{w_Gw_P\}} = \{\varphi\in I_P(\sigma)\mid \varphi(T_{n_v}) = 0\ (v\in W^{P}_0\setminus\{w_{G}w_{P}\})\}$.
Then $\sigma'\hookrightarrow I_P(\sigma)\to \St_P(\sigma)$ is an isomorphism.
By Proposition~\ref{prop:tensor description of I_P}, $\sigma'$ is $j_P^+(\mathcal{H}_P^+)$-stable and isomorphic to $n_{w_Gw_P}\sigma$ as $(\mathcal{H}_{P}^+,j_P^+)$-modules.
By Lemma~\ref{lem:twist by longest element stabilize sigma}, we get the first part of the proposition.

Next, we prove $\pi(T_{w}) = (-1)^{\ell(w)}$ for $w\in W_{\aff,P_2}(1)$ by the following three steps.
\begin{enumerate}
\item The claim is true for $w = n_v$ where $v\in W_{0,P_2}$.
\item For any $v_1,v_2\in W_{0,P_2}$ and $w\in W(1)$, we have $\pi(T_{n_{v_1}wn_{v_2}}) = \pi(T_{n_{v_1}})\pi(T_{w})\pi(T_{n_{v_2}})$.
\item We have $\pi(T_{w}) = (-1)^{\ell(w)}$ for any $w\in W_{\aff,P_2}(1)$.
\end{enumerate}

(1)
We may assume $v = s\in S_{0,P_2}$.
Let $\overline{\varphi}\in \St_P(\sigma)$ and $\varphi\in \sigma'$ its lift.
For $w\in W_{0}^P = W_{0,P_2}$, we have
\[
(\varphi T_{n_s})(T_{n_w}) =
\begin{cases}
\varphi(T_{n_{w_{P_2}}}) & w = sw_{P_2},\\
(q_s - 1)\varphi(T_{n_{w_{P_2}}}) & w = w_{P_2},\\
0 & \text{otherwise}.
\end{cases}
\]
Indeed, if $w < sw$, then $(\varphi T_{n_s})(T_{n_w}) = \varphi(T_{n_{sw}})$.
Hence if $sw\ne w_{P_2}$ then this is zero.
If $w > sw$, then $(\varphi T_{n_s})(T_{n_w}) = \varphi(T_{n_s}^2T_{n_{sw}}) = \varphi(c_{n_s}T_{n_w} + q_sT_{n_{sw}})$.
Since $s(sw) > sw$, $sw\ne w_{P_2}$.
Hence $\varphi(T_{n_{sw}}) = 0$.
Therefore $(\varphi T_{n_s})(T_{n_w}) = \varphi(c_{n_s}T_{n_w}) = \varphi(T_{n_w})e_Q(\sigma)(n_w^{-1}\cdot c_{n_s})$.
This is zero if $w\ne w_{P_2}$.
When $w = w_{P_2}$, take $c_t\in \Z$ such that $n_{w_{P_2}}^{-1}\cdot c_{n_s} = \sum_t c_tT_t$.
Since $n_{w_{P_2}}^{-1}\cdot c_{n_s}\in C[Z_\kappa\cap W_{\aff,P_2}(1)]$ (Lemma~\ref{lem:on c, about aff,Levi}), $c_t\ne 0$ implies $t\in Z_\kappa \cap W_{\aff,P_2}(1)$.
Hence $e_Q(\sigma)(T_t) = 1$ for such $t$.
Therefore $e_Q(\sigma)(n_{w_{P_2}}^{-1}\cdot c_{n_s}) = \sum_t c_t = q_s - 1$ by Proposition~\ref{prop:expansion of c}.
We get the above calculation.

Take $\alpha\in\Delta_{P_2}$ such that $s = s_\alpha$ and put $\alpha' = -w_{P_2}(\alpha)$ and $s' = s_{\alpha'}$.
Let $Q_1$ be a parabolic subgroup corresponding to $\Delta_P\cup\{\alpha'\}$.
Then $w_{Q_1} = w_Ps'$ and $w_Gw_{Q_1} = w_{P_2}s'$.
Define $\psi\in I_{Q_1}(e_{Q_1}(\sigma))$ by $\psi(T_{n_w}) = 0$ for $w\in W_0^{Q_1}\setminus \{w_{P_2}s'\}$ and $\psi(T_{n_{w_{P_2}s'}}) = \varphi(T_{n_{w_{P_2}}})$.
We prove
\[
\psi(T_{n_w}) = 
\begin{cases}
\varphi(T_{n_{w_{P_2}}}) & w = sw_{P_2},\\
q_s\varphi(T_{n_{w_{P_2}}}) & w = w_{P_2},\\
0 & \text{otherwise}.
\end{cases}
\]
Let $w \in W^P_0 = W_{0,P_2}$ and take $w_1\in W^{Q_1}_0$ and $w_2\in W_{0,Q_1\cap P_2}$ such that $w = w_1w_2$.
Then we have $\psi(T_{n_w}) = \psi(T_{n_{w_1}})e_{Q_1}(T_{n_{w_2}}^{Q_1})$.
Hence if $w_1\ne w_{P_2}s'$, namely $w\notin w_{P_2}s'W_{0,Q_1\cap P_2}$, we have $\psi(T_{n_w}) = 0$.
Note that $W_{0,Q_1\cap P_2} = \{1,s'\}$ and $w_{P_2}s' = sw_{P_2}$.
Therefore $\psi(T_{n_w}) = 0$ if $w\ne w_{P_2},sw_{P_2}$.
We have $\psi(T_{n_{w_{P_2}}}) = \psi(T_{n_{w_{P_2}s'}})e_{Q_1}(\sigma)(T^{Q_1}_{n_{s'}})$ and $e_{Q_1}(\sigma)(T^{Q_1}_{n_{s'}}) = e_{Q_1}(\sigma)(T_{n_{s'}}^{Q_1*} + c_{n_{s'}}) = 1 + q_{s'} - 1 = q_{s'}$.
Since $s$ and $s'$ is conjugate, $q_s = q_{s'}$.
Hence $\psi(T_{n_{w_{P_2}}}) = q_s\varphi(T_{n_{w_{P_2}}})$.
Finally we have $\psi(T_{n_{sw_{P_2}}}) = \varphi(T_{n_{w_{P_2}}})$ by the definition of $\psi$.
We get the above formula.

Therefore $(\varphi T_{n_s} - \psi)(T_{n_w}) = 0 = -\varphi(T_{n_w})$ if $w\ne w_{P_2}$ and equal to $-\varphi(T_{n_{w_{P_2}}})$ if $w = w_{P_2}$.
Since an element in $I_P(\sigma)$ is determined by the value at $W_0^P = W_{0,P_2}$, we have $\varphi T_{n_s} - \psi = -\varphi$.
Hence $\overline{\varphi} T_{n_s} = -\overline{\varphi}$ in $\St_P(\sigma)$.

(2)
Assume that $v_2 = 1$.
To prove (2), by induction on the length of $v_1$, we may assume $v_1 = s\in S_0$.
If $n_sw > w$, then it is obvious.
Assume that $n_sw < w$.
Then we have $\pi(T_{w}) = \pi(T_{n_s^{-1}}T_{n_sw}) = \pi(T_{n_s^{-1}})\pi(T_{n_sw}) = -\pi(T_{n_sw})$ by (1).
Hence $\pi(T_{n_sw}) = -\pi(T_{w}) = \pi(T_{n_s})\pi(T_{w})$.
The same argument implies (2) for any $v_2$.

(3)
Take $w\in W_{\aff,P_2}(1)$.
Then we can take $w_1,w_2\in W_{0,P_2}$ and $\lambda\in \Lambda(1)\cap W_{\aff,P_2}(1)$ such that $w = n_{w_1}\lambda n_{w_2}$ and $\lambda$ is anti-dominant with respect to $\Sigma_{P_2}^+$.
Since $\Delta_P$ is orthogonal to $\Delta_{P_2}$, $\Sigma_{P_2}^+ = \Sigma^+\setminus\Sigma_P^+$.
Hence $\lambda$ is $P$-positive.
As in Remark~\ref{rem:value at Lambda_aff,P_2}, $\sigma(T_\lambda^P) = 1$.

Since $\lambda\in \Lambda(1)\cap W_\aff(1)$, $\ell(\lambda)$ is even by Lemma~\ref{lem:length of lambda_aff is even}.
On the other hand, since $\lambda$ is $P$-positive, we have $\pi(T_\lambda) = \sigma(T^P_\lambda)$.
Therefore $\pi(T_\lambda)  = 1 = (-1)^{\ell(\lambda)}$.
Using (2), we have 
\[
\pi(T_{w}) = \pi(T_{n_{w_1}})\pi(T_\lambda)\pi(T_{n_{w_2}}) = (-1)^{\ell(w_1) + \ell(\lambda) + \ell(w_2)} = (-1)^{\ell(w)}.
\]
We get the proposition.
\end{proof}

\subsection{Tensor product decomposition}
For the content of this subsection, see also \cite{AHHV2}.
We start with the following lemma.
\begin{lem}\label{lem:tensor product module}
Let $P_1,P_2$ be parabolic subgroups such that $\Delta_{P_1}$ is orthogonal to $\Delta_{P_2}$ and $\Delta_{P_1}\cup\Delta_{P_2} = \Delta$.
Assume that an $\mathcal{H}_{P_i}$-module $\sigma_i$ has the extension to $\mathcal{H}$.
Then we have the unique action of $\mathcal{H}$ on $e_G(\sigma_1)\otimes e_G(\sigma_2)$ which satisfies $(x_1\otimes x_2)T_w^* = x_1T_w^*\otimes x_2T_w^*$ for $w\in W(1), x_1\in e_G(\sigma_1)$ and $x_2\in e_G(\sigma_2)$.
\end{lem}
\begin{proof}
The action obviously satisfies the braid relations.
It is sufficient to prove that the action satisfies the quadratic relations.

Let $s\in W_\aff(1)$ be a lift of a simple reflection.
If $s\in W_{P_1}(1)$, then $s\in W_{\aff,P_1}(1)$.
Hence $e_G(\sigma_2)(T_s^*) = 1$.
Therefore $x_1T_s^*\otimes x_2T_s^* = x_1T_s^*\otimes x_2$.
Hence it satisfies the quadratic relations.
The quadratic relations hold too if $s\in W_{P_2}(1)$.
\end{proof}
\begin{rem}\label{rem:action of aff to tensor product module}
By the above proof, we have $(x_1\otimes x_2)T_w^* = x_1T_w^*\otimes x_2$ if $w\in W_{P_1,\aff}(1)$.
Hence $(x_1\otimes x_2)X = x_1X\otimes x_2$ for any $X\in \bigoplus_{w\in W_{P_1,\aff}(1)}CT_w^*$.
In particular, for $X = T_w$ or $X = E_o(w)$ for any $w\in W_{P_1,\aff}(1)$ and any orientation $o$.
We also have $(x_1\otimes x_2)X = x_1\otimes x_2X$ for any $X = T_w$ or $X = E_o(w)$ where $w\in W_{P_2,\aff}(1)$.
\end{rem}

In the rest of this subsection, let $P$ be a parabolic subgroup, $\sigma$ an $\mathcal{H}_P$-module which has the extension to $\mathcal{H}$.
Take a parabolic subgroup $Q$ containing $P$.
Let $P_2$ be the parabolic subgroup corresponding to $\Delta\setminus\Delta_P$.
Note that $\Delta_{P_2}$ is orthogonal to $\Delta_P$.

\begin{lem}\label{lem:parabolic induction of extension is extension}
We have $I_Q(\trivrep) \simeq e_G(I_{P\cap Q}^{P_2}(\trivrep))$.
More generally, for an $\mathcal{H}_{P_2\cap Q}$-module $\sigma$ which has an extension to $\mathcal{H}_Q$, we have $I_Q(e_Q(\sigma))\simeq e_G(I_{P_2\cap Q}^{P_2}(\sigma))$.
\end{lem}
We use the following lemma.
\begin{lem}\label{lem:localization, middle positive/negative subalgebra}
Let $P,Q$ be parabolic subgroups.
\begin{enumerate}
\item Let $\lambda_P^-$ as in Proposition~\ref{prop:localization as Levi subalgebra}. Then $\mathcal{H}_{P\cap Q}^{P-} \simeq \mathcal{H}_{P\cap Q}^-(T^{P\cap Q}_{\lambda_P^-})^{-1}$.
\item Let $\lambda_P^+$ as in Proposition~\ref{prop:localization as Levi subalgebra}. Then $\mathcal{H}_{P\cap Q}^{P+} \simeq \mathcal{H}_{P\cap Q}^+(T^{P\cap Q}_{\lambda_P^+})^{-1}$.
\end{enumerate}
\end{lem}
\begin{proof}
We only prove the first statement.
The proof of the second statement is the same.

Let $\lambda\in \Lambda(1)$ such that $\langle \alpha,\nu(\lambda)\rangle\ge 0$ for any $\alpha\in\Sigma_P^+\setminus\Sigma_{P\cap Q}^+$.
We can take $n\in\Z_{>0}$ such that $\langle \alpha,\nu(\lambda(\lambda_P^-)^n)\rangle \ge 0$ for any $\alpha\in\Sigma^+\setminus\Sigma_P^+$.
Since $\langle \alpha,\nu((\lambda_P^-)^n)\rangle = 0$ for any $\alpha\in\Sigma_P^+$, we have $\langle \alpha,\nu(\lambda(\lambda_P^-)^n)\rangle = \langle\alpha,\nu(\lambda)\rangle\ge 0$ for any $\alpha\in\Sigma_P^+\setminus\Sigma_{P\cap Q}^+$.
Hence $\langle \alpha,\nu(\lambda (\lambda_P^-)^n)\rangle\ge 0$ for any $\alpha\in\Sigma^+\setminus\Sigma_{P\cap Q}^+$.
Therefore for any $w\in W_{P\cap Q}(1)$ which is $(P\cap Q)$-negative in $P$, there exists $n\in\Z_{\ge 0}$ such that $w(\lambda_P^-)^n$ is $(P\cap Q)$-negative in $G$.
Then we have $E_{o_{-,P\cap Q}}^{P\cap Q}(w) = E_{o_{-,P\cap Q}}^{P\cap Q}(w(\lambda_P^-)^n)(T^{P\cap Q}_{\lambda_{P}^-})^{-n}\in \mathcal{H}_{P\cap Q}^-(T^{P\cap Q}_{{\lambda_{P}}^-})^{-1}$.
\end{proof}

\begin{proof}[Proof of Lemma~\ref{lem:parabolic induction of extension is extension}]
Consider the map
\[
I_Q(e_Q(\sigma))= \Hom_{(\mathcal{H}_Q^-,j_Q^{-*})}(\mathcal{H},e_Q(\sigma))
\to
\Hom_{(\mathcal{H}_{P_2\cap Q}^-,j_{P_2\cap Q}^{P_2-*})}(\mathcal{H}_{P_2}^-,\sigma)
\]
defined by $\varphi\mapsto \varphi\circ j_{P_2}^{-*}$.
Then obviously this is $(\mathcal{H}_{P_2}^-,j_{P_2}^{-*})$-equivariant.

Let $\lambda_{P_2}^-$ as in Proposition~\ref{prop:localization as Levi subalgebra}.
Then we have $\mathcal{H}_{P_2}^-(T_{\lambda_{P_2}^-}^{P_2})^{-1} = \mathcal{H}_{P_2}$.
By Lemma~\ref{lem:localization, middle positive/negative subalgebra}, we have $\mathcal{H}_{P_2\cap Q}^-(T_{\lambda_{P_2}^-}^{P_2\cap Q})^{-1} = \mathcal{H}_{P_2\cap Q}^{P_2-}$.
Therefore we have
\[
I_{P_2\cap Q}^{P_2}(\sigma)
= \Hom_{(\mathcal{H}_{P_2\cap Q}^{P_2-},j_{P_2\cap Q}^{P_2-*})}(\mathcal{H}_{P_2},\sigma)
=
\Hom_{(\mathcal{H}_{P_2\cap Q}^-,j_{P_2\cap Q}^{P_2-*})}(\mathcal{H}_{P_2}^-,\sigma).
\]
Hence we get an $(\mathcal{H}_{P_2}^-,j_{P_2}^{-*})$-homomorphism $I_Q(e_Q(\sigma))\to I_{P_2\cap Q}^{P_2}(\sigma)$.
For $v\in W_0^Q = W_{0,P_2}^{P_2\cap Q}$, we have $\varphi(T_{n_v}) = (\varphi\circ j_{P_2}^{-*})(T_{n_v}^{P_2})$ by Corollary~\ref{cor:image by j, positie and negative case}.
Hence this homomorphism is an isomorphism by Proposition~\ref{prop:decomposition of I_P}.

Let $w\in W_{P,\aff}(1)$, $v\in W_0^Q$ and $\varphi\in I_Q(e_Q(\sigma))$.
We have $W_0^Q\subset W_{0,P_2}$.
Since $n_v\in W_{P_2,\aff}(1)$, $\ell(w n_v) = \ell(w) + \ell(n_v)$.
The subgroup $W_{P,\aff}(1)$ is normal in $W(1)$~\cite[II.7 Remark~4]{arXiv:1412.0737_accepted}.
Hence $n_v^{-1}wn_v\in W_{P,\aff}(1)$ and $\ell(wn_v) = \ell(n_v) + \ell(n_v^{-1}wn_v)$.
Therefore $(\varphi T^*_{w})(T^*_{n_v}) = \varphi(T^*_{w}T^*_{n_v}) = \varphi(T^*_{wn_v}) = \varphi(T^*_{n_v}T^*_{n_v^{-1}wn_v}) = \varphi(T^*_{n_v})e_Q(\sigma)(T^{Q*}_{n_v^{-1}wn_v})$.
By the definition of the extension, $e_Q(\sigma)(T^{Q*}_{n_v^{-1}wn_v}) = 1$.
Hence $(\varphi T_w^*)(T_{n_v}^*) = \varphi(T_{n_v}^*)$.
Since an element in $I_Q(e_Q(\sigma))$ is determined by the values at $T^*_{n_v}$ for $v\in W^Q_0$ (Proposition~\ref{prop:decomposition of I_P}), $T_{w}^*$ acts trivially on $I_Q(e_Q(\sigma))$.
Hence $I_Q(e_Q(\sigma))\simeq e_G(I_{P_2\cap Q}^{P_2}(\sigma))$.
\end{proof}

By this lemma and Lemma~\ref{lem:tensor product module}, we have an $\mathcal{H}$-module structure on $I_Q(\trivrep)\otimes e_G(\sigma)$.
\begin{prop}\label{prop:decomposition into tensor product}
We have $I_Q(e_Q(\sigma))\simeq I_Q(\trivrep)\otimes e_G(\sigma)$.
\end{prop}
\begin{proof}
Define the homomorphism $I_Q(\trivrep)\otimes e_G(\sigma)\to e_Q(\sigma)$ by $\varphi\otimes x\mapsto \varphi(1)x$.
For $w\in W_Q^{-}(1)$, we have 
\begin{align*}
(\varphi\otimes x)T_w^*  = \varphi T_w^*\otimes xe_G(\sigma)(T_w^*) & \mapsto (\varphi T_w^*)(1)xe_G(\sigma)(T_w^*)\\
& = \varphi(1)\trivrep(T^{Q*}_w)xe_G(\sigma)(T_w^*)\\
& = \varphi(1)xe_G(\sigma)(T_w^*).
\end{align*}
By \cite[Proposition~4.19]{arXiv:1406.1003_accepted}, we have $e_G(\sigma)(T_w^*) = e_Q(\sigma)(T_w^{Q*})$.
Hence the homomorphism is $(\mathcal{H}_Q^-,j_Q^{-*})$-equivariant.
Therefore we get a homomorphism $I_Q(\trivrep)\otimes e_G(\sigma)\to I_Q(e_Q(\sigma))$.
By Proposition~\ref{prop:decomposition of I_P}, we have the decompositions $I_Q(\trivrep)\otimes e_G(\sigma) = \bigoplus_{w\in W^Q_0}C\otimes e_G(\sigma)$ and $I_Q(e_Q(\sigma)) = \bigoplus_{w\in W^Q_0}e_Q(\sigma)$ which makes the following diagram commutative
\[
\begin{tikzcd}
I_Q(\trivrep)\otimes e_G(\sigma)\arrow{r}\arrow[-]{d}{\wr} & I_Q(e_Q(\sigma))\arrow[-]{d}{\wr}\\
\displaystyle\bigoplus_{w\in W^Q_0}e_G(\sigma)\arrow[double line,-]{r} & \displaystyle\bigoplus_{w\in W^Q_0}e_G(\sigma)\\
\end{tikzcd}
\]
Hence we get the proposition.
\end{proof}

\section{Inductions}\label{sec:Inductions}
Let $P$ be a parabolic subgroup.
Recall that our parabolic induction $I_P$ is defined by
\[
I_P(\sigma) = \Hom_{(\mathcal{H}_P^-,j_P^{-*})}(\mathcal{H},\sigma)
\]
for an $\mathcal{H}_P$-module $\sigma$.
Since we also have two subalgebras $\mathcal{H}_P^\pm$ and four homomorphisms $j_P^{\pm},j_P^{\pm*}$, we can define four ``inductions''.
In this section, we study the relations between such functors.

\subsection{Modules $\sigma_{\ell - \ell_P}$}
Before studying such functors, we first consider the representation $\sigma_{\ell  - \ell_P}$ attached to a representation $\sigma$ of $\mathcal{H}_P$ where $P$ is a parabolic subgroup.

Let $P$ be a parabolic subgroup and $\sigma$ an $\mathcal{H}_P$-module.
We define a linear map $\sigma_{\ell - \ell_P}$ by 
\[
\sigma_{\ell - \ell_P}(T_w^P) = (-1)^{\ell(w) - \ell_P(w)}\sigma(T_w^P)
\]
for $w\in W_P(1)$.
From the following lemma, this is again an $\mathcal{H}_P$-module.

\begin{lem}\label{lem:tensor of length?}
\begin{enumerate}
\item The linear map $\mathcal{H}_P\to C$ defined by $T^P_w\mapsto (-1)^{\ell(w)}$ is a character of $\mathcal{H}_P$ and it sends $T_s$ to $-1$ for any $s\in S_{\aff,P}(1)$.\label{enum:length gives a character}
\item Let $\pi$ be an $\mathcal{H}$-module, $\chi$ a character of $\mathcal{H}$ such that $\chi(T_t) = 1$ for any $t\in Z_\kappa$ and $\chi(T_s) = -1$ for any $s\in S_\aff(1)$.
Then the linear map defined by $T_w\mapsto (-1)^{\ell(w)}\chi(T_w)\pi(T_w)$ is an $\mathcal{H}$-module.\label{enum:tensoring character}
\end{enumerate}
\end{lem}
\begin{proof}
(\ref{enum:length gives a character})
First we prove the last claim.
Let $s\in S_{\aff,P}(1)$.
Then $s$ is a reflection in $W$.
Hence $\ell(s)$ is odd.
Therefore we have $(-1)^{\ell(s)} = -1$.

We check the quadratic relations.
Let $s\in S_{\aff,P}(1)$.
The quadratic relation is $(T^P_{s})^2 = q_{s,P}T^P_{s^2} + c_{s}T^P_{s}$ and the left hand side goes to $1$ by the map in the lemma.
Since $s^2\in Z_\kappa$, $T^P_{s^2}$ goes to $1$.
Hence $q_{s,P}T^P_{s^2}$ goes to $q_{s,P}$.
The image of $c_{s}$ under the map is $q_{s,P} - 1$ by Proposition~\ref{prop:expansion of c}.
We have already proved that $T^P_s$ goes to $-1$.
Hence the right hand side goes to $q_{s,P} + (q_{s,P} - 1)(-1) = 1$.
Therefore the map preserves the quadratic relations.

We check the braid relations.
Let $w_1,w_2\in W_P(1)$.
Then $T^P_{w_1w_2}$ goes to $(-1)^{\ell(w_1w_2)} = (-1)^{\ell(w_1) + \ell(w_2)}$ which is the product of the images of $T^P_{w_1}$ and $T^P_{w_2}$.
Hence the map preserves the braid relations.
We get (\ref{enum:length gives a character}).

(\ref{enum:tensoring character})
Let $\pi'$ be the map given in (\ref{enum:tensoring character}) and first we check that $\pi'$ preserves the quadratic relations.
Let $s\in S_\aff(1)$.
We prove $\pi'(T_{s})^2 = \pi'(q_sT_{s^2}) + \pi'(c_sT_s)$.
We calculate the right hand side.
Since $s^2\in Z_\kappa$, we have $\chi(T_{s^2}) = 1$.
We also have $\ell(s^2) = 0$.
Hence $\pi'(T_{s^2}) = \pi(T_{s^2})$.
Take $c_s(t)\in \Z$ such that $c_s = \sum_{t\in Z_\kappa}c_s(t)T_t$.
Then we have $\pi'(c_sT_s) = \sum_{t\in Z_\kappa}c_s(t)\pi'(T_tT_s) = \sum_{t\in Z_\kappa}c_s(t)(-1)^{\ell(ts)}\chi(T_t)\chi(T_s)\pi(T_t)\pi(T_s)$.
Since $t\in Z_\kappa$, we have $\ell(ts) = \ell(s) = 1$.
We also have $\chi(T_t) = 1$ by the assumption.
Hence $\pi'(c_sT_s) = -\chi(T_s)\sum_{t\in Z_\kappa}c_s(t)\pi(T_t)\pi(T_s) = -\chi(T_s)\pi(c_s)\pi(T_s)$.
Therefore $\pi'(q_sT_{s^2}) + \pi'(c_sT_s) = q_s\pi(T_{s^2}) - \chi(T_s)\pi(c_s)\pi(T_s)$.
Since we assume $\chi(T_s) = -1$, we get $q_s\pi(T_{s^2}) - \chi(T_s)\pi(c_s)\pi(T_s) = q_s\pi(T_{s^2}) + \pi(c_s)\pi(T_s) = \pi(q_sT_{s^2} + c_sT_s) = \pi(T_s^2) = \pi(T_s)^2$.
We have $\pi'(T_s)^2 = \chi(T_s)^2\pi(T_s)^2$.
By the assumption, $\chi(T_s) = -1$.
Hence $\pi'(T_s)^2 = \pi(T_s)^2$.
We get the quadratic relations.

Let $w_1,w_2\in W_P(1)$ such that $\ell(w_1w_2) = \ell(w_1) + \ell(w_2)$.
Then 
\begin{align*}
\pi'(T_{w_1}T_{w_2}) & = (-1)^{\ell(w_1w_2)}\chi(T_{w_1w_2})\pi(T_{w_1w_2})\\
& = (-1)^{\ell(w_1) + \ell(w_2)}\chi(T_{w_1}T_{w_2})\pi(T_{w_1}T_{w_2})\\
& = (-1)^{\ell(w_1)}(-1)^{\ell(w_2)}\chi(T_{w_1})\chi(T_{w_2})\pi(T_{w_1})\pi(T_{w_2}).
\end{align*}
Since $\chi(T_{w_2})$ is a scalar, we have
\begin{align*}
& (-1)^{\ell(w_1)}(-1)^{\ell(w_2)}\chi(T_{w_1})\chi(T_{w_2})\pi(T_{w_1})\pi(T_{w_2})\\
& = (-1)^{\ell(w_1)}\chi(T_{w_1})\pi(T_{w_1})(-1)^{\ell(w_2)}\chi(T_{w_2})\pi(T_{w_2})\\
& = \pi'(T_{w_1})\pi'(T_{w_2}).
\end{align*}
Hence $\pi'$ preserves the braid relations.
\end{proof}
Let $\iota = \iota_G\colon \mathcal{H}\to \mathcal{H}$ be a linear map defined by $T_w\mapsto (-1)^{\ell(w)}T^*_w$ for $w\in W(1)$.
Then this is an involution \cite[Proposition~4.23]{MR3484112}.
For any $\mathcal{H}$-module $\pi$, set $\pi^\iota = \pi\circ\iota$.

\begin{lem}\label{lem:order: ell - ell and twist}
We have $(\sigma^{\iota_P})_{\ell - \ell_P} = (\sigma_{\ell - \ell_P})^{\iota_P}$.
\end{lem}
\begin{proof}
We prove $(\sigma^{\iota_P})_{\ell - \ell_P}(T^P_w) = (\sigma_{\ell - \ell_P})^{\iota_P}(T^P_w)$ for any $w\in W_P(1)$.
We may assume that $w\in S_{\aff,P}(1)$ or $\ell_P(w) = 0$.

First assume that $w\in S_{\aff,P}(1)$ and denote $w$ by $s$.
By Lemma~\ref{lem:tensor of length?} (\ref{enum:length gives a character}), we have $(-1)^{\ell(s)} = -1$.
Hence $(\sigma^{\iota_P})_{\ell - \ell_P}(T^P_s) = (-1)^{\ell(s) - \ell_P(s)}\sigma(\iota_P(T^P_s)) = \sigma(\iota_P(T^P_s)) = -\sigma(T_s^{P*})$.
On the other hand, we have $(\sigma_{\ell - \ell_P})^{\iota_P}(T^P_s) = -(\sigma_{\ell  -\ell_P})(T_s^{P*}) = -(\sigma_{\ell  -\ell_P})(T^{P}_s) + (\sigma_{\ell  -\ell_P})(c_s)$.
We have $(\sigma_{\ell  -\ell_P})(T^P_s) = (-1)^{\ell(s) - \ell_P(s)}\sigma(T^P_s) = \sigma(T^P_s)$ since $(-1)^{\ell(s)} = -1$ and $\ell_P(s) = 1$.
We have $(\sigma_{\ell  -\ell_P})(c_s) = \sigma(c_s)$.
Hence $-(\sigma_{\ell  -\ell_P})(T^P_s) + (\sigma_{\ell  -\ell_P})(c_s) = -\sigma(T^P_s - c_s) = -\sigma(T_s^{P*})$.

Next assume that $\ell_P(w) = 0$.
Then $\iota_P(T^P_w) = (-1)^{\ell_P(w)}T_w^{P*} = T^P_w$.
Hence we have $(\sigma_{\ell - \ell_P})^{\iota_P} = \sigma_{\ell - \ell_P}(T^P_w) = (-1)^{\ell(w) - \ell_P(w)}\sigma(T^P_w)$.
We also have $(\sigma^{\iota_P})_{\ell - \ell_P}(T^P_w) = (-1)^{\ell(w) - \ell_P(w)}\sigma(\iota_P(T^P_w)) = (-1)^{\ell(w) - \ell_P(w)}\sigma(T^P_w)$.
We get the lemma.
\end{proof}

By Lemma~\ref{lem:order: ell - ell and twist}, we have $(\sigma^{\iota_P})_{\ell - \ell_P} = (\sigma_{\ell - \ell_P})^{\iota_P}$.
We denote it by $\sigma^{\iota_P}_{\ell - \ell_P}$.

\begin{lem}
We have $(\sigma_{\ell -\ell_P})_{\ell  -\ell_P} = \sigma$ and $(\sigma^{\iota_P})^{\iota_P} = \sigma$.
\end{lem}
\begin{proof}
Obvious from the definition and $\iota_P^2 = \id$.
\end{proof}
\begin{lem}\label{lem:twist by length and iota, twice}
We have $(\sigma_{\ell -\ell_P}^{\iota_P})_{\ell  -\ell_P}^{\iota_P} = \sigma$.
\end{lem}
\begin{proof}
The lemma follows from the above lemma and Lemma~\ref{lem:order: ell - ell and twist}.
\end{proof}

\begin{lem}\label{lem:twist by length, for T^*,E}
Let $\sigma$ be an $\mathcal{H}_P$-module and $w\in W_P(1)$.
\begin{enumerate}
\item We have $\sigma_{\ell - \ell_P}(T_w^{P*}) = (-1)^{\ell(w) - \ell_P(w)}\sigma(T^{P*}_w)$.\label{enum:twist by length for T^*}
\item For any orientation $o$, $\sigma_{\ell - \ell_P}(E^P_o(w)) = (-1)^{\ell(w) - \ell_P(w)}\sigma(E^P_o(w))$.
\end{enumerate}
\end{lem}
\begin{proof}
By Lemma~\ref{lem:order: ell - ell and twist}, we have 
\begin{align*}
\sigma_{\ell - \ell_P}(T_w^{P*}) & = (-1)^{\ell_P(w)}(\sigma_{\ell - \ell_P})^{\iota_P}(T_w^P)\\
& =  (-1)^{\ell_P(w)}(\sigma^{\iota_P})_{\ell - \ell_P}(T_w^P)\\
& =  (-1)^{\ell(w) - \ell_P(w)}(-1)^{\ell_P(w)}\sigma^{\iota_P}(T_w^P)\\
& =  (-1)^{\ell(w) - \ell_P(w)}\sigma(T_w^{P*}).
\end{align*}
Now we prove (2) by induction on the length of $w$.
If $\ell_P(w) = 0$, then $E^P_o(w) = T^P_w$.
Hence the lemma follows from the definition of $\sigma_{\ell - \ell_P}$.

Assume that $\ell_P(w) > 0$ and take $s\in S_{\aff,P}(1)$ such that $\ell_P(s^{-1}w) < \ell_P(w)$.
We have $E^P_o(w) = E^P_o(s)E^P_{o\cdot s}(s^{-1}w)$.
Since $E^P_o(s)$ is $T_s^P$ or $T_s^{P*}$, we have $\sigma_{\ell - \ell_P}(E^P_o(s)) = (-1)^{\ell(s) - \ell_P(s)}\sigma(E^P_o(s))$ as we have already proved.
By inductive hypothesis, 
\begin{align*}
\sigma_{\ell - \ell_P}(E_o^P(w)) & = \sigma_{\ell - \ell_P}(E^P_o(s))\sigma_{\ell - \ell_P}(E^P_{o\cdot s}(s^{-1}w)\\
& = (-1)^{\ell(s) - \ell_P(s)}\sigma(E^P_o(s))(-1)^{\ell(s^{-1}w) - \ell_P(s^{-1}w)}\sigma(E^P_{o\cdot s}(s^{-1}w)).
\end{align*}
Since $(-1)^{\ell(s) - \ell_P(s)}(-1)^{\ell(s^{-1}w) - \ell_P(s^{-1}w)} = (-1)^{\ell(w) - \ell_P(w)}$, we have 
\begin{align*}
\sigma_{\ell - \ell_P}(E_o^P(w)) & = (-1)^{\ell(w) - \ell_P(w)}\sigma(E^P_o(s))\sigma(E^P_{o\cdot s}(s^{-1}w))\\
& = (-1)^{\ell(w) - \ell_P(w)}\sigma(E^P_o(s)E^P_{o\cdot s}(s^{-1}w))\\
& = (-1)^{\ell(w) - \ell_P(w)}\sigma(E^P_o(w)).
\end{align*}
We get the lemma.
\end{proof}

\begin{rem}\label{rem:ell - ell_P is algebra hom}
Applying Lemma~\ref{lem:tensor of length?} to the right regular representation of $\mathcal{H}_P$ (namely, $\pi\colon \mathcal{H}_P\to \End(\mathcal{H}_P)^\opp$ defined by $Y\pi(X) = YX$), we get the following: the linear map $\mathcal{H}_P\to \mathcal{H}_P$ defined by $T^P_w\mapsto (-1)^{\ell(w) - \ell_P(w)}T^P_w$ is an algebra homomorphism.
By Lemma~\ref{lem:twist by length, for T^*,E}, this map sends $T_w^{P*}$ and $E^P_o(w)$ to $(-1)^{\ell(w) - \ell_P(w)}T_w^{P*}$ and $(-1)^{\ell(w) - \ell_P(w)}E^P_o(w)$, respectively where $o$ is any orientation.
\end{rem}

\begin{lem}\label{lem:twist by length is trivial on H_aff}
The algebra homomorphism defined in the remark is identity on $\mathcal{H}_{\aff,P}$.
Hence $\sigma_{\ell - \ell_P}|_{\mathcal{H}_{\aff,P}} = \sigma|_{\mathcal{H}_{\aff,P}}$ for any $\mathcal{H}_P$-module $\sigma$.
\end{lem}
\begin{proof}
The algebra $\mathcal{H}_{\aff,P}$ is generated by $T_s$ where $s\in S_{\aff,P}(1)$.
By Lemma~\ref{lem:tensor of length?}, $(-1)^{\ell(s)} = -1$.
Since the image of $s$ in $W_P$ is an affine simple reflection, we have $\ell_P(s) = 1$.
Hence $(-1)^{\ell_P(s)} = -1$.
\end{proof}

\begin{lem}
We have $\Delta(\sigma_{\ell - \ell_P}^{\iota_P}) = \Delta(\sigma)$.
\end{lem}
\begin{proof}
Let $\alpha\in\Delta(\sigma)\setminus\Delta_P$, $P_\alpha$ a parabolic subgroup corresponding to $\{\alpha\}$ and $\lambda\in \Lambda(1)\cap W_{\aff,P_\alpha}(1)$ and we prove $\sigma_{\ell - \ell_P}^{\iota_P}(T_\lambda^P) = 1$.
Since $\ell_P(\lambda) = 0$ by Lemma~\ref{lem:length zero}, we have $\iota_P(T_\lambda^P) = (-1)^{\ell_P(\lambda)}T_\lambda^{P*} = T_\lambda^P$.
Hence $\sigma_{\ell - \ell_P}^{\iota_P}(T_\lambda^P) = (-1)^{\ell(\lambda) - \ell_P(\lambda)}\sigma(T_\lambda^P)$.
We have $\ell_P(\lambda) = 0$.
Since $\lambda\in\Lambda(1)\cap W_\aff(1)$, $\ell(\lambda)$ is even by Lemma~\ref{lem:length of lambda_aff is even}.
Hence $\sigma_{\ell - \ell_P}^{\iota_P}(T_\lambda^P) = \sigma(T_\lambda^P) = 1$.
Therefore $\Delta(\sigma)\subset \Delta(\sigma_{\ell - \ell_P}^{\iota_P})$.
Applying this to $\sigma_{\ell - \ell_P}^{\iota_P}$, by Lemma~\ref{lem:twist by length and iota, twice}, we have $\Delta(\sigma_{\ell - \ell_P}^{\iota_P})\subset \Delta(\sigma)$.
\end{proof}

\begin{lem}\label{lem:twist by length and I_P,St}
Let $P_1\supset P$ be parabolic subgroups of $G$ and $\sigma$ an $\mathcal{H}_P$-module
\begin{enumerate}
\item We have $I^{P_1}_P(\sigma_{\ell_{P_1} - \ell_P})_{\ell - \ell_{P_1}} \simeq I_P^{P_1}(\sigma_{\ell - \ell_P})$.
\item For parabolic subgroups $Q\subset Q_1$ between $P$ and $P(\sigma)$, we have $\St^{Q_1}_Q(\sigma_{\ell_{Q_1} - \ell_P})_{\ell - \ell_{Q_1}} \simeq \St^{Q_1}_Q(\sigma_{\ell - \ell_P})$.
\end{enumerate}
\end{lem}

\begin{proof}
(1)
Define a linear map $f\colon \mathcal{H}_{P_1}\to \mathcal{H}_{P_1}$ by $T_w^{P_1}\mapsto (-1)^{\ell(w) - \ell_{P_1}(w)}T_w^{P_1}$.
Then by Remark~\ref{rem:ell - ell_P is algebra hom}, $f$ is an algebra homomorphism.
Put $\varphi^f = \varphi\circ f$ for $\varphi\in I_P^{P_1}(\sigma_{\ell_{P_1} - \ell_P})_{\ell - \ell_{P_1}}$.
Then for $X\in \mathcal{H}_{P_1}$ and $w\in W_{P}^{P_1-}(1)$, we have 
\begin{align*}
\varphi^f(Xj_P^{P_1-*}(T_w^{P*})) & = \varphi^f(XT_w^{P_1*})\\
& = \varphi(f(X)f(T_w^{P_1*}))\\
& = (-1)^{\ell(w) - \ell_{P_1}(w)}\varphi(f(X)T_w^{P_1*})\\
& = (-1)^{\ell(w) - \ell_{P_1}(w)}\varphi(f(X))\sigma_{\ell_{P_1} - \ell_P}(T_w^{P*})\\
& = (-1)^{\ell(w) - \ell_{P}(w)}\varphi(f(X))\sigma(T_w^{P*})\\
& = \varphi(f(X))\sigma_{\ell - \ell_P}(T_w^{P*})\\
& = \varphi^f(X)\sigma_{\ell - \ell_P}(T_w^{P*}).
\end{align*}
Hence $\varphi^f\in I^{P_1}_P(\sigma_{\ell - \ell_P})$.
For $w\in W_{P_1}(1)$ and $X\in \mathcal{H}_{P_1}$, we have 
\begin{align*}
(\varphi^f T^{P_1}_w)(X) & = \varphi^f(T^{P_1}_wX)\\
& = \varphi(f(T^{P_1}_w)f(X))\\
& = (-1)^{\ell(w) - \ell_{P_1}(w)}\varphi (T_w^{P_1}f(X))\\
& = (\varphi T_w^{P_1})(f(X)) = (\varphi T_w^{P_1})^f(X).
\end{align*}
Hence $\varphi\mapsto \varphi^f$ is an $\mathcal{H}$-module homomorphism, therefore an isomorphism $I_P^{P_1}(\sigma_{\ell_{P_1} - \ell_{P}})_{\ell -\ell_{P_1}}\to I^{P_1}_P(\sigma_{\ell -\ell_P})$.

(2)
First we prove the lemma for $Q = Q_1$, namely $\St_{Q}^{Q_1} = e_Q$.
Let $w\in W_P(1)$.
Then we have 
\begin{align*}
e_Q(\sigma_{\ell_Q - \ell_P})_{\ell - \ell_Q}(T_w^{Q*}) & = (-1)^{\ell(w) - \ell_Q(w)}e_Q(\sigma_{\ell_Q - \ell_P})(T_w^{Q*})\\
& = (-1)^{\ell(w) - \ell_Q(w)}\sigma_{\ell_Q - \ell_P}(T_w^{P*})\\
& = (-1)^{\ell(w) -\ell_Q(w)}(-1)^{\ell_Q(w) - \ell_P(w)}\sigma(T_w^{P*})\\
& = (-1)^{\ell(w) - \ell_P(w)}\sigma(T_w^{P*}) = \sigma_{\ell - \ell_P}(T_w^{P*}).
\end{align*}
If $w\in W_{Q\cap P_2,\aff}(1)$ then $T_w^{Q*}\in \mathcal{H}_{\aff,Q}$.
Hence by Lemma~\ref{lem:twist by length is trivial on H_aff}, we have $e_Q(\sigma_{\ell_Q - \ell_P})_{\ell - \ell_Q}(T_w^{Q*}) = e_Q(\sigma_{\ell_Q - \ell_P})(T_w^{Q*})$.
The definition of the extension says that it is $1$.
Therefore by a characterization of $e_Q(\sigma_{\ell - \ell_P})$, we have (2) in this case.

In general, consider the exact sequence
\[
\bigoplus_{Q_1\supset Q_2\supsetneq Q}I^{Q_1}_{Q_2}(e_{Q_2}(\sigma_{\ell_{Q_1} - \ell_P}))
\to I_Q^{Q_1}(e_Q(\sigma_{\ell_{Q_1} - \ell_P}))
\to \St_Q^{Q_1}(\sigma_{\ell_{Q_1} - \ell_P})\to 0.
\]
Hence 
\begin{multline*}
\bigoplus_{Q_1\supset Q_2\supsetneq Q}I^{Q_1}_{Q_2}(e_{Q_2}(\sigma_{\ell_{Q_1} - \ell_P}))_{\ell - \ell_{Q_1}}
\to I_Q^{Q_1}(e_Q(\sigma_{\ell_{Q_1} - \ell_P}))_{\ell - \ell_{Q_1}}\\
\to \St_Q^{Q_1}(\sigma_{\ell_{Q_1} - \ell_P})_{\ell - \ell_{Q_1}}\to 0.
\end{multline*}
Using (1) and (2) for $Q = Q_1$, for $Q_1\supset Q_2\supset Q$, we have 
\begin{align*}
I_{Q_2}^{Q_1}(e_{Q_2}(\sigma_{\ell_{Q_1} - \ell_P}))_{\ell - \ell_{Q_1}}
& = I_{Q_2}^{Q_1}(e_{Q_2}(\sigma_{\ell_{Q_2} - \ell_P})_{\ell_{Q_1} - \ell_{Q_2}})_{\ell - \ell_{Q_1}}\\
& = I_{Q_2}^{Q_1}(e_{Q_2}(\sigma_{\ell_{Q_2} - \ell_P})_{\ell - \ell_{Q_2}})\\
& = I_{Q_2}^{Q_1}(e_{Q_2}(\sigma_{\ell - \ell_P})).
\end{align*}
Therefore
\[
\bigoplus_{Q_1\supset Q_2\supsetneq Q}I^{Q_1}_{Q_2}(e_{Q_2}(\sigma_{\ell - \ell_P}))
\to I_Q^{Q_1}(e_Q(\sigma_{\ell - \ell_P}))
\to \St_Q^{Q_1}(\sigma_{\ell_{Q_1} - \ell_P})_{\ell - \ell_{Q_1}}\to 0.
\]
Hence we get the lemma.
\end{proof}

\subsection{The functor $I'_P$}
We define the functor $I_P'$ as follows.
\begin{defn}
For an $\mathcal{H}_P$-module $\sigma$, put
\[
I_P'(\sigma) = \Hom_{(\mathcal{H}_P^-,j_P^-)}(\mathcal{H},\sigma).
\]
\end{defn}
We remark the following proposition.
\begin{prop}
Let $P$ be a parabolic subgroup and $\sigma$ an $\mathcal{H}_P$-module.
Then the map $\varphi\mapsto \varphi\circ\iota$ induces an isomorphism $I_P(\sigma)^\iota \simeq I'_P(\sigma_{\ell - \ell_P}^{\iota_P})$.
\end{prop}
\begin{proof}
For $w\in W_P^-(1)$, we have 
\begin{align*}
(\varphi\circ\iota)(Xj_P^-(T_w^P)) & = (\varphi\circ\iota)(XT_w)\\
& = \varphi(\iota(X)\iota(T_w))\\
& = (-1)^{\ell(w)}\varphi(\iota(X)T_w^*)\\
& = (-1)^{\ell(w)}\varphi(\iota(X))\sigma(T^{P*}_w)\\
& = (-1)^{\ell(w) - \ell_P(w)}\varphi^\iota(X)\sigma^{\iota_P}(T^{P}_w).
\end{align*}
Hence $\varphi\circ\iota\in\Hom_{(\mathcal{H}_P^-,j_P^-)}(\mathcal{H},\sigma^{\iota_P}_{\ell - \ell_P})$.
By the same argument implies that $\psi\mapsto \psi\circ\iota$ gives a homomorphism $I'_P(\sigma_{\ell - \ell_P}^{\iota_P})\to I_P(\sigma)^\iota$ which is the inverse of the above homomorphism.
\end{proof}
From the properties of $I_P$, we get the properties of $I'_P$.
\begin{prop}\label{prop:fundamental properties of I'}
We have the following.
\begin{enumerate}
\item The functor $I'_P$ is exact.
\item The map $\varphi\mapsto (\varphi(T_{n_w}^*))$ gives an isomorphism $I'_P(\sigma)\simeq \bigoplus_{w\in W^P_0}\sigma$.
\item Let $Q$ be a parabolic subgroup containing $P$.
Then $I'_Q\circ I_P^{Q\prime}\simeq I_P'$ by the homomorphism $\varphi\mapsto (X\mapsto \varphi(X)(1))$.
\end{enumerate}
\end{prop}

\begin{proof}
(1) follows from the exactness of $I_P$.
Proposition~\ref{prop:decomposition of I_P} implies (2) and Proposition~\ref{prop:transitivity of inductions} implies (3).
\end{proof}

\subsection{Other inductions}
The reason why we introduce only one induction $I'_P$ is that the other two inductions are not new by the following proposition.
\begin{prop}\label{prop:comparison with two more inductions}
Put $P' = n_{w_Gw_P}\opposite{P}n_{w_Gw_P}^{-1}$.
\begin{enumerate}
\item The map $\varphi\mapsto (X\mapsto \varphi(XT_{n_{w_Gw_P}}))$ gives an isomorphism
\[
I_P'(\sigma) = \Hom_{(\mathcal{H}_P^-,j_P^-)}(\mathcal{H},\sigma)\xrightarrow{\sim}\Hom_{(\mathcal{H}_{P'}^+,j_{P'}^+)}(\mathcal{H},n_{w_Gw_P}\sigma).
\]
\item The map $\varphi\mapsto (X\mapsto \varphi(XT_{n_{w_Gw_P}}^*))$ gives an isomorphism
\[
I_P(\sigma) = \Hom_{(\mathcal{H}_P^-,j_P^{-*})}(\mathcal{H},\sigma)\xrightarrow{\sim}\Hom_{(\mathcal{H}_{P'}^+,j_{P'}^{+*})}(\mathcal{H},n_{w_Gw_P}\sigma).
\]
\end{enumerate}
\end{prop}
\begin{rem}
By $\varphi\mapsto \varphi\circ\iota$, we have
\[
\Hom_{(\mathcal{H}_P^+,j_P^{+})}(\mathcal{H},\sigma)^\iota
\simeq
\Hom_{(\mathcal{H}_P^+,j_P^{+*})}(\mathcal{H},\sigma^{\iota_P}_{\ell - \ell_P}).
\]
Hence the statement (2) follows from (1).
\end{rem}

First we check that the map is a homomorphism.
For the calculation, we need the following lemma.
Recall the notation ${}^PW_0$ from subsection \ref{subsec:Preliminaries, Parabolic induction}.
\begin{lem}\label{lem:E_ov = E_o}
Let $w\in {}^PW_0, v\in W_{0,P}$ and $\lambda\in Z(W_P(1))Z_\kappa$.
For $o = o_+$ or $o_-$, we have $E_{o\cdot v}(\lambda n_w) = E_o(\lambda n_w)$.
\end{lem}
\begin{proof}
We may assume $\lambda\in Z(W_P(1))$.
We prove the lemma in $\mathcal{H}[q_s^{\pm 1}]$.
First we assume $o = o_-$ and prove the lemma by induction on $\ell(w)$.
Assume that $w = 1$.
Take anti-dominant $\lambda_1,\lambda_2\in Z(W_P(1))$ such that $\lambda = \lambda_1\lambda_2^{-1}$.
Then by \eqref{eq:product formula}, we have $E_{o_-}(\lambda)E_{o_-}(\lambda_2) = q_{\lambda}^{1/2}q_{\lambda_2}^{1/2}q_{\lambda_1}^{-1/2}E_{o_-}(\lambda_1)$.
Since $\lambda_1,\lambda_2$ are anti-dominant, we have $E_{o_-}(\lambda_1) = T_{\lambda_1}$ and $E_{o_-}(\lambda_2) = T_{\lambda_2}$ by \eqref{eq:E_o for lambda}.
Hence $E_{o_-}(\lambda) = q_{\lambda}^{1/2}q_{\lambda_1}^{-1/2}q_{\lambda_2}^{1/2}T_{\lambda_1}T_{\lambda_2}^{-1}$.
Since $n_v^{-1}\cdot \lambda_1 = \lambda_1$ and $n_v^{-1}\cdot \lambda_2 = \lambda_2$  are both anti-dominant, we have $E_{o_-\cdot v}(\lambda) = q_{\lambda}^{1/2}q_{\lambda_1}^{-1/2}q_{\lambda_2}^{1/2}T_{\lambda_1}T_{\lambda_2}^{-1}$ by the same argument.
Hence we get the lemma in this case.

Assume that $\ell(w) > 0$ and take $s\in S_0$ such that $ws < w$.
Then by \cite[Lemma~3.1]{MR0435249}, we have $ws\in {}^PW_0$.
By the product formula \eqref{eq:product formula}, we have
\begin{align*}
E_{o_-}(\lambda n_w) & = q_{\lambda n_w}^{1/2}q_{\lambda n_{ws}}^{-1/2}q_{n_s}^{-1/2}E_{o_-}(\lambda n_{ws})E_{o_-\cdot ws}(n_s),\\
E_{o_-\cdot v}(\lambda n_w) & = q_{\lambda n_w}^{1/2}q_{\lambda n_{ws}}^{-1/2}q_{n_s}^{-1/2}E_{o_-\cdot v}(\lambda n_{ws})E_{o_-\cdot vws}(n_s).
\end{align*}
By inductive hypothesis, we have $E_{o_-}(\lambda n_{ws}) = E_{o_-\cdot v}(\lambda n_{ws})$.
Hence it is sufficient to prove that $E_{o_-\cdot ws}(n_s) = E_{o_-\cdot vws}(n_s)$.
Since $ws < w$, we have $E_{o_-\cdot ws}(n_s) = T_{n_s}$ by \eqref{eq:E_o for W_0}.
Take $\alpha\in\Delta$ such that $s = s_\alpha$.
Then $ws(\alpha) = -w(\alpha) > 0$.
If $vws(\alpha) < 0$, then $ws(\alpha)\in \Sigma_P^+$ since $v\in W_{0,P}$.
Since $w\in {}^PW_0$, we have $-\alpha = w^{-1}(ws(\alpha))\in w^{-1}(\Sigma_P^+)\subset \Sigma^+$.
This is a contradiction.
Hence $vws(\alpha) > 0$.
Therefore $vws < vw$.
By \eqref{eq:E_o for W_0}, we have $E_{o_-\cdot vws}(n_s) = T_{n_s}$.
We get $E_{o_-\cdot ws}(n_s) = E_{o_-\cdot vws}(n_s)$ and finish the inductive step.

Now we get $E_{o_-\cdot v}(\lambda n_w) = E_{o_-}(\lambda n_w)$.
Applying $\iota$ to both sides with \cite[Lemma~5.31]{MR3484112}, we get $(-1)^{\ell(\lambda n_w)}E_{o_+\cdot v}(\lambda n_w) = (-1)^{\ell(\lambda n_w)}E_{o_+}(\lambda n_w)$.
Hence $E_{o_+\cdot v}(\lambda n_w) = E_{o_+}(\lambda n_w)$.
\end{proof}

We start to prove Proposition~\ref{prop:comparison with two more inductions}.
\begin{lem}\label{lem:comparison with two more inductions, it is hom}
The map given in Proposition~\ref{prop:comparison with two more inductions} is an $\mathcal{H}$-module homomorphism.
\end{lem}
\begin{proof}
Put $n = n_{w_Gw_P}$.
The lemma is equivalent to that the map $\varphi\mapsto \varphi(T_n)$ from $I_P'(\sigma)$ to $n\sigma$ gives an $(\mathcal{H}_{P'}^+,j_{P'}^+)$-module homomorphism.
Let $w\in W_{P'}(1)$ be a $P'$-positive element.
We have
\[
(\varphi j_{P'}^+(E_{o_{+,P'}}^{P'}(w)))(T_n) 
= \varphi (j_{P'}^+(E_{o_{+,P'}}^{P'}(w))T_n).
\]
and, by Lemma~\ref{lem:image of E by j}, we have
\[
j_{P'}^+(E_{o_{+,P'}}^{P'}(w))
=
j_{P'}^+(E_{o_{-,P'}\cdot w_{P'}}^{P'}(w))
=
E_{o_-\cdot w_{P'}}(w).
\]
We have $w_Gw_P = w_{P'}w_G\in {}^{P'}W_0$.
Hence by Lemma~\ref{lem:E_ov = E_o}, we have $T_n = E_{o_-}(n) = E_{o_-\cdot n_{w_{P'}}w}(n)$.
By Lemma~\ref{lem:length, positive/negative elements}, we have $ \ell(w) + \ell(n) = \ell(wn) = \ell(nn^{-1}wn) = \ell(n) + \ell(n^{-1}wn)$.
Here we use that $n\in W^P_0$ and $n^{-1}wn\in W_P(1)$ is $P$-negative.
Therefore
\begin{align*}
E_{o_-\cdot w_{P'}}(w)T_n & = E_{o_-\cdot w_{P'}}(w)E_{o_-\cdot n_{w_{P'}}w}(n)\\
& = E_{o_-\cdot w_{P'}}(wn)\\
& = E_{o_-\cdot w_{P'}}(n)E_{o_-\cdot n_{w_{P'}}n}(n^{-1}wn)\\
& = T_nE_{o_-\cdot n_{w_{P'}}n}(n^{-1}wn).
\end{align*}
Since $n_{w_{P'}}n = n_{w_{P'}}n_{w_Gw_P} = n_{w_G}$, we have $o_-\cdot n_{w_P}n = o_-\cdot w_G = o_+$.
Hence $E_{o_-\cdot n_{w_{P'}}n}(n^{-1}wn) = E_{o_+}(n^{-1}wn) = j_P^-(E^P_{o_{+,P}}(n^{-1}wn))$ by Lemma~\ref{lem:image of E by j}.
Therefore we have
\begin{equation}\label{eq:ET_n=T_nE}
j_{P'}^+(E_{o_{+,P'}}^{P'}(w))T_n  = T_nj_P^-(E^P_{o_{+,P}}(n^{-1}wn)).
\end{equation}
Hence we get
\begin{align*}
\varphi (j_{P'}^+(E_{o_{+,P'}}^{P'}(w))T_n) & = \varphi(T_nj_P^-(E^P_{o_{+,P}}(n^{-1}wn)))\\
& = \varphi(T_n)\sigma(E^P_{o_{+,P}}(n^{-1}wn))\\
& = \varphi(T_n)(n\sigma)(E^{P'}_{o_{+,P'}}(w)).
\end{align*}
We get the lemma.
\end{proof}

We construct the homomorphism in the opposite direction.
\begin{lem}\label{lem:comparison of two more inductions, hom opposite direction}
Let $\lambda = \lambda_P^-\in Z(W_P(1))$ as in Proposition~\ref{prop:localization as Levi subalgebra}.
Put $n = n_{w_Gw_P}$ and $P' = n\opposite{P}n^{-1}$.
Then $\varphi\mapsto (X\mapsto \varphi(XE_{o_+}(\lambda n^{-1})))$ gives a homomorphism 
\[
\Hom_{(\mathcal{H}_{P'}^+,j_{P'}^+)}(\mathcal{H},n\sigma) \to \Hom_{(\mathcal{H}_P^-,j_P^-)}(\mathcal{H},\sigma) = I'_P(\sigma).
\]
\end{lem}
\begin{proof}
We prove that $\varphi\mapsto \varphi(E_{o_+}(\lambda n^{-1}))$ is an $(\mathcal{H}_P^-,j_P^-)$-homomorphism $\Hom_{(\mathcal{H}_{P'}^+,j_P^+)}(\mathcal{H},n\sigma) \to \sigma$.
Let $w\in W_P(1)$ be a $P$-negative element.
Then 
\[
(\varphi j_P^-(E_{o_{+,P}}^P(w)))(E_{o_+}(\lambda n^{-1})) = \varphi(j_P^-(E_{o_{+,P}}^P(w))E_{o_+}(\lambda n^{-1}))
\]
and $j_P^-(E_{o_{+,P}}^P(w)) = E_{o_+}(w)$ by Lemma~\ref{lem:image of E by j}.
Since $w\in W_P(1)$, $\lambda\in Z(W_P(1))$ and $(w_Gw_P)^{-1}\in {}^PW_0$, we have $E_{o_+}(\lambda n^{-1}) = E_{o_+\cdot w}(\lambda n^{-1})$ by Lemma~\ref{lem:E_ov = E_o}.
We also have, by Lemma~\ref{lem:length, positive/negative elements, strictly}, $\ell(w) + \ell(\lambda n^{-1}) = \ell(w\lambda n^{-1})$.
Hence 
\[
E_{o_+}(w)E_{o_+}(\lambda n^{-1})
=
E_{o_+}(w)E_{o_+\cdot w}(\lambda n^{-1})
=
E_{o_+}(w\lambda n^{-1}).
\]
Since $\lambda\in Z(W_P(1))$, we have $w\lambda n^{-1} = \lambda wn^{-1} = \lambda n^{-1}(nwn^{-1})$.
The element $nwn^{-1}$ is $P'$-positive.
By Lemma~\ref{lem:length, positive/negative elements, strictly}, $\ell(\lambda n^{-1}nwn^{-1}) = \ell(\lambda n^{-1}) + \ell(nwn^{-1})$. (We have $\lambda n^{-1} = n^{-1}(n\lambda n^{-1})$ and since $n(\Sigma^+\setminus\Sigma_P^+) = \Sigma^-\setminus\Sigma_{P'}^-$, $n\lambda n^{-1}$ satisfies the condition of $\lambda_{P'}^+$ in Proposition~\ref{prop:localization as Levi subalgebra}. We also have $(w_Gw_P)^{-1} = w_Pw_G = w_Gw_{P'}\in W_0^{P'}$.)
Hence we have $E_{o_+}(w\lambda n^{-1}) = E_{o_+}(\lambda n^{-1})E_{o_+\cdot n^{-1}}(nwn^{-1})$.
We have $o_+\cdot n^{-1} = o_+\cdot w_Pw_G = o_+\cdot w_Gw_{P'} = o_-\cdot w_{P'}$.
Hence we have 
\begin{align*}
E_{o_+\cdot n^{-1}}(nwn^{-1} ) &= E_{o_-\cdot w_{P'}}(nwn^{-1})\\
& = j_{P'}^+(E^{P'}_{o_{-,P'}\cdot w_{P'}}(nwn^{-1}))\\
& = j_{P'}^+(E^{P'}_{o_{+,P'}}(nwn^{-1}))
\end{align*}
by Lemma~\ref{lem:image of E by j}.
Therefore, we have
\begin{equation}\label{eq:EE(lambda n)=E(lambda n)E}
j_P^-(E_{o_{+,P}}^P(w))E_{o_+}(\lambda n^{-1}) = E_{o_+}(\lambda n^{-1})j_{P'}^+(E^{P'}_{o_{+,P'}}(nwn^{-1}))
\end{equation}
Therefore
\begin{align*}
(\varphi j_P^-(E_{o_{+,P}}^P(w)))(E_{o_+}(\lambda n^{-1}))
& = \varphi(j_P^-(E_{o_{+,P}}^P(w))(E_{o_+}(\lambda n^{-1})))\\
& = \varphi(E_{o_+}(\lambda n^{-1})j_{P'}^+(E^{P'}_{o_{+,P'}}(nwn^{-1})))\\
& = \varphi(E_{o_+}(\lambda n^{-1}))(n\sigma)(E^{P'}_{o_{+,P'}}(nwn^{-1}))\\
& = \varphi(E_{o_+}(\lambda n^{-1}))\sigma(E^{P}_{o_{+,P}}(w)).
\end{align*}
We get the lemma.
\end{proof}

\begin{proof}[Proof of Proposition~\ref{prop:comparison with two more inductions}]
We prove that the compositions of the homomorphisms in Proposition~\ref{prop:comparison with two more inductions} and Lemma~\ref{lem:comparison of two more inductions, hom opposite direction} are isomorphisms.
Let $\Phi$ be the homomorphism in Proposition~\ref{prop:comparison with two more inductions} and $\Psi$ that in Lemma~\ref{lem:comparison of two more inductions, hom opposite direction}.

Put $n = n_{w_Gw_P}$ and $P' = n\opposite{P}n^{-1}$.
For $\varphi\in \Hom_{(\mathcal{H}_{P'}^+,j_{P'}^+)}(\mathcal{H},n\sigma)$, $\Phi(\Psi(\varphi))$ is given by
\[
\Phi(\Psi(\varphi))(X) = \Psi(\varphi)(XT_n) = \varphi(XT_nE_{o_+}(\lambda n^{-1}))
\]
where $\lambda = \lambda_P^-$ as in Proposition~\ref{prop:localization as Levi subalgebra}.
We have $T_n  = E_{o_-}(n)$ by \eqref{eq:E_o for W_0}.
Since $w_Gw_P = w_{P'}w_G\in {}^{P'}W_0$, we have $E_{o_-}(n) = E_{o_-\cdot w_{P'}}(n)$ by Lemma~\ref{lem:E_ov = E_o}.
We have $o_-\cdot w_{P'}n = o_-\cdot w_G = o_+$.
Since $(w_Gw_P)^{-1} = w_Pw_G\in {}^PW_0$, by Lemma~\ref{lem:length, positive/negative elements, strictly} and \ref{lem:conjugate and length}, we have $\ell(\lambda n^{-1}) = \ell(\lambda) - \ell(n) = \ell(n\cdot \lambda) - \ell(n)$.
Hence, by Lemma~\ref{lem:image of E by j}, we have
\begin{equation}\label{eq:TnE(lambda n)}
\begin{split}
T_nE_{o_+}(\lambda n^{-1})
&=
E_{o_-\cdot w_{P'}}(n)E_{o_-\cdot w_{P'}n}(\lambda n^{-1})\\
&=
E_{o_-\cdot w_{P'}}(n\lambda n^{-1})\\
&=
j_{P'}^+(E^{P'}_{o_{-,P'}\cdot w_{P'}}(n\lambda n^{-1}))\\
&=
j_{P'}^+(E^{P'}_{o_{+,P'}}(n\lambda n^{-1})).
\end{split}\end{equation}
Therefore, we have 
\begin{align*}
\Phi(\Psi(\varphi))(X)
&=
\varphi(Xj^+_{P'}(E^{P'}_{o_{+,P'}}(n\lambda n^{-1})))\\
&=
\varphi(X)(n\sigma)(E^{P'}_{o_{+,P'}}(n\lambda n^{-1}))\\
&=
\varphi(X)\sigma(E^{P}_{o_{+,P}}(\lambda)).
\end{align*}
Since $\lambda$ is in $Z(W_P(1))$, $\sigma(E^{P}_{o_{+,P}}(\lambda))$ is invertible.
Hence $\Phi\circ\Psi$ is invertible.

Next, for $\psi\in I'_P(\sigma) = \Hom_{(\mathcal{H}_P^-,j_P^-)}(\mathcal{H},\sigma)$, we have
\[
\Psi(\Phi(\psi))(X) = \psi(XE_{o_+}(\lambda n^{-1})T_n).
\]
As in the above argument, we have $T_n = E_{o_-\cdot w_{P'}}(n) = E_{o_+\cdot n^{-1}}(n)$.
We also have $\ell(\lambda) = \ell(\lambda n^{-1}) + \ell(n)$ as in the above.
We have
\begin{equation}\label{eq:E(lambda )Tn}
E_{o_+}(\lambda n^{-1})T_n
=
E_{o_+}(\lambda n^{-1})E_{o_+\cdot n^{-1}}(n)
=
E_{o_+}(\lambda)
=
j_P^-(E_{o_{+,P}}^P(\lambda)).
\end{equation}
Therefore we have 
\[
\Psi(\Phi(\psi))(X) = \psi(Xj_P^-(E_{o_{+,P}}^P(\lambda))) = \psi(X)\sigma(E^P_{o_{+,P}}(\lambda)).
\]
Since $\lambda$ is in $Z(W_P(1))$, $\sigma(E^{P}_{o_{+,P}}(\lambda))$ is invertible.
Hence $\Psi\circ\Phi$ is invertible.
\end{proof}

\subsection{Tensor products}
Recall that we have
\[
I_P(\sigma)\simeq n_{w_Gw_P}\sigma\otimes_{(\mathcal{H}_{P'}^+,j_{P'}^+)}\mathcal{H}.
\]
where $P' = n_{w_Gw_P}\opposite{P}n_{w_Gw_P}^{-1}$ by Proposition~\ref{prop:tensor description of I_P}.
Again, we can consider the four inductions defined via the tensor product.
By $x\otimes X\mapsto x\otimes \iota(X)$, we have
\[
(\sigma\otimes_{(\mathcal{H}_P^+,j_P^+)}\mathcal{H})^\iota
\simeq
\sigma_{\ell - \ell_P}^{\iota_P}\otimes_{(\mathcal{H}_P^+,j_P^{+*})}\mathcal{H}
\]
and
\[
(\sigma\otimes_{(\mathcal{H}_P^-,j_P^-)}\mathcal{H})^\iota
\simeq
\sigma_{\ell - \ell_P}^{\iota_P}\otimes_{(\mathcal{H}_P^-,j_P^{-*})}\mathcal{H}.
\]
\begin{prop}\label{prop:comparison of tensor inductions}
Let $P$ be a parabolic subgroup and $\sigma$ an $\mathcal{H}_P$-module.
Put $P' = n_{w_Gw_P}\opposite{P}n_{w_Gw_P}^{-1}$.
\begin{enumerate}
\item The map $x\otimes X\mapsto x\otimes T_{n_{w_Gw_P}}X$ gives an isomorphism
\[
\sigma\otimes_{(\mathcal{H}^-_P,j_P^-)}\mathcal{H}\to n_{w_Gw_P}\sigma\otimes_{(\mathcal{H}^+_{P'},j_{P'}^+)}\mathcal{H}.
\]
\item The map $x\otimes X\mapsto x\otimes T_{n_{w_Gw_P}}^*X$ gives an isomorphism
\[
\sigma\otimes_{(\mathcal{H}^-_P,j_P^{-*})}\mathcal{H}\to n_{w_Gw_P}\sigma\otimes_{(\mathcal{H}^+_{P'},j_{P'}^{+*})}\mathcal{H}.
\]
\end{enumerate}
\end{prop}
\begin{proof}
(2) follows from (1).
We prove (1).

Put $n = n_{w_Gw_P}$.
Let $\Phi$ be a homomorphism given in the proposition and first we prove that this is a well-defined $\mathcal{H}$-homomorphism.
We prove that the linear map $x\mapsto x\otimes T_n$ is an $(\mathcal{H}_P^-,j_P^-)$-homomorphism $\sigma\to n\sigma\otimes_{(\mathcal{H}^+_{P'},j_{P'}^+)}\mathcal{H}$.
Let $x\in \sigma$ and $w\in W_P^-(1)$.
Then $nwn^{-1}\in W_{P'}^+(1)$.
Hence by \eqref{eq:ET_n=T_nE}, we have
\[
j_{P'}^+(E_{o_{+,P'}}^{P'}(nwn^{-1}))T_n  = T_nj_P^-(E^P_{o_{+,P}}(w))
\]
Hence, in $n\sigma\otimes_{(\mathcal{H}^+_{P'},j_{P'}^+)}\mathcal{H}$, we have
\begin{align*}
x\otimes T_nj_P^-(E^P_{o_{+,P}}(w))
& =
x\otimes j_{P'}^+(E_{o_{+,P'}}^{P'}(nwn^{-1}))T_n\\
& =
x(n\sigma)(E_{o_{+,P'}}^{P'}(nwn^{-1}))\otimes T_n\\
& =
x\sigma(E_{o_{+,P}}^{P}(w))\otimes T_n.
\end{align*}
Therefore $\Phi$ is an $\mathcal{H}$-module homomorphism.

Next let $\lambda = \lambda_P^-$ and consider the linear map $\Psi\colon x\otimes X\mapsto x\otimes E_{o_+}(\lambda n^{-1})X$.
We prove that $\Psi$ is also a well-defined $\mathcal{H}$-homomorphism $n\sigma\otimes_{(\mathcal{H}_{P'}^+,j_{P'}^+)}\mathcal{H}\to \sigma\otimes_{(\mathcal{H}_P^-,j_P^-)}\mathcal{H}$.
Let $w\in W_{P'}^+(1)$ and $x\in\sigma$.
Then $n^{-1}wn\in W_P^-(1)$.
By \eqref{eq:EE(lambda n)=E(lambda n)E}, we have
\[
j_P^-(E_{o_{+,P}}^P(n^{-1}wn))E_{o_+}(\lambda n^{-1}) = E_{o_+}(\lambda n^{-1})j_{P'}^+(E^{P'}_{o_{+,P'}}(w)).
\]
Hence, in $\sigma\otimes_{(\mathcal{H}^-_P,j_P^-)}\mathcal{H}$, we have
\begin{align*}
x\otimes E_{o_+}(\lambda n^{-1})j_{P'}^+(E^{P'}_{o_{+,P'}}(w))
& = 
x\otimes j_P^-(E_{o_{+,P}}^P(n^{-1}wn))E_{o_+}(\lambda n^{-1})\\
& = 
x\sigma(E_{o_{+,P}}^P(n^{-1}wn))\otimes E_{o_+}(\lambda n^{-1})\\
& = 
x(n\sigma)(E_{o_{+,P'}}^{P'}(w))\otimes E_{o_+}(\lambda n^{-1}).
\end{align*}
Therefore $\Psi$ is an $\mathcal{H}$-homomorphism.

Let $x\in \sigma$ and $X\in \mathcal{H}$.
By \eqref{eq:TnE(lambda n)}, we have
\begin{align*}
\Phi(\Psi(x\otimes X)) & = x\otimes T_nE_{o_+}(\lambda n^{-1})X\\
& = x\otimes j_{P'}^+(E_{o_{+,P'}}^{P'}(n\lambda n^{-1}))X\\
& = x(n\sigma)(E_{o_{+,P'}}^{P'}(n\lambda n^{-1}))\otimes X\\
& = x\sigma(E^P_{o_{+,P}}(\lambda))\otimes X.
\end{align*}
Since $\lambda\in Z(W_P(1))$, $\sigma(E^P_{o_{+,P}}(\lambda))$ is invertible.
Hence $\Phi\circ\Psi$ is invertible.
By \eqref{eq:E(lambda )Tn}, we also have
\begin{align*}
\Psi(\Phi(x\otimes X)) & = x\otimes E_{o_+}(\lambda n^{-1})T_nX\\
& = x\otimes j_P^-(E^P_{o_{+,P}}(\lambda))X\\
& = x\sigma(E_{o_{+,P}}^P(\lambda))\otimes X.
\end{align*}
This is again invertible.
\end{proof}

\begin{cor}\label{cor:parabolic induction via tensor product 2}
Let $P$ be a parabolic subgroup and $\sigma$ an $\mathcal{H}_P$-module.
Then we have
\begin{align*}
I_P(\sigma) & \simeq \sigma\otimes_{(\mathcal{H}_P^-,j_P^-)}\mathcal{H},\\
I'_P(\sigma) & \simeq \sigma\otimes_{(\mathcal{H}_P^-,j_P^{-*})}\mathcal{H}.
\end{align*}
\end{cor}
\begin{proof}
The first one follows from the Propositions~\ref{prop:tensor description of I_P} and \ref{prop:comparison of tensor inductions} and the second one follows from the twist of the first one.
\end{proof}

\begin{prop}
Let $P$ be a parabolic subgroup and $\sigma$ an $\mathcal{H}_P$-module.
\begin{enumerate}
\item The isomorphism $\sigma\otimes_{(\mathcal{H}_P^-,j_P^-)}\mathcal{H}\to I_P(\sigma)$ is given by the following.
For $x\in \sigma$, let $\varphi_x\in I_P(\sigma)$ be an element such that $\varphi_x(1) = x$ and $\varphi_x(T_{n_{w}}^*) = 0$ for any $w\in W_0^P\setminus\{1\}$.
Then the isomorphism is given by $x\otimes X\mapsto \varphi_x X$.
\item The isomorphism $\sigma\otimes_{(\mathcal{H}_P^-,j_P^{-*})}\mathcal{H}\to I'_P(\sigma)$ is given by the following.
For $x\in \sigma$, let $\varphi_x\in I_P(\sigma)$ be an element such that $\varphi_x(1) = x$ and $\varphi_x(T_{n_{w}}) = 0$ for any $w\in W_0^P\setminus\{1\}$.
Then the isomorphism is given by $x\otimes X\mapsto \varphi_x X$.
\end{enumerate}
In particular, these isomorphisms do not depend on a choice of a lift $n_{w_Gw_P}$.
\end{prop}
\begin{proof}
The second statement follows from the first one.
From the construction, the image of $x\in \sigma$ under the isomorphism in Corollary~\ref{cor:parabolic induction via tensor product 2} is given by $\psi_xT_{n_{w_Gw_P}}$ where $\psi_x$ is characterized by $\psi_x(T_{n_{w_Gw_P}}) = x$ and $\psi_x(T_{n_w}) = 0$ for any $w\in W^P_0\setminus\{w_Gw_P\}$.
We prove $\psi_xT_{n_{w_Gw_P}} = \varphi_x$.

Set $\psi = \psi_xT_{n_{w_Gw_P}}$.
We have $\psi(1) = \psi_x(T_{n_{w_Gw_P}}) = x$.
If $w\in W_0^P\setminus\{1\}$, then 
\begin{align*}
T_{n_{w_Gw_P}}T_{n_w}^* & = E_{o_+\cdot (w_Gw_P)^{-1}}(n_{w_Gw_P})E_{o_+}(n_w)\\
& = q_{w_Gw_P}^{1/2}q_{w}^{1/2}q_{w_Gw_Pw}^{-1/2}E_{o_+\cdot (w_Gw_P)^{-1}}(n_{w_Gw_P}n_w)\\
& \in \sum_{v\in W_0,v\le w_Gw_Pw}C[Z_\kappa]T_{n_v}.
\end{align*}
Hence it is sufficient to prove that if $v\le w_Gw_Pw$ then $\psi_x(T_{n_v}) = 0$.
Since $w\notin W_{P,0}$, we have $w_Gw_Pw\notin w_Gw_PW_{P,0}$.
Hence by \cite[Lemma~4.13 (3)]{arXiv:1406.1003_accepted}, we have $v\notin w_Gw_PW_{P,0}$.
Take $v_1\in W_0^P$ and $v_2\in W_{P,0}$ such that $v = v_1v_2$.
Then we have $\psi_x(T_{n_v}) = \psi_x(T_{n_{v_1}}T_{n_{v_2}}) = \psi_x(T_{n_{v_1}})\sigma(T^P_{n_{v_2}})$.
Since $v_1\ne w_Gw_P$, this is zero.
\end{proof}

\section{Adjoint functors}\label{sec:Adjoint functors}
\subsection{Adjoint functors $L_P$ and $R_P$}
Let $P$ be a parabolic subgroup.
By the definition of the parabolic induction $I_P$, it has the left adjoint functor $L_P$.
The functor $L_P$ is defined by
\[
L_P(\pi) = \pi\otimes_{(\mathcal{H}_P^-,j_P^{-*})}\mathcal{H}_P = \pi E_{o_-}(\lambda_P^-)^{-1}
\]
where $\lambda_P^-$ is as in Proposition~\ref{prop:localization as Levi subalgebra}.
Since this is a localization, this functor is exact.

By Proposition~\ref{prop:tensor description of I_P}, the functor $I_P$ also has the right adjoint functor $R_P$.
It is defined as follows.
Set $P' = n_{w_Gw_P}\opposite{P}n_{w_Gw_P}^{-1}$.
Then we have
\[
R_P(\pi) = n_{w_Gw_P}^{-1}\Hom_{(\mathcal{H}_{P'}^+,j_{P'}^+)}(\mathcal{H}_{P'},\pi).
\]
This is left exact.
Let $\lambda_{P'}^+$ be as in Proposition~\ref{prop:localization as Levi subalgebra}.
By Proposition~\ref{prop:localization as Levi subalgebra}, we have $\mathcal{H}_{P'} = \mathcal{H}_{P'}^+(T^{P'}_{\lambda_{P'}^+})^{-1}$.
Hence $\varphi\mapsto (\varphi((T^{P'}_{\lambda_{P'}^+})^{-n}))$ gives an isomorphism
\[
R_P(\pi) \simeq\{(x_n)_{n\in\Z_{\ge 1}}\mid x_n\in \pi,\ x_{n + 1}T_{\lambda_{P'}^+} = x_n\}.
\]

Let $P_1$ be a parabolic subgroup containing $P$.
Then the left adjoint functor (resp.\ the right adjoint functor) of $I_P^{P_1}$ is denoted by $L_P^{P_1}$ (resp.\ $R_P^{P_1}$).
\subsection{Parabolic inductions and adjoint functors}
In this subsection, we prove the following proposition.
The condition on $\sigma$ is found in the study in \cite{AHHV2}.
\begin{prop}\label{prop:parabolic induction and its adjoint}
Let $P,Q$ be parabolic subgroups and $\sigma$ an $\mathcal{H}_Q$-module.
Assume that $\bigcap_{n\in\Z_{\ge 0}}p^n\sigma = 0$.
Then we have $R_P\circ I_Q(\sigma)\simeq I_{P\cap Q}^P\circ R_{P\cap Q}^Q(\sigma)$.
\end{prop}
Before proving the proposition, we reformulate the proposition in terms of
\begin{align*}
\widetilde{I}_Q(\sigma) & = \sigma\otimes_{(\mathcal{H}_Q^+,j_Q^+)}\mathcal{H},\\
\widetilde{R}_P(\pi) & = \Hom_{(\mathcal{H}_P^+,j_P^+)}(\mathcal{H}_P,\pi).
\end{align*}
By Proposition~\ref{prop:tensor description of I_P}, we have
\[
I_Q(\sigma) = \widetilde{I}_{Q'}(n_{w_Gw_Q}\sigma),\quad R_P(\pi) = n_{w_Gw_P}^{-1}\widetilde{R}_{P'}(\pi)
\]
where $P' = n_{w_Gw_P}\opposite{P} n_{w_Gw_P}^{-1}$ and $Q' = n_{w_Gw_Q}\opposite{Q} n_{w_Gw_Q}^{-1}$.

\begin{lem}\label{lem:adjoint and induction, tilde version}
Let $P,Q$ be a parabolic subgroup and $\sigma$ an $\mathcal{H}_Q$-module.
Assume that $\bigcap_{n\in\Z_{\ge 0}}p^n\sigma = 0$.
Then we have $\widetilde{R}_P\circ \widetilde{I}_Q(\sigma)\simeq \widetilde{I}_{P\cap Q}^P\circ \widetilde{R}_{P\cap Q}^Q(\sigma)$.
\end{lem}
We prove that Lemma~\ref{lem:adjoint and induction, tilde version} implies Proposition~\ref{prop:parabolic induction and its adjoint}.
We have
\[
R_P\circ I_Q = n_{w_Gw_P}^{-1}\widetilde{R}_{P'}\circ\widetilde{I}_{Q'}n_{w_Gw_Q}.
\]
By Lemma~\ref{lem:adjoint and induction, tilde version}, we have $\widetilde{R}_{P'}\circ\widetilde{I}_{Q'} = \widetilde{I}_{P'\cap Q'}^{P'}\circ \widetilde{R}_{P'\cap Q'}^{Q'}$.
Let $P_1$ be a parabolic subgroup such that $\Delta_{P_1} = w_P(-\Delta_{P\cap Q}) = w_Pw_{P\cap Q}(\Delta_{P\cap Q})$.
Then we have $w_Gw_P(\Delta_{P_1}) = w_G(-\Delta_{P\cap Q}) = w_G(-\Delta_P)\cap w_G(-\Delta_Q) = \Delta_{P'}\cap \Delta_{Q'} = \Delta_{P'\cap Q'}$.
The adjoint action of $n_{w_Gw_P}$ induces an isomorphism $\mathcal{H}_{P}\simeq \mathcal{H}_{P'}$ and it induces an isomorphism $\mathcal{H}_{P_1}^{P+}$ to $\mathcal{H}_{P'\cap Q'}^{P'+}$.
Hence we get $n_{w_Gw_P}^{-1}\widetilde{R}_{P'\cap Q'}^{P'} = \widetilde{R}_{P_1}^Pn_{w_Gw_P}^{-1}$.
Similarly we have $\widetilde{I}_{P'\cap Q'}^{Q'}n_{w_Gw_Q} = n_{w_Gw_Q}\widetilde{I}_{P_2}^{Q}$ where $P_2$ is a parabolic subgroup corresponding to $w_Q(-\Delta_{P\cap Q})$.
Therefore we have
\[
R_P\circ I_Q = \widetilde{I}_{P_1}^Pn_{w_Gw_P}^{-1}n_{w_Gw_Q}\widetilde{R}_{P_2}^Q.
\]
Since we have $n_{w_Gw_P}n_{w_Pw_{P\cap Q}} = n_{w_Gw_{P\cap Q}} = n_{w_Gw_Q}n_{w_Qw_{P\cap Q}}$, we have $n_{w_Gw_P}^{-1}n_{w_Gw_Q} = n_{w_Pw_{P\cap Q}}n_{w_Qw_{P\cap Q}}^{-1}$.
Since $w_Pw_{P\cap Q}(\Delta_{P\cap Q}) = \Delta_{P_1}$, we have $\widetilde{I}_{P_1}^Pn_{w_Pw_{P\cap Q}} = I_{P\cap Q}^P$.
Similarly we have $n_{w_Qq_{P\cap Q}}^{-1}\widetilde{R}_{P_2}^Q = R_{P\cap Q}^Q$.
Hence we get Proposition~\ref{prop:parabolic induction and its adjoint}.

In the rest of this section, we prove Lemma~\ref{lem:adjoint and induction, tilde version}.
Recall a filtration introduced in subsection~\ref{subsec:A filtration}.
Let $A\subset {}^QW_0$ be a closed subset and fix a maximal element $w\in A$.
Set $A' = A\setminus\{w\}$.
Then the quotient
\[
\left(\sum_{v\in A}\sigma\otimes T_{n_v}^*\right)
/
\left(\sum_{v\in A'}\sigma\otimes T_{n_v}^*\right)
\]
is isomorphic to $\sigma$ as vector spaces and the action of $E_{o_-}(\lambda)$ is given by $\widetilde{q}(Q,n_w\cdot \lambda)\sigma(E^Q_{o_{-,Q}}(n_w\cdot \lambda))$ by Lemma~\ref{lem:Bruhat filtration via tensor product} where $\widetilde{q}(Q,\mu) = q(Q',n_{w_Gw_Q}\cdot \mu)$.
By Remark~\ref{rem:q(P,*)=0 iff P'-positive}, $\widetilde{q}(Q,\mu)\ne 1$ if and only if $\mu$ is $Q$-positive.
\begin{lem}\label{lem:condition on q != 1, positive version}
Let $\lambda_P^+$ be as in Proposition~\ref{prop:localization as Levi subalgebra}.
Then $n_w\cdot \lambda_P^+$ is $Q$-positive if and only if $w\in {}^QW_0\cap W_{0,P} = {}^{P\cap Q}W_{0,P}$.
\end{lem}
\begin{proof}
The element $n_w\cdot \lambda_P^+$ is $Q$-positive if and only if $\langle \alpha,w\nu(\lambda_P^+)\rangle \le 0$ for any $\alpha\in\Sigma^+\setminus\Sigma_Q^+$.
Since $\langle \beta,\nu(\lambda_P^+)\rangle \le 0$ if and only if $\beta\in\Sigma^+\cup \Sigma_P$, the element $n_w\cdot \lambda_P^+$ is $Q$-positive if and only if $w^{-1}(\Sigma^+\setminus\Sigma_Q^+)\subset \Sigma^+\cup \Sigma_P$.
Since $w\in {}^QW_0$, we have $w^{-1}(\Sigma_Q^+)\subset \Sigma^+$.
Hence this is equivalent to $w^{-1}(\Sigma^+)\subset \Sigma^+\cup \Sigma_P$.
Therefore $w^{-1}(\Sigma^-)\supset \Sigma^-\setminus\Sigma^-_P$ by taking the complement of the both sides.
Hence $w(\Sigma^-\setminus\Sigma^-_P)\subset \Sigma^-$.
Therefore we have $w\in W_{0,P}$.

For the last part, ${}^QW_0\cap W_{0,P} \subset {}^{P\cap Q}W_{0,P}$ is obvious.
If $w\in W_{0,P}$, then $w^{-1}(\Delta\setminus\Delta_P) \subset\Sigma^+$.
Hence if $w\in {}^{P\cap Q}W_{0,P}$, we have $w^{-1}((\Delta\setminus\Delta_P)\cup \Delta_{P\cap Q}) \subset\Sigma^+$.
We have $\Delta_Q\subset (\Delta\setminus\Delta_P)\cup \Delta_{P\cap Q}$.
Hence $w\in {}^QW_0$.
\end{proof}

Put $I = \sum_{v\in {}^{P\cap Q}W_{0,P}}\sigma\otimes T_{n_v}^*\subset \widetilde{I}_Q(\sigma)$.
\begin{lem}
$I = \sum_{v\in W_{0,P}}\sigma \otimes T_{n_v}^* = \sum_{v\in W_{0,P}}\sigma \otimes T_{n_v}$.
\end{lem}
\begin{proof}
For $v\in W_{0,P}$, Take $v_1\in W_{0,P\cap Q},v_2\in {}^{P\cap Q}W_{0,P}$ such that $v = v_1v_2$.
Then for any $x\in \sigma$ we have $x\otimes T_{n_v}^* = xT^{Q*}_{n_{v_1}}\otimes T_{n_{v_2}}^*$ since $j_Q^+(T_{n_{v_2}}^{Q*}) = T_{n_{v_2}}^{*}$ by Corollary~\ref{cor:image by j, positie and negative case}.
This gives the first equality.
The second equality follows from a usual triangular argument.
\end{proof}

\begin{lem}
The subspace $I$ is stable under the action of $E_{o_-}(w)$ where $w\in W_P(1)$.
In particular, $I$ is stable under the action of $j_P^+(\mathcal{H}_P^+)$.
\end{lem}
\begin{proof}
We prove $(\sigma\otimes T_{n_v})E_{o_-}(w)\subset I$.
Take $a\in W_{0,P}$ and $\lambda\in \Lambda(1)$ such that $w = \lambda n_a$.
Then by the Bernstein relations~\cite[Corollary~5.43]{MR3484112}, in $\mathcal{H}[q_s^{\pm 1}]$, we have 
\begin{align*}
T_{n_v}E_{o_-}(w) & \in (C[q_s^{\pm 1}] T_{n_v}E_{o_-}(\lambda)T_{n_a})\cap \mathcal{H}\\
& \subset \left(\sum_{b\le v,\mu\in \Lambda(1)}C[q_s^{\pm 1}]E_{o_-}(\mu)T_{n_b}T_{n_a}\right)\cap \mathcal{H}.
\end{align*}
Since $v\in W_{0,P}$, $b\le v$ implies $b\in W_{0,P}$.
We have $a\in W_{0,P}$.
Hence $T_{n_b}T_{n_a}\in \sum_{c\in W_{0,P}}C[Z_\kappa]T_{n_c}$.
Therefore
\begin{align*}
T_{n_v}E_{o_-}(w) & \in\left(\sum_{c\in W_{0,P},\mu\in \Lambda(1)}C[q_s^{\pm 1}]E_{o_-}(\mu)T_{n_c}\right)\cap \mathcal{H}\\
& = \left(\sum_{c\in W_{0,P},\mu\in \Lambda(1)}C[q_s^{\pm 1}]E_{o_-}(\mu n_c)\right)\cap \mathcal{H}\\
& = \sum_{c\in W_{0,P},\mu\in \Lambda(1)}C[q_s]E_{o_-}(\mu n_c).
\end{align*}
Hence it is sufficient to prove that $\sigma\otimes E_{o_-}(\mu n_c)\subset I$, namely we may assume $v = 1$.
Take $\lambda_Q^+$ as in Proposition~\ref{prop:localization as Levi subalgebra} such that $\lambda_Q^+\lambda$ is $Q$-positive.
Then for $x\in \sigma$, we have
\begin{align*}
x\otimes E_{o_-}(\lambda n_a)
& =
x\sigma(E_{o_{-,Q}}^Q(\lambda_Q^+))^{-1}\otimes E_{o_-}(\lambda_Q^+)E_{o_-}(\lambda n_a)\\
& \in 
Cx\sigma(E_{o_{-,Q}}^Q(\lambda_Q^+))^{-1}\otimes E_{o_-}(\lambda_Q^+\lambda n_a).
\end{align*}
Here we use the product formula \eqref{eq:product formula}.
Therefore we may assume that $\lambda$ is $Q$-positive.
Take $a_1\in W_{0,P\cap Q}$ and $a_2\in {}^{P\cap Q}W_{0,P}$ such that $a = a_1a_2$.
Then $\lambda n_{a_1}$ is $Q$-positive.
By Lemma~\ref{lem:condition for length is additive, P-negative}, we have $\ell(\lambda n_a) = \ell(\lambda n_{a_1}) + \ell(n_{a_2})$.
Hence by \eqref{eq:product formula}, we have $E_{o_-}(\lambda n_a) = E_{o_-}(\lambda n_{a_1})E_{o_-\cdot a_1}(n_{a_2})$.
By Lemma~\ref{lem:E_ov = E_o}, we have $E_{o_-\cdot a_1}(n_{a_2}) = E_{o_-}(n_{a_2})$.
Hence
\begin{align*}
x\otimes E_{o_-}(\lambda n_a) & = x\otimes E_{o_-}(\lambda n_{a_1}) E_{o_-}(n_{a_2})\\
& = x\sigma(E^Q_{o_{-,Q}}(\lambda n_{a_1}))\otimes  T_{n_{a_2}}.
\end{align*}
We get the lemma.
\end{proof}
Consider the exact sequence
\[
0\to I\to I_Q(\sigma)\to I_Q(\sigma)/I\to 0
\]
as $(\mathcal{H}_{P}^+,j_P^+)$-modules.
We have the exact sequence
\begin{multline*}
0\to \Hom_{(\mathcal{H}_P^+,j_P^+)}(\mathcal{H}_P,I)\to \Hom_{(\mathcal{H}_P^+,j_P^+)}(\mathcal{H}_P,I_Q(\sigma))\\\to \Hom_{(\mathcal{H}_P^+,j_P^+)}(\mathcal{H}_P,I_Q(\sigma)/I).
\end{multline*}

\begin{lem}
Assume that $\bigcap_n p^n\sigma = 0$.
Then $\Hom_{(\mathcal{H}_P^+,j_P^+)}(\mathcal{H}_P,I_Q(\sigma)/I) = 0$.
\end{lem}
\begin{proof}
Let $\lambda_P^+$ as in Proposition~\ref{prop:localization as Levi subalgebra}.
By Lemma~\ref{lem:image of E by j}, $j_P^+(E_{o_{-,P}}^P(\lambda_P^+)) = E_{o_-}(\lambda_P^+)$.
Hence by Proposition~\ref{prop:localization as Levi subalgebra}, we have
\[
\Hom_{(\mathcal{H}_P^+,j_P^+)}(\mathcal{H}_P,I_Q(\sigma)/I)
=
\Hom_{C[E_{o_-}(\lambda_P^+)]}(C[E_{o_-}(\lambda_P^+)^{\pm 1}],I_Q(\sigma)/I).
\]
Define a $C[E_{o_-}(\lambda_P^+)]$-module $\sigma_w$ by 
\[
\sigma_w(E_{o_-}(\lambda_P^+)) = \widetilde{q}(P,n_w\cdot \lambda)e_Q(\sigma)(E^Q_{o_{-,Q}}(n_w\cdot \lambda_P^+))
\]
on the same space as $\sigma$.
Then $I_Q(\sigma)/I$ has a $C[E_{o_-}(\lambda_P^+)]$-stable filtration whose subquotient is given by $\sigma_w$ where $w\in {}^QW_0\setminus {}^{P\cap Q}W_{0,P}$.
By Lemma~\ref{lem:condition on q != 1, positive version}, $\widetilde{q}(P,n_w\cdot \lambda)\ne 1$.
Hence it is a positive power of $p$.
Therefore the image of $\psi\in \Hom_{C[E_{o_-}(\lambda_P^+)]}(C[E_{o_-}(\lambda_P^+)^{\pm 1}],\sigma_w)$ is contained in $\bigcap_{n\ge 0}p^n\sigma_w$.
This is zero by the assumption.
\end{proof}

Hence to prove Lemma~\ref{lem:adjoint and induction, tilde version}, it is sufficient to prove 
\begin{equation}\label{eq:adjoint and induction, isom at cell}
\Hom_{(\mathcal{H}_P^+,j_P^+)}(\mathcal{H}_P,I)\simeq \Hom_{(\mathcal{H}_{P\cap Q}^{Q+},j_{P\cap Q}^{Q+})}(\mathcal{H}_{P\cap Q},\sigma)\otimes_{(\mathcal{H}_{P\cap Q}^{P+},j_{P\cap Q}^{P+})}\mathcal{H}_P.
\end{equation}
To construct a homomorphism from the right hand side of \eqref{eq:adjoint and induction, isom at cell} to the left hand side, it is sufficient to construct an $(\mathcal{H}_{P\cap Q}^{P+},j_{P\cap Q}^{P+})$-homomorphism $\Hom_{(\mathcal{H}_{P\cap Q}^{Q+},j_{P\cap Q}^{Q+})}(\mathcal{H}_{P\cap Q},\sigma)\to \Hom_{(\mathcal{H}_P^+,j_P^+)}(\mathcal{H}_P,I)$.
Let $\lambda_P^+$ as in Proposition~\ref{prop:localization as Levi subalgebra}.
By Lemma~\ref{lem:localization, middle positive/negative subalgebra}, we have
\[
\Hom_{(\mathcal{H}_P^+,j_P^+)}(\mathcal{H}_P,I)|_{(\mathcal{H}_{P\cap Q}^{P+},j_{P\cap Q}^{P+})}
\simeq
\Hom_{(\mathcal{H}_{P\cap Q}^{+},j_{P\cap Q}^{+})}(\mathcal{H}_{P\cap Q}^{P+},I).
\]
(Both sides are isomorphic to $\Hom_{C[E_{o_-}(\lambda_P^+)]}(C[E_{o_-}(\lambda_P^+)^{\pm 1}],I)$.)
Now the restriction to from $\mathcal{H}_{P\cap Q}$ to $\mathcal{H}_{P\cap Q}^{P+}$ gives a map $\Hom_{(\mathcal{H}_{P\cap Q}^{Q+},j_{P\cap Q}^{Q+})}(\mathcal{H}_{P\cap Q},\sigma)\to \Hom_{(\mathcal{H}_{P\cap Q}^{+},j_{P\cap Q}^{Q+})}(\mathcal{H}_{P\cap Q}^{P+},\sigma)$.
Combining the inclusion $\sigma = \sigma\otimes 1\hookrightarrow I$ which is $(\mathcal{H}^+_Q,j_Q^+)$-equivariant, we get 
\begin{multline*}
\Hom_{(\mathcal{H}_{P\cap Q}^{Q+},j_{P\cap Q}^{Q+})}(\mathcal{H}_{P\cap Q},\sigma)\to \Hom_{(\mathcal{H}_{P\cap Q}^{+},j_{P\cap Q}^{+})}(\mathcal{H}_{P\cap Q}^{P+},I)\\\simeq \Hom_{(\mathcal{H}_P^+,j_P^+)}(\mathcal{H}_P,I)
\end{multline*}
and it gives a homomorphism between \eqref{eq:adjoint and induction, isom at cell}.

We prove that this gives an isomorphism.
We have the decomposition
\[
I = \bigoplus_{v\in {}^{P\cap Q}W_{0,P}}\sigma\otimes T_{n_v}.
\]
We have $\langle \alpha,\nu(\lambda_P^+)\rangle = 0$ for any $\alpha\in\Delta_P$.
Hence the Bernstein relations~\cite[Corollary~5.43]{MR3484112} tells that $T_{n_v}E_{o_-}(\lambda_P^+) = E_{o_-}(n_v\cdot \lambda_P^+)T_{n_v} = E_{o_-}(\lambda_P^+)T_{n_v}$ for any $v\in W_{0,P}$.
(For the last part, recall that $\lambda_P^+$ is in the center of $W_P(1)$.)
Therefore each summand $\sigma\otimes T_{n_v}$ is stable under the action of $E_{o_-}(\lambda_P^+)$ and we have
\[
\Hom_{(\mathcal{H}_{P}^{+},j_{P}^{+})}(\mathcal{H}_{P},I)
\simeq
\bigoplus_{v\in {}^{P\cap Q}W_{0,P}}\Hom_{C[E_{o_-}(\lambda_P^+)]}(C[E_{o_-}(\lambda_P^+)^{\pm 1}],\sigma\otimes T_{n_v}).
\]
For $x\in \sigma$, from the above calculation we have $x\otimes T_{n_v}E_{o_-}(\lambda_P^+) = x\otimes E_{o_-}(\lambda_P^+)T_{n_v} = xE^Q_{o_{-,Q}}(\lambda_P^+)\otimes T_{n_v}$.
Hence each summand is isomorphic to $\Hom_{C[E^Q_{o_{-,Q}}(\lambda_P^+)]}(C[E^Q_{o_{-,Q}}(\lambda_P^+)^{\pm 1}],\sigma)$.
Therefore we have
\begin{equation}\label{eq:adjoint and parabolic, final form of RHS}
\Hom_{(\mathcal{H}_{P}^{+},j_{P}^{+})}(\mathcal{H}_{P},I)
\simeq
\bigoplus_{v\in {}^{P\cap Q}W_{0,P}}\Hom_{C[E^{Q}_{o_{-,Q}}(\lambda_P^+)]}(C[E^Q_{o_{-,Q}}(\lambda_P^+)^{\pm 1}],\sigma).
\end{equation}
On the other hand, the right hand side of \eqref{eq:adjoint and induction, isom at cell} is 
\begin{equation}\label{eq:adjoint and parabolic, final form of LHS}
\bigoplus_{v\in {}^{P\cap Q}W_{0,P\cap Q}}\Hom_{(\mathcal{H}_{P\cap Q}^{Q+},j_{P\cap Q}^{Q+})}(\mathcal{H}_{P\cap Q},\sigma)\otimes T_{n_v}.
\end{equation}
\begin{lem}\label{lem:two lambda^+'s}
The element $\lambda_P^+$ satisfies the condition of $\lambda_{P\cap Q}^{Q+}$ in Proposition~\ref{prop:localization as Levi subalgebra}.
\end{lem}
\begin{proof}
The element $\lambda_P^+\in Z(W_P(1))$ satisfies $\langle \alpha,\nu(\lambda_P^+)\rangle < 0$ for any $\alpha\in\Sigma^+\setminus\Sigma_P^+$.
Hence $\langle \alpha,\nu(\lambda_P^+)\rangle < 0$ for any $\alpha\in\Sigma_Q^+\setminus\Sigma_{P\cap Q}^+$ and $\lambda_P^+\in Z(W_{P\cap Q}(1))$.
\end{proof}
Hence we have $\mathcal{H}_{P\cap Q} = \mathcal{H}_{P\cap Q}^{Q+}E_{o_{-,P\cap Q}}^{P\cap Q}(\lambda_P^+)^{-1}$.
Note that, by Proposition~\ref{lem:image of E by j}, we have $j_{P\cap Q}^{Q+}(E_{o_{-,P\cap Q}}^{P\cap Q}(\lambda_P^+)) = E_{o_{-,Q}}^{Q}(\lambda_P^+)$.
Therefore \eqref{eq:adjoint and parabolic, final form of LHS} is equal to 
\[
\bigoplus_{v\in {}^{P\cap Q}W_{0,P\cap Q}}\Hom_{C[E_{o_{-,Q}}^Q(\lambda_P^+)]}(C[E_{o_{-,Q}}^Q(\lambda_P^+)^{\pm 1}],\sigma)\otimes T_{n_v}.
\]
This is isomorphic to \eqref{eq:adjoint and parabolic, final form of RHS}.
This ends the proof of Lemma~\ref{lem:adjoint and induction, tilde version}.

\begin{cor}\label{cor:left adjoint and parabolic induction}
Assume that $p$ is nilpotent in $C$.
Then we have $L_P\circ I_Q\simeq I_P^{P\cap Q}\circ L_{P\cap Q}^P$.
\end{cor}
\begin{proof}
If $p$ is nilpotent, then any module satisfies $\bigcap_n p^n\sigma = 0$.
Hence we have $R_P\circ I_Q\simeq I_{P\cap Q}^P\circ R_{P\cap Q}^P$.
Taking the left adjoint functors of the both sides, we get the corollary.
\end{proof}

\subsection{$L_P$ and Steinberg modules}
In this section, we fix a parabolic subgroup $P\subset G$ such that $\Delta_P$ and $\Delta\setminus \Delta_P$ are orthogonal to each other.
Let $\sigma$ be an $\mathcal{H}_P$-module which has the extension $e(\sigma)$ to $\mathcal{H}$.
Then we have an $\mathcal{H}$-module $\St_Q(\sigma)$ for $Q\supset P$.
Let $R$ be another parabolic subgroup.

\begin{lem}\label{lem:left adjoint of extension}
We have $L_R(e_G(\sigma)) \simeq e_R(L^P_{P\cap R}(\sigma))$.
\end{lem}
\begin{proof}
Let $\lambda_R^-$ as in Proposition~\ref{prop:localization as Levi subalgebra}.
Then by Lemma~\ref{lem:two lambda^+'s}, this satisfies the condition of $\lambda_{R\cap P}^{P-}$.
Hence we have $L_{P\cap R}^P(\sigma) = \sigma E^P_{o_{-,P}}(\lambda_R^-)^{-1}$.
By the definition, $\lambda_R^-$ is dominant, hence in particular, it is $P$-negative.
Hence $e_G(\sigma)(E_{o_-}(\lambda_R^-)) = e_G(\sigma)(j_P^{-*}(E_{o_{-,P}}^P(\lambda_R^-))) = \sigma(E_{o_{-,P}}^P(\lambda_R^-))$ by Lemma~\ref{lem:image of E by j}.
Therefore, as vector spaces, we have
\[
L_R(e_G(\sigma)) = e_G(\sigma)E_{o_-}(\lambda_R^-)^{-1} \simeq \sigma E_{o_{-,P}}^P(\lambda_R^-)^{-1} = L_{P\cap R}^P(\sigma).
\]
Namely, the linear map defined by $e_G(\sigma)\otimes_{(\mathcal{H}_R^-,j_R^{-*})}\mathcal{H}_R\ni x\otimes E^R_{o_{-,R}}(\lambda_R^-)^{-n}\mapsto x\otimes E^{P\cap R}_{o_{-,P\cap R}}(\lambda_R^-)^{-n}\in \sigma\otimes_{(\mathcal{H}_{P\cap R}^{P-},j_{P\cap R}^{P-*})}\mathcal{H}_{P\cap R}$ is an isomorphism.

We prove that this map is $(\mathcal{H}_{P\cap R}^{R-},j_{P\cap R}^{R-*})$-equivariant.
Let $w\in W_{P\cap R}^{R}(1)$ and take $k\in\Z_{\ge 0}$ such that $w(\lambda_R^-)^k$ is $R$-negative.
Since the length of $\lambda_R^-$ as an element of $W_R(1)$ is zero, we have $T_w^{R*}E^R_{o_{-,R}}(\lambda_R^-)^k = T_w^{R*}T_{(\lambda_R^-)^k}^{R*} = T_{w(\lambda_R^-)^k}^{R*}$.
We also have that $E^R_{o_{-,R}}(\lambda_R^-)$ is in the center of $\mathcal{H}_R^-$.
Replacing $R$ with $P\cap R$, we also have $T_w^{(P\cap R)*}E^{P\cap R}_{o_{-,P\cap R}}(\lambda_R^-)^k = T^{(P\cap R)*}_{w(\lambda_R^-)^k}$.
Hence for $x\in \sigma$, we have
\begin{align*}
e_G(\sigma)\otimes_{(\mathcal{H}_R^-,j_R^{-*})}\mathcal{H}_R& \ni
x\otimes E^R_{o_{-,R}}(\lambda_R^-)^{-n}T_{w}^{R*}\\
& = x\otimes T_{w}^{R*}E^R_{o_{-,R}}(\lambda_R^-)^{-n}\\
& = x\otimes T_{w(\lambda_R^-)^k}^{R*}E^R_{o_{-,R}}(\lambda_R^-)^{-(n + k)}\\
& = xe_G(\sigma)(T_{w(\lambda_R^-)^k}^{*})\otimes E^R_{o_{-,R}}(\lambda_R^-)^{-(n + k)}\\
& = x\sigma(T_{w(\lambda_R^-)^k}^{P*})\otimes E^R_{o_{-,R}}(\lambda_R^-)^{-(n + k)}\\
& \mapsto x\sigma(T_{w(\lambda_R^-)^k}^{P*})\otimes E^{P\cap R}_{o_{-,P\cap R}}(\lambda_R^-)^{-(n + k)}\\
& = x\otimes T_{w(\lambda_R^-)^k}^{(P\cap R)*}E^{P\cap R}_{o_{-,P\cap R}}(\lambda_R^-)^{-(n + k)}\\
& = x\otimes E^{P\cap R}_{o_{-,P\cap R}}(\lambda_R^-)^{-n}T_{w}^{(P\cap R)*}\\
& \in \sigma\otimes_{(\mathcal{H}_{P\cap R}^{P-},j_{P\cap R}^{P-*})}\mathcal{H}_{P\cap R}.
\end{align*}
Therefore the above homomorphism is $(\mathcal{H}_{P\cap R}^{R-},j_{P\cap R}^{R-*})$-equivariant.

Let $P_2$ be a parabolic subgroup corresponding to $\Delta\setminus\Delta_P$.
Fix $w\in W_{P_2\cap R,\aff}(1)$ and take $v\in W_0$ and $\lambda\in \Lambda(1)$ such that $w = n_v\lambda$.
Then we have $\nu(\lambda) \in \R(\Delta_{P_2}^\vee\cap \Delta_R^\vee)$.
In particular, $\langle \alpha,\nu(\lambda)\rangle = 0$ for any $\alpha\in\Delta_P$.
Take a dominant $\mu \in Z(W_R(1))\cap W_{P_2,\aff}(1)$ such that $\langle \alpha,\nu(\mu)\rangle = 0$ for any $\alpha\in \Delta_R\cup \Delta_P$ and $\langle \alpha,\nu(\mu)\rangle$ is sufficiently large for $\alpha\in\Delta\setminus(\Delta_R\cup \Delta_P) = \Delta_{P_2}\setminus\Delta_R$.
Then $\langle \alpha,\nu(\lambda\mu)\rangle = 0$ for any $\alpha\in\Delta_P$.
For $\alpha\in\Sigma_{P_2}^+\setminus\Sigma_R^+$, we can take $\mu$ such that $\langle \alpha,\nu(\lambda\mu)\rangle \ge 0$.
For such $\mu$, we have $\langle \alpha,\nu(\lambda\mu)\rangle\ge 0$ for any $\alpha\in\Sigma^+_P\cup (\Sigma^+_{P_2}\setminus\Sigma_R^+) = \Sigma^+\setminus\Sigma_R^+$.
Namely we can take $\mu$ such that $\lambda \mu$ is $R$-negative.
Since $w\mu$ is $R$-negative, for $x\in\sigma$, we have
\[
x\otimes E^R_{o_{-,R}}(\lambda_R^-)^{-n}T_{w\mu}^{R*} = xe_G(\sigma)(T_{w\mu}^{*})\otimes E_{o_{-,R}}^R(\lambda_R^-)^{-n}
\in e_G(\sigma)\otimes_{(\mathcal{H}_R^-,j_{R}^{-*})}\mathcal{H}_R.
\]
Since $w\mu \in W_{P_2,\aff}(1)$, $e_G(\sigma)(T_{w\mu}^{*}) = 1$.
Hence $L_R(e_G(\sigma))(T_{w\mu}^{R*}) = 1$.
In particular, by taking $w = 1$, we have $L_R(e_G(\sigma))(T_{\mu}^{R*}) = 1$.
Since $\mu \in Z(W_R(1))$, we have $T_{w\mu}^{R*} = T_{w}^{R*}T_{\mu}^{R*}$.
Therefore we get $L_R(e_G(\sigma))(T_w^{R*}) = 1$.
By the characterization of $e_R(L_{P\cap R}^P(\sigma))$, we get the lemma.
\end{proof}

\begin{prop}\label{prop:left adj of St}
Assume that $p$ is nilpotent in $C$.
We have
\[
	L_R(\St_Q(\sigma)) = 
	\begin{cases}
	\St^R_{Q\cap R}(L_{P\cap R}^P(\sigma)) & (\Delta_Q\cup \Delta_R = \Delta),\\
	0 & (\text{otherwise}).
	\end{cases}
\]
\end{prop}
\begin{proof}
We have an exact sequence
\[
	\bigoplus_{Q'\supsetneq Q}I_{Q'}(e_{Q'}(\sigma))\to I_Q(e_Q(\sigma))\to \St_Q(\sigma)\to 0.
\]
Since $L_R$ is exact, we have
\[
	\bigoplus_{Q'\supsetneq Q}L_R(I_{Q'}(e_{Q'}(\sigma)))\to L_R(I_Q(e_Q(\sigma)))\to L_R(\St_Q(\sigma))\to 0.
\]
By Corollary~\ref{cor:left adjoint and parabolic induction}, we have
\[
	\bigoplus_{Q'\supsetneq Q}I_{Q'\cap R}^R(L_{Q'\cap R}^{Q'}(e_{Q'}(\sigma)))\to I_{Q\cap R}^R(L_{Q\cap R}^Q(e_Q(\sigma)))\to L_R(\St_Q(\sigma))\to 0.
\]
By Lemma~\ref{lem:left adjoint of extension}, we have
\[
	\bigoplus_{Q'\supsetneq Q}I_{Q'\cap R}^R(e_{Q'\cap R}(L_{P\cap R}^P(\sigma))) \to I_{Q\cap R}^R(e_{Q\cap R}(L_{P\cap R}^P(\sigma)))\to L_R(\St_Q(\sigma))\to 0.
\]
Hence if there exists $Q'\supsetneq Q$ such that $Q'\cap R = Q\cap R$, then $L_R(\St_Q(\sigma)) = 0$.
Such $Q'$ exists if and only if $\Delta_Q\cup \Delta_R \ne \Delta$.
Therefore, if $\Delta_Q\cup \Delta_R \ne \Delta$, then $L_R(\St_Q(\sigma)) = 0$.
If $\Delta_Q\cup \Delta_R = \Delta$ then $\{Q'\cap R\mid Q'\supsetneq Q\} = \{Q''\mid R\supset Q''\supsetneq Q\cap R\}$.
Therefore we get $L_R(\St_Q(\sigma))\simeq \St^R_{Q\cap R}(L_{P\cap R}^P(\sigma))$.
\end{proof}

\subsection{$R_P$ and an exactness}
In the next subsection, we will prove the following proposition.
\begin{prop}\label{prop:right adj of St}
Let $P$ be a parabolic subgroup, $\sigma$ an $\mathcal{H}_P$-module which has the extension to $\mathcal{H}$ and $Q$ a parabolic subgroup containing $P$.
Assume that $\bigcap_{n\in\Z_{\ge 0}}p^n\sigma = 0$.
Then for any parabolic subgroup $R$, we have
\[
R_R(\St_Q(\sigma)) =
\begin{cases}
 \St^R_{Q\cap R}(R_{P\cap R}^P(\sigma)) & (\Delta_Q = \Delta_{Q\cap R}\cup \Delta_P),\\
 0 & (\text{otherwise}).
\end{cases}
\]
\end{prop}
The proof of this proposition is similar to that of Proposition~\ref{prop:left adj of St}.
However, to use that argument, we need the following lemma.
Note that this is not obvious since $R_R$ is not right exact.
\begin{lem}\label{lem:R preserves the exactness, definition of Steinberg}
Let $P$ be a parabolic subgroup, $\sigma$ an $\mathcal{H}_P$-module which has the extension to $\mathcal{H}$ and $Q$ a parabolic subgroup containing $P$.
Then for any parabolic subgroup $R$, the sequence
\[
\bigoplus_{Q_1\supsetneq Q}R_R(I_{Q_1}(e_{Q_1}(\sigma)))\to R_R(I_Q(e_Q(\sigma)))\to R_R(\St_Q(\sigma))\to 0
\]
is exact.
\end{lem}
We remark that we do not assume that $\bigcap_n p^n\sigma = 0$.
The aim of this subsection is to prove this lemma.

\begin{lem}
To prove Lemma~\ref{lem:R preserves the exactness, definition of Steinberg}, we may assume $\sigma = \trivrep$ and $R$ contains $P$.
\end{lem}
\begin{proof}
Put $R' = n_{w_Gw_R}\opposite{R}n_{w_Gw_R}^{-1}$.
Let $R_1'$ (resp.\ $R_2'$) be a parabolic subgroup corresponding to $\Delta_{R'}\cup \Delta_P$ (resp.\ $\Delta_{R'}\cup \Delta_{P_2}$).
Let $\lambda_1 = \lambda_{R_1'}^+$ and $\lambda_2 = \lambda_{R_2'}^+$ be as in Proposition~\ref{prop:localization as Levi subalgebra}.
Moreover we take $\lambda_1$ from $W_{P_2,\aff}(1)$ and $\lambda_2$ from $W_{P,\aff}(1)$.
Then $\lambda_1\lambda_2$ satisfies the condition of $\lambda_{R'}^+$.
We have
\[
R_R(\pi) = n_{w_Gw_R}^{-1}\Hom_{(\mathcal{H}_{R'}^+,j_{R'}^+)}(\mathcal{H}_{R'},\pi).
\]
Since $\mathcal{H}_{R'} = \mathcal{H}_{R'}(T_{\lambda_{R'}^+}^{R'})^{-1}$, we have
\[
R_R(\pi) \simeq \Hom_{C[T_{\lambda_{R'}^+}]}(C[(T_{\lambda_{R'}^+})^{\pm 1}],\pi)
\]
as vector spaces.
Therefore we have
\[
R_R(\pi) \simeq \Hom_{C[T_{\lambda_1}]}(C[(T_{\lambda_1})^{\pm 1}],\Hom_{C[T_{\lambda_2}]}(C[(T_{\lambda_2})^{\pm 1}],\pi)).
\]
Take $\pi = I_Q(e_Q(\sigma))$.
Then $I_Q(e_Q(\sigma))\simeq I_Q(\trivrep)\otimes e_G(\sigma)$ by Lemma~\ref{prop:decomposition into tensor product} and $(x_1\otimes x_2)T_{\lambda_1} = x_1T_{\lambda_1}\otimes x_2$ and $(x_1\otimes x_2)T_{\lambda_2} = x_1\otimes x_2T_{\lambda_2}$ for $x_1\in I_Q(\trivrep)$ and $x_2\in e_G(\sigma)$ by Remark~\ref{rem:action of aff to tensor product module}.
Hence we get
\begin{align*}
&R_R(I_Q(e_Q(\sigma)))\\
& \simeq \Hom_{C[T_{\lambda_1}]}(C[(T_{\lambda_1})^{\pm 1}],I_Q(\trivrep))\otimes \Hom_{C[T_{\lambda_2}]}(C[(T_{\lambda_2})^{\pm 1}],e_G(\sigma)).
\end{align*}
Since we also have the same formula for $I_{Q_1}(e_{Q_1}(\sigma))$ where $Q_1\supsetneq Q$ and $\St_Q\sigma$, the sequence in Lemma~\ref{lem:R preserves the exactness, definition of Steinberg} is equal to the sequence
\[
\begin{tikzcd}
\displaystyle\bigoplus_{Q_1\supsetneq Q}\Hom_{C[T_{\lambda_1}]}(C[(T_{\lambda_1})^{\pm 1}],I_{Q_1}(\trivrep))\otimes \Hom_{C[T_{\lambda_2}]}(C[(T_{\lambda_2})^{\pm 1}],e_G(\sigma))\arrow{d}\\
\Hom_{C[T_{\lambda_1}]}(C[(T_{\lambda_1})^{\pm 1}],I_Q(\trivrep))\otimes \Hom_{C[T_{\lambda_2}]}(C[(T_{\lambda_2})^{\pm 1}],e_G(\sigma))\arrow{d}\\
\Hom_{C[T_{\lambda_1}]}(C[(T_{\lambda_1})^{\pm 1}],\St_Q(\trivrep))\otimes \Hom_{C[T_{\lambda_2}]}(C[(T_{\lambda_2})^{\pm 1}],e_G(\sigma))\arrow{d}\\
0
\end{tikzcd}
\]
Hence it is sufficient to prove that the exact sequence $\bigoplus_{Q_1\supsetneq Q}I_{Q_1}(\trivrep)\to I_Q(\trivrep)\to \St_Q(\trivrep)\to 0$ is still exact after applying $\Hom_{C[T_{\lambda_1}]}(C[(T_{\lambda_1})^{\pm 1}],\cdot)$.
Let $R_1$ be a parabolic subgroup corresponding to $\Delta_R\cup \Delta_P$.
Then we have $R_1' = n_{w_Gw_{R_1}}\opposite{R_1}n_{w_Gw_{R_1}}^{-1}$ and $\Hom_{C[T_{\lambda_1}]}(C[(T_{\lambda_1})^{\pm 1}],\cdot) \simeq R_{R_1}$ as vector spaces.
Hence we may assume $\sigma = \trivrep$ and $R = R_1$, namely $R$ contains $P$.
\end{proof}

We prove Lemma~\ref{lem:R preserves the exactness, definition of Steinberg} assuming $\sigma = \trivrep$ and $R$ contains $P$.
Let $A\subset W_0^Q$ be an open subset and $w\in A$ a minimal element and $I_Q(\trivrep)_A$ a filtration defined in subsection~\ref{subsec:A filtration}.
Set $A' = A\setminus\{w\}$.
As in Remark~\ref{rem:exactness on each Bruhat cell}, put $I_{Q_1,A} = I_{Q_1}(\trivrep)\cap I_Q(\trivrep)_A$ and let $\St_{Q,A}$ be the image of $I_Q(\trivrep)_A$ in $\St_Q(\trivrep)$.
Then we have a commutative diagram with exact rows and columns (Remark~\ref{rem:exactness on each Bruhat cell}):
\[
\begin{tikzcd}
0\arrow{d} & 0\arrow{d} & 0\arrow{d} & \\
\bigoplus_{Q_1\supsetneq Q}I_{Q_1,A'}\arrow{r}\arrow{d} & I_{Q,A'}\arrow{r}\arrow{d} & \St_{Q,A'}\arrow{r}\arrow{d} & 0\\
\bigoplus_{Q_1\supsetneq Q}I_{Q_1,A }\arrow{r}\arrow{d} & I_{Q,A }\arrow{r}\arrow{d} & \St_{Q,A }\arrow{r}\arrow{d} & 0\\
\bigoplus_{Q_1\supsetneq Q}I_{Q_1,A}/I_{Q_1,A'}\arrow{r}\arrow{d} & I_{Q,A}/I_{Q,A'}\arrow{r}\arrow{d} & \St_{Q,A}/\St_{Q,A'}\arrow{r}\arrow{d} & 0\\
0 & 0 & 0.
\end{tikzcd}
\]
It is sufficient to prove that this diagram remains exact after applying $R_R$.
By induction on $\#A$, it is sufficient to prove that $R_R$ preserves of the exactness of the following three sequences.
\begin{gather}
\bigoplus_{Q_1\supsetneq Q}I_{Q_1,A}/I_{Q_1,A'}\to I_{Q,A}/I_{Q,A'}\to \St_{Q,A}/\St_{Q,A'}\to 0,\label{eq:first exact seq, for R_R}\\
0\to I_{Q_1,A'}\to I_{Q_1,A}\to I_{Q_1,A}/I_{Q_1,A'}\to 0\quad (Q_1\supset Q),\label{eq:second exact seq, for R_R}\\
0\to \St_{Q,A'}\to \St_{Q,A}\to \St_{Q,A}/\St_{Q,A'}\to 0.\label{eq:third exact seq, for R_R}
\end{gather}
First we prove that $R_R$ preserves the exactness of \eqref{eq:first exact seq, for R_R}.
Assume that $w\notin W^{Q_1}_0$ for any $Q_1\supsetneq Q$.
Then $\bigoplus_{Q_1\supsetneq Q}I_{Q_1,A}/I_{Q_1,A'} = 0$ by Lemma~\ref{lem:successive quotient of Bruhat filtration, small induction}.
Hence $I_{Q,A}/I_{Q,A'}\xrightarrow{\sim}\St_{Q,A}/\St_{Q,A'}$.
Therefore $R_R$ preserves the exactness of \eqref{eq:first exact seq, for R_R}.
If $w\in W^{Q_1}_0$ for some $Q_1\supsetneq Q$, then $I_{Q_1,A}/I_{Q_1,A'}\simeq I_{Q,A}/I_{Q,A'}$ for such $Q_1$ by Lemma~\ref{lem:successive quotient of Bruhat filtration, small induction}.
Hence $\St_{Q,A}/\St_{Q,A'} = 0$.
Moreover, fix $Q_0\supsetneq Q$ such that $w\in W^{Q_0}_0$.
Then $I_{Q,A}/I_{Q,A'}\simeq I_{Q_0,A}/I_{Q_0,A'}\hookrightarrow \bigoplus_{Q_1\supsetneq Q}I_{Q_1,A}/I_{Q_1,A'}$ is a splitting of $\bigoplus_{Q_1\supsetneq Q}I_{Q_1,A}/I_{Q_1,A'}\to I_{Q,A}/I_{Q,A'}$.
Hence \eqref{eq:first exact seq, for R_R} is exact after applying $R_R$.

Set $R' = n_{w_Gw_R}\opposite{R}n_{w_Gw_R}^{-1}$ and let $\lambda_{R'}^+$ be as in Proposition~\ref{prop:localization as Levi subalgebra}.
To prove that $R_R$ preserves the exactness of \eqref{eq:second exact seq, for R_R} and \eqref{eq:third exact seq, for R_R}, we analyze the action of $X = T_{\lambda_{R'}^+}$ on $I_{Q_1,A}/I_{Q_1,A'}$ for $Q_1\supset Q$.

\begin{lem}
The action of $X = T_{\lambda_{R'}^+}$ on $I_{Q_1,A}/I_{Q_1,A'}$ is a power of $p$ and it is $1$ if and only if $w_G(\Sigma_Q)\subset \Sigma_{R'}$ and $w\in W_{0,R'}w_Gw_Q$.
\end{lem}
\begin{proof}
We may assume $Q_1 = Q$ by Lemma~\ref{lem:successive quotient of Bruhat filtration, small induction}.
Note that since $\lambda_{R'}^+$ is anti-dominant, we have $E_{o_-}(\lambda_{R'}^+) = T_{\lambda_{R'}^+}$.
Hence by Proposition~\ref{prop:Bruhat filtration and action of A}, the action of $T_{\lambda_{R'}^+} = E_{o_-}(\lambda_{R'}^+)$ is given by $q(Q,n_w^{-1}\cdot \lambda_{R'}^+)\trivrep(E_{o_-}^Q(n_w^{-1}\cdot\lambda_{R'}^+))$.
Since $w\in W^Q_0$, $w(\Sigma^+_Q)\subset \Sigma^+$.
Therefore $n_w^{-1}\cdot \lambda_{R'}^+$ is anti-dominant with respect to $\Sigma_Q^+$ since $\lambda_{R'}^+$ is anti-dominant with respect to $\Sigma^+$.
Hence $\trivrep(E_{o_-}^Q(n_w^{-1}\cdot\lambda_{R'}^+)) = \trivrep(T_{n_w^{-1}\cdot\lambda_{R'}^+}^Q) = q_{n_w^{-1}\cdot\lambda_{R'}^+,Q}$.
Therefore the action of $T_{\lambda_{R'}^+}$ on this subquotient is given by a power of $p$ and it is $1$ if and only if $q(Q,n_w^{-1}\cdot \lambda_{R'}^+) = 1$ and $\ell_Q(n_w^{-1}\cdot\lambda_{R'}^+) = 0$.

Put $Q' = n_{w_Gw_Q}\opposite{Q}n_{w_Gw_Q}^{-1}$.
Note that $n_w^{-1}\cdot \lambda_{R'}^+$ is $Q$-negative if and only if $(n_{w_Gw_Q}n_w^{-1})\cdot \lambda_{R'}^+$ is $Q'$-positive since $(w_Gw_Q)^{-1}(\Sigma^+\setminus\Sigma_{Q'}^+) = \Sigma^-\setminus\Sigma_Q^-$.
Therefore $n_w^{-1}\cdot \lambda_{R'}^+$ is $Q$-negative if and only if $w_Gw_Qw^{-1}\in W_{0,R'}$ by Lemma~\ref{lem:condition on q != 1, positive version}, namely $w\in W_{0,R'}w_Gw_Q$.

The length of $n_w^{-1}\cdot\lambda_{R'}^+\in W_Q(1)$ is $0$ if and only if $\langle \alpha,\nu(n_w^{-1}\cdot \lambda_{R'}^+)\rangle = 0$ for any $\alpha\in\Sigma_Q$.
Since $\langle\beta,\nu(\lambda_{R'}^+)\rangle = 0$ if and only if $\beta\in \Sigma_{R'}$, the length of $\lambda_{R'}^+\in W_Q(1)$ is $0$ if and only if $w(\Sigma_Q)\subset \Sigma_{R'}$.
Since we have $w\in W_{0,R'}w_Gw_Q$, $w(\Sigma_Q)\subset \Sigma_{R'}$ if and only if  $w_G(\Sigma_Q)\subset  \Sigma_{R'}$.
\end{proof}
The subset $W_{0,R'}\cap {}^{Q'}W_0$ is closed in ${}^{Q'}W_0$.
Hence by Proposition~\ref{prop:tensor description of I_P}, $W_{0,R'}w_Gw_Q\cap W^Q_0$ is open in $W^Q_0$.
Therefore the exactness of \eqref{eq:second exact seq, for R_R} follows from the following general lemma.
In the following lemma, we call a subset $A$ of a partially ordered set open if $a\in A$ and $b\ge a$ implies $b\in A$.
\begin{lem}\label{lem:general lemma, for preserving exactness}
Let $\Gamma$ be a partially ordered set and $M$ a $C[X]$-module with the decomposition into $C$-submodules $M = \bigoplus_{\gamma\in\Gamma}M_\gamma$.
Assume the following.
\begin{itemize}
\item For each open subset $\Delta\subset \Gamma$, $M_\Delta = \bigoplus_{\gamma\in\Delta}M_\gamma$ is $C[X]$-stable.
\item For each open subset $\Delta\subset\Gamma$ and a minimal element $\gamma\in\Delta$, the action of $X$ on $M_\Delta/M_{\Delta\setminus\{\gamma\}}$ is given by $p^{n_\gamma}$ for some $n_\gamma\in\Z_{\ge 0}$.
\item The subset $\Gamma_0 = \{\gamma\in\Gamma_0\mid n_\gamma = 0\}$ is open.
\end{itemize}
Then for each open subset $\Delta\subset\Gamma$ and a minimal element $\gamma\in\Delta$, the homomorphism
\[
\Hom_{C[X]}(C[X^{\pm 1}],M_\Delta)
\to
\Hom_{C[X]}(C[X^{\pm 1}],M_\Delta/M_{\Delta\setminus\{\gamma\}})
\]
is surjective.
\end{lem}
\begin{proof}
Put $\Delta_0 = \Delta\cap \Gamma_0$.
We divide the map into two maps:
\begin{gather}
\Hom_{C[X]}(C[X^{\pm 1}],M_\Delta)\to \Hom_{C[X]}(C[X^{\pm 1}],M_\Delta/M_{\Delta_0\setminus\{\gamma\}})\label{eq:surjective for Hom, in kernel X is invertible},\\
\Hom_{C[X]}(C[X^{\pm 1}],M_\Delta/M_{\Delta_0\setminus\{\gamma\}})\to \Hom_{C[X]}(C[X^{\pm 1}],M_\Delta/M_{\Delta\setminus\{\gamma\}})\label{eq:surjective for Hom, power of p}.
\end{gather}
We prove that both maps are surjective.

\noindent (1)
We prove the surjectivity of \eqref{eq:surjective for Hom, in kernel X is invertible}.
Since $M_{\Delta_0\setminus\{\gamma\}}$ have a filtration such that $X$ in invertible on successive quotients, $X$ is also invertible on $M_{\Delta_0\setminus\{\gamma\}}$.
Hence the claim follows from the following claim.
\begin{claim}
Let $N$ be a $C[X]$-module and assume that $X$ is invertible on $N$.
Then $\Ext^1_{C[X]}(C[X^{\pm 1}],N) = 0$.
\end{claim}
\begin{proof}[Proof of Claim]
Let $0\to N\to L\to C[X^{\pm 1}]\to 0$ be an exact sequence of $C[X]$-modules.
Since $X$ is invertible on $N$ and $C[X^{\pm 1}]$, $X$ is also invertible on $L$.
Namely $0\to N\to L\to C[X^{\pm 1}]\to 0$ is also an exact sequence of $C[X^{\pm 1}]$-modules.
Hence $\Ext^1_{C[X]}(C[X^{\pm 1}],N) = \Ext^1_{C[X^{\pm 1}]}(C[X^{\pm 1}],N)$.
This is obviously zero.
\end{proof}

\noindent
(2)
Next we prove the surjectivity of \eqref{eq:surjective for Hom, power of p} assuming $\gamma\in\Gamma_0$.
We have $M_\Delta/M_{\Delta_0\setminus\{\gamma\}}\supset M_{\Delta_0}/M_{\Delta_0\setminus\{\gamma\}}\simeq M_\gamma\simeq M_\Delta/M_{\Delta\setminus\{\gamma\}}$.
Hence from the following diagram
\[
\begin{tikzcd}
\Hom_{C[X]}(C[X^{\pm 1}],M_{\Delta_0}/M_{\Delta_0\setminus\{\gamma\}})\arrow{rd}{\sim}\arrow[hookrightarrow]{d}\\
\Hom_{C[X]}(C[X^{\pm 1}],M_\Delta/M_{\Delta_0\setminus\{\gamma\}})\arrow{r} & \Hom_{C[X]}(C[X^{\pm 1}],M_\Delta/M_{\Delta\setminus\{\gamma\}}),
\end{tikzcd}
\]
\eqref{eq:surjective for Hom, power of p} is surjective in this case.

\noindent
(3)
Finally we prove the surjectivity of \eqref{eq:surjective for Hom, power of p} assuming $\gamma\notin \Gamma_0$.
Note that in general we have $\Hom_{C[X]}(C[X^{\pm 1}],N) = \{(m_n)_{n\in\Z_{\ge 0}}\mid m_n\in N, Xm_n = m_{n - 1}\}$ by $\varphi\mapsto (\varphi(X^{-n}))$.
Recall that the action of $X$ on $M_\Delta/M_{\Delta\setminus\{\gamma\}}$ is given by $p^{n_\gamma}$ with $n_\gamma > 0$. (We have assumed that $\gamma\notin\Gamma_0$.)
Hence to give an element in $\Hom_{C[X]}(C[X^{\pm 1}],M_\Delta/M_{\Delta\setminus\{\gamma\}})$ is equivalent to give a sequence of elements $(m^{(\gamma)}_n)$ in $M_\gamma$ such that $p^{n_\gamma}m^{(\gamma)}_n = m^{(\gamma)}_{n - 1}$.
We prove that we can extend this element to $m_n = (m^{(\delta)}_n)_{\delta\in\Delta,n\in\Z_{\ge 0}}$ in $M_\Delta = \bigoplus_{\delta\in\Delta}M_\delta$ such that $Xm_n = m_{n - 1}$.
Since $p^{n_\gamma}m_n^{(\gamma)} = m_{n - 1}^{(\gamma)}$ with $n_\gamma > 0$, we have $m_n^{(\gamma)}\in \bigcap_k p^kM_\gamma$.
We prove that we can take an extension $m_n = (m^{(\delta)}_n)$ from $\bigcap_k p^kM_\Delta$.

According to the decomposition $M = \bigoplus_{\gamma\in \Gamma}M_\gamma$, we have a linear map $X_{\gamma_1,\gamma_2}\in \Hom_C(M_{\gamma_2},M_{\gamma_1})$ such that $Xm_\gamma = \sum_{\gamma'\in \Gamma}X_{\gamma',\gamma}m_\gamma$ for $m_\gamma\in M_\gamma\subset M$.
By the assumption $X_{\gamma_1,\gamma_2} = 0$ if $\gamma_1\not\ge\gamma_2$.
We also have $X_{\gamma',\gamma'} = p^{n_{\gamma'}}$ with $n_{\gamma'} \ge 0$.
The condition $Xm_n = m_{n - 1}$ is equivalent to
\begin{equation}\label{eq:condition from Hom(C X^pm 1)}
m_{n - 1}^{(\delta_1)} = \sum_{\delta_2 \le \delta_1}X_{\delta_1,\delta_2}m_n^{(\delta_2)}.
\end{equation}
We prove the existence of $m_n = (m_n^{(\delta)})$ using a triangular argument, namely we take $m_n^{(\delta)}$ which satisfies \eqref{eq:condition from Hom(C X^pm 1)} inductively on $\delta$.

Let $\delta\in\Delta\setminus\{\gamma\}$ and assume that we have taken $m_n^{(\delta')}$ for $\delta'\in\Delta$ such that $\delta' < \delta$.
Since $m_n^{(\delta')}\in \bigcap_k p^kM_{\delta'}$ for any $\delta' < \delta$, we can take $x_n^{(\delta')}\in \bigcap_k p^kM_{\delta'}$ such that $p^{n_\delta}x_n^{(\delta')} = m_n^{(\delta')}$.
We also have $x_{n - 1}^{(\delta)}\in \bigcap_k p^kM_{\delta}$ such that $p^{n_\delta}x_{n - 1}^{(\delta)} = m_{n - 1}^{(\delta)}$.
Then define $m_n^{(\delta)} = x_{n - 1}^{(\delta)} - \sum_{\delta' < \delta}X_{\delta,\delta'}x_n^{(\delta')}$ and it satisfies $m_{n - 1}^{(\delta)} = p^{n_{\delta}}m_n^{(\delta)} + \sum_{\delta' < \delta}X_{\delta,\delta'}m_n^{(\delta')}$.
Since $X_{\delta,\delta} = p^{n_\delta}$, it means that \eqref{eq:condition from Hom(C X^pm 1)} holds for $\delta_1 = \delta$.
\end{proof}

For \eqref{eq:third exact seq, for R_R}, we use the following lemma with Lemma~\ref{lem:general lemma, for preserving exactness}.
\begin{lem}
Set $B = W_0^Q\setminus\bigcup_{Q_1\supsetneq Q}W_0^{Q_1}$.
For $w\in B$, we define $\St_{Q,w}\subset \St_Q(\trivrep)$ by the image of
\[
\{\varphi\in I_Q(\trivrep)\mid \varphi(T_{n_v}) = 0\ (v\in W_0^Q\setminus\{w\})\}
\hookrightarrow I_Q(\trivrep)\to \St_{Q}(\trivrep).
\]
Then we have $\St_Q(\trivrep) = \bigoplus_{w\in B}\St_{Q,w}$ and for any open $A\subset W_0^Q$, we have $\bigoplus_{w\in A\cap B}\St_{Q,w} = \St_{Q,A}$.
\end{lem}
\begin{proof}
We prove $\sum_{v\in A\cap B}\St_{Q,v} = \St_{Q,A}$ by induction on $\#A$.
Let $w\in A$ be a minimal element and set $A' = A\setminus\{w\}$.
If $w\notin B$ then $A\cap B = A'\cap B$ and $\St_{Q,A} = \St_{Q,A'}$ by Lemma~\ref{lem:successive quotient of filtration on St}.
If $w\in B$, then $\St_{Q,A}/\St_{Q,A'}\simeq I_{Q,A}/I_{Q,A'}\simeq \{\varphi\in I_Q(\trivrep)\mid \varphi(T_{n_v}) = 0\ (v\in W_0^Q\setminus\{w\})\}$.
Therefore $\St_{Q,A'} + \St_{Q,w} = \St_{Q,A}$.
By inductive hypothesis, $\St_{Q,A'} = \sum_{v\in A'\cap B}\St_{Q,v}$.
We get $\sum_{v\in A\cap B}\St_{Q,v} = \St_{Q,A}$.

Let $x_w\in \St_{Q,w}\ (w\in B)$ such that $\sum_{w\in B} x_w = 0$.
Assume that there exists $w\in B$ such that $x_w\ne 0$ and assume that $w$ is minimal subject to $x_w\ne 0$.
Put $A' = \{v\in W_0^Q\mid \text{$v\ge v_1$ for some $v_1\in W^Q_0\setminus\{w\}$ such that $x_{v_1}\ne 0$}\}$.
Then $A'$ is open and $w\notin A'$.
Hence $\sum_{v\ne w}x_v \in \St_{Q,A'}$.
Set $A = A'\cup \{w\}$.
Let $y_w\in I_Q(e_Q(\sigma))$ such that $y_w(T_{n_v}) = 0$ for $v\in W_0^Q\setminus\{w\}$ and the image of $y_w$ in $\St_{Q,w}$ is $x_w$.
Then the image of $y_w$ under
\[
\{\varphi\in I_Q(\trivrep)\mid \varphi(T_{n_v}) = 0\ (v\in W^Q_0\setminus\{w\})\}
\simeq I_{Q,A}/I_{Q,A'}
\simeq \St_{Q,A}/\St_{Q,A'}
\]
is equal to the image of $x_w$.
Here the last isomorphism is by Lemma~\ref{lem:successive quotient of filtration on St}.
Since $\sum_{v\ne w}x_v\in \St_{Q,A'}$, it is equal to the image of $\sum_{v\in B}x_v$ by $\St_{Q,A}\to \St_{Q,A}/\St_{Q,A'}$.
Since $\sum_{v\in B}x_v = 0$, we have $y_w = 0$.
This contradicts to $x_w \ne 0$.
\end{proof}

\subsection{$R_P$ and Steinberg modules}
As in the previous subsection, let $P$ be a parabolic subgroup and $\sigma$ an $\mathcal{H}_P$-module which has the extension to $\mathcal{H}$.
We prove Proposition~\ref{prop:right adj of St} in this subsection.
As in the case of $L_R$, we start with the following lemma.
\begin{lem}\label{lem:extension and right adjoint}
Assume that $\bigcap_{n\in\Z_{\ge 0}}p^n\sigma = 0$.
We have
\[
R_R(e_G(\sigma)) =
\begin{cases}
e_R(R^P_{R\cap P}(\sigma)) & (\Delta = \Delta_R\cup\Delta_P),\\
0 & (\text{otherwise}).
\end{cases}
\]
\end{lem}
\begin{proof}

First assume that $\Delta\ne\Delta_R\cup\Delta_P$ and we prove $R_R(e_G(\sigma)) = 0$.
Let $R_1$ be a parabolic subgroup corresponding to $\Delta_R\cup\Delta_P$.
It is sufficient to prove that $R_{R_1}(e_G(\sigma)) = 0$.
Take $\lambda_{R'_1}^+$ as in Proposition~\ref{prop:localization as Levi subalgebra} for $R_1'$ where $R_1' = n_{w_Gw_{R_1}}\opposite{R_1}n_{w_Gw_{R_1}}^{-1}$.
Let $P_2$ be a parabolic subgroup corresponding to $\Delta\setminus\Delta_P$.
Since $R'_1\ne G$ and $P\subset R'_1$, there exists $\alpha\in\Delta\setminus\Delta_{R'_1} = \Delta_{P_2}\setminus\Delta_{R'_1}$ and for such $\alpha$, we have $\langle\alpha,\nu(\lambda_{R'_1}^+)\rangle < 0$.
Hence the length of $\lambda_{R'_1}^+$ as an element of $W_{P_2}(1)$ is positive.
Therefore $e_G(\sigma)(T_{\lambda_{R'_1}^+})$ is divided by $p$.
Hence, the image of any $\varphi\in \Hom_{(\mathcal{H}_{R'_1}^+,j_{R'_1}^+)}(\mathcal{H}_{R'_1},e_G(\sigma)) = \Hom_{C[T_{\lambda_{R'_1}^+}]}(C[(T_{\lambda_{R'_1}^+})^{\pm 1}],e_G(\sigma))$ is in $\bigcap_{n\in\Z_{\ge 0}}p^ne_G(\sigma) = 0$.
(As $C$-modules, we have $e_G(\sigma) = \sigma$.)
We get $R_{R_1}(e_G(\sigma)) = 0$.

Now assume that $\Delta_R\cup \Delta_P = \Delta$, or in other words, $\Delta_{P_2}\subset \Delta_R$.
Since $\Delta = \Delta_P\cup \Delta_{P_2}$ and $\Delta_R = \Delta_{R\cap P}\cup \Delta_{P_2}$ are orthogonal decompositions, we have decompositions $W_0 = W_{0,P}\times W_{0,P_2}$ and $W_{0,R} = W_{0,P\cap R}\times W_{0,P_2}$ as Coxeter groups.
Hence we have $w_G = w_Pw_{P_2}$ and $w_R = w_{P\cap R}w_{P_2}$.
We have  $w_Gw_R = w_Pw_{P\cap R}$.
Set $R' = n_{w_Gw_R}\opposite{R}n_{w_Gw_R}^{-1}$.
We also have $P\cap  R' = n_{w_Pw_{P\cap R}}\opposite{(P\cap R)}n_{w_Pw_{P\cap R}}^{-1}$.
By the definition of the right adjoint functors, we have
\begin{align*}
R_R(e_G(\sigma)) & = n_{w_Gw_{R}}^{-1}\Hom_{(\mathcal{H}_{R'}^+,j_{R'}^+)}(\mathcal{H}_{R'},e_G(\sigma)),\\
R^P_{P\cap R}(\sigma) & = n_{w_Pw_{P\cap R}}^{-1}\Hom_{(\mathcal{H}_{P\cap R'}^{P+},j_{P\cap R'}^{P+})}(\mathcal{H}_{P\cap R'},\sigma)
\end{align*}
Since $w_Gw_{R} = w_Pw_{P\cap R}$ from the assumption, replacing $R'$ with $R$, with Lemma~\ref{lem:extension and twist}, it is sufficient to prove that $A = \Hom_{(\mathcal{H}_R^+,j_R^+)}(\mathcal{H}_R,e_G(\sigma))$ is isomorphic to $e_R(B)$ where $B = \Hom_{(\mathcal{H}_{P\cap R}^{P+},j_{P\cap R}^{P+})}(\mathcal{H}_{P\cap R},\sigma)$.

The map $\varphi\mapsto (\varphi((T_{\lambda_R^+}^R)^{-n}))$ gives an isomorphism
\[
A \simeq \{(x_n)_{n\in\Z_{\ge 0}}\mid x_{n + 1}e_G(\sigma)(T_{\lambda_R^+}) = x_n,\ x_n\in e_G(\sigma)\}
\]
Since $\lambda_R^+$ is anti-dominant, $T_{\lambda_R^+} = E_{o_-}(\lambda_R^+)$ and $T^P_{\lambda_R^+} = E^P_{o_{-,P}}(\lambda_R^+)$ by \eqref{eq:E_o for lambda}.
Hence we have $T_{\lambda_R^+} = E_{o_-}(\lambda_R^+) = j_P^{-*}(E^P_{o_{-,P}}(\lambda_R^+)) = j_P^{-*}(T^P_{\lambda_R^+})$ by Proposition~\ref{lem:image of E by j}.
Therefore $e_G(\sigma)(T_{\lambda_R^+}) = \sigma(T^P_{\lambda_R^+})$.
Hence we have
\[
A \simeq \{(x_n)_{n\in\Z_{\ge 0}}\mid x_{n + 1}\sigma(T^P_{\lambda_R^+}) = x_n,\ x_n\in \sigma\}
\]
Since $R\supset P_2$, we have $\Sigma^+\setminus\Sigma_R^+ = \Sigma_P^+\setminus\Sigma_{P\cap R}^+$.
Therefore we can take $\lambda_R^+$ as $\lambda_{R\cap P}^{P+}$.
Hence
\[
\{(x_n)_{n\in\Z_{\ge 0}}\mid x_{n + 1}\sigma(T^P_{\lambda_R^+}) = x_n\} \simeq B.
\]
Namely there exists an isomorphism $A\simeq B$ as vector spaces which is characterized by $\varphi((T^R_{\lambda_R^+})^{-n}) = \psi((T^{P\cap R}_{\lambda_R^+})^{-n})$ for any $n\in\Z_{\ge 0}$ where $\varphi\in A$ corresponds to $\psi\in B$ by this isomorphism.

We prove that this isomorphism is $(\mathcal{H}_{P\cap R}^{R-},j_{P\cap R}^{R-*})$-equivariant.
Let $w\in W_{P\cap R}^{R-}(1)$.
By the assumption $\Delta_{P_2}\subset \Delta_R$, $\Sigma_R^+\setminus\Sigma_{P\cap R}^+ = \Sigma_{P_2}^+ = \Sigma^+\setminus\Sigma_P^+$.
Hence $w\in W_P^-(1)$.
Take $k\in\Z_{\ge 0}$ such that $w(\lambda_R^+)^k$ is $R$-positive.
Since $\varphi$ is $(\mathcal{H}_R^+,j_R^+)$-equivariant, using Lemma~\ref{lem:image of E by j}, we have
\begin{align*}
(\varphi j_{P\cap R}^{R-*}(E^{P\cap R}_{o_{-,P\cap R}}(w)))((T^R_{\lambda_R^+})^{-n})
& = \varphi(E_{o_{-,R}}^R(w)(T^R_{\lambda_R^+})^{-n})\\
& = \varphi(E_{o_{-,R}}^R(w(\lambda_R^+)^k)(T^R_{\lambda_R^+})^{-(n + k)})\\
& = \varphi((T^R_{\lambda_R^+})^{-(n + k)} j_R^+(E_{o_{-}}(w(\lambda_R^+)^k)))\\
& = \varphi((T^R_{\lambda_R^+})^{-(n + k)})e_G(\sigma)(E_{o_-}(w(\lambda_R^+)^k)).
\end{align*}

Since $\lambda_R^+$ is in the center of $W_R(1)$ and $P_2\subset R$, $\lambda_R^+$ is also in the center of $W_{P_2}(1)$, hence we have $\langle \alpha,\nu(\lambda_R^+)\rangle = 0$ for any $\alpha\in\Delta_{P_2}$.
Hence for any $\alpha\in\Sigma^+\setminus\Sigma_P^+$, we have $\langle \alpha,\nu(\lambda_R^+)\rangle = 0$ since $\Sigma^+\setminus\Sigma_P^+ = \Sigma_{P_2}^+$.
Therefore $\lambda_R^+$ is both $P$-positive and $P$-negative.
Recall that $w$ is also $P$-negative.
Therefore $w(\lambda_R^+)^k$ is also $P$-negative.
Hence $e_G(\sigma)(E_{o_-}(w(\lambda_R^+)^k)) = \sigma(E^P_{o_{-,P}}(w(\lambda_R^+)^k))$ by the definition of the extension and Lemma~\ref{lem:image of E by j}.
Therefore 
\begin{align*}
\varphi((T^R_{\lambda_R^+})^{-(n + k)})e_G(\sigma)(E_{o_-}(w(\lambda_R^+)^k))
& = \varphi((T^R_{\lambda_R^+})^{-(n + k)})\sigma(E^P_{o_{-,P}}(w(\lambda_R^+)^k))\\
& = \psi((T^{P\cap R}_{\lambda_{R}^+})^{-(n + k)})\sigma(E^P_{o_{-,P}}(w(\lambda_R^+)^k)).
\end{align*}
Since $w(\lambda_R^+)^k\in W_{P\cap R}(1)$ is $R$-positive, we have $w(\lambda_R^+)^k\in W_{P\cap R}^{P+}(1)$.
Therefore $E^P_{o_{-,P}}(w(\lambda_R^+)^k) = j_{P\cap R}^{P+}(E^{P\cap R}_{o_{-,P\cap R}}(w(\lambda_R^+)^k))$ by Lemma~\ref{lem:image of E by j}.
Hence
\begin{align*}
\psi((T^{P\cap R}_{\lambda_{R}^+})^{-(n + k)})\sigma(E^P_{o_{-,P}}(w(\lambda_R^+)^k))
& = \psi((T^{P\cap R}_{\lambda_{R}^+})^{-(n + k)}E^{P\cap R}_{o_{-,P\cap R}}(w(\lambda_R^+)^k))\\
& = (\psi E^{P\cap R}_{o_{-,P\cap R}}(w))((T^{P\cap R}_{\lambda_{R}^+})^{-n}).
\end{align*}
Hence $A\simeq B$ as $(\mathcal{H}_{P\cap R}^{R-},j_{P\cap R}^{R-*})$-modules.

Let $w\in W_{P_2\cap R,\aff}(1)$.
Then $w\in W_{P_2,\aff}(1)$.
Since $\Sigma^+\setminus\Sigma_R^+\subset \Sigma^+\setminus\Sigma_{P_2}^+ = \Sigma_{P}^+$ and for any $\alpha\in\Sigma_P^+$ is orthogonal to elements in $\nu(\Lambda(1)\cap W_{P_2,\aff}(1))$, $w$ is both $R$-positive and $R$-negative.
Hence by Corollary~\ref{cor:image by j, positie and negative case}, we have $j_R^+(T_w^{R*}) = T_w^*$.
Therefore, for $\varphi\in A$, we have
\begin{align*}
(\varphi T_w^{R*})((T^R_{\lambda_R^+})^{-n})
& =
\varphi(T_w^{R*}(T^R_{\lambda_R^+})^{-n})\\
& =
\varphi((T^R_{\lambda_R^+})^{-n} T_w^{R*})\\
& =
\varphi((T^R_{\lambda_R^+})^{-n})e_G(\sigma)(j_R^+(T_w^{R*}))\\
& =
\varphi((T^R_{\lambda_R^+})^{-n})e_G(\sigma)(T_w^{*})\\
& =
\varphi((T^R_{\lambda_R^+})^{-n}).
\end{align*}
Therefore $T_w^{R*} = 1$ on $A$.
Hence $A\simeq e_R(B)$ by the definition of the extension.
\end{proof}

\begin{proof}[Proof of Proposition~\ref{prop:right adj of St}]
By Lemma~\ref{lem:R preserves the exactness, definition of Steinberg}, we have
\[
\bigoplus_{Q_1\supsetneq Q}R_R(I_{Q_1}(e_{Q_1}(\sigma)))\to R_R(I_Q(e_Q(\sigma)))\to R_R(\St_Q(\sigma))\to 0
\]
By Proposition~\ref{prop:parabolic induction and its adjoint},
\[
	\bigoplus_{Q'\supsetneq Q}I_{Q'\cap R}^R(R_{Q'\cap R}^{Q'}(e_{Q'}(\sigma)))\to I_{Q\cap R}^R(R_{Q\cap R}^{Q}(e_Q(\sigma)))\to R_R(\St_Q(\sigma))\to 0.
\]
If $\Delta_Q\ne \Delta_{Q\cap R}\cup \Delta_P$, then $R_{Q\cap R}^{Q}(e_Q(\sigma)) = 0$ by Lemma~\ref{lem:extension and right adjoint}.
Hence $R_R(\St_Q(\sigma)) = 0$.
Assume that $\Delta_Q = \Delta_{Q\cap R}\cup \Delta_P$.
Then $R_{Q\cap R}^{Q}(e_Q(\sigma)) = e_{Q\cap R}(R_{P\cap R}^P(\sigma))$.
Let $Q'\supsetneq Q$.
If $\Delta_{Q'} = \Delta_{Q'\cap R}\cup \Delta_P$, then $R_{Q'\cap R}^{Q'}(e_{Q'}(\sigma)) = e_{Q'\cap R}(R_{P\cap R}^P(\sigma))$ and $Q'\cap R\supsetneq Q\cap R$.
Otherwise, it is zero.
Putting $Q'' = Q'\cap R$, we have
\begin{multline*}
	\bigoplus_{R\supset Q''\supsetneq Q\cap R}I_{Q''}^R(e_{Q''}(R_{P\cap R}^R(\sigma)))\to I^R_{Q\cap R}(e_{Q\cap R}(R_{P\cap R}^P(\sigma)))\\
	\to R_R(\St_Q(\sigma))\to 0.
\end{multline*}
Therefore we have $R_R(\St_Q(\sigma)) = \St^R_{Q\cap R}(R_{P\cap R}^P(\sigma))$.
\end{proof}

\subsection{Supersingular modules}
Assume that $p = 0$ in $C$.
\begin{prop}\label{prop:vanishing of adjoint+supsersingular}
Let $P\subsetneq G$ be a proper parabolic subgroup and $\pi$ a supersingular $\mathcal{H}$-module.
Then we have $L_P(\pi) = R_P(\pi) = 0$.
\end{prop}

We need a lemma.
\begin{lem}
Assume that $\lambda\in Z(\Lambda(1))$ satisfies that for any $w\in W_0$, we have $n_w\cdot \lambda = \lambda$ if and only if $w(\nu(\lambda)) = \nu(\lambda)$.
Put $\mathcal{O} = W(1)\cdot \lambda$.
Then for any orientation $o$ and $n\in\Z_{\ge 0}$, we have $z_\mathcal{O}^nE_o(\lambda) = E_o(\lambda)^{n + 1}$.
\end{lem}
\begin{proof}
Since $\lambda$ is in the center of $\Lambda(1)$, we have $\mathcal{O} = \{n_w\cdot \lambda\mid w\in W_0\}$.
We also have $\mathcal{O} = \{n_w\cdot \lambda\mid w\in W_0/\Stab_{W_0}(\nu(\lambda))\}$ by the condition on $\lambda$.
Therefore we have $z_\mathcal{O} = \sum_{w\in W_0/\Stab_{W_0}\nu(\lambda)}E_o(n_w\cdot \lambda)$.
Hence
\[
z_\mathcal{O}^2 = \sum_{v_1,v_2\in W_0/\Stab_{W_0}\nu(\lambda)}E_o(n_{v_1}\cdot \lambda)E_o(n_{v_2}\cdot \lambda).
\]
If $v_1,v_2\in W_0$ does not belong to the same coset in $W_0/\Stab_{W_0}\nu(\lambda)$, then $v_1(\nu(\lambda))$ and $v_2(\nu(\lambda))$ are not in the same closed chamber.
Hence $E_o(n_{v_1}\cdot \lambda)E_o(n_{v_2}\cdot \lambda) = 0$ by \eqref{eq:product formula} and Lemma~\ref{lem:length, on lambda}.
Therefore 
\[
z_\mathcal{O}^2 = \sum_{v\in W_0/\Stab_{W_0}\nu(\lambda)}E_o(n_{v}\cdot \lambda)^2.
\]
Inductively, we get 
\[
z_\mathcal{O}^n = \sum_{v\in W_0/\Stab_{W_0}\nu(\lambda)}E_o(n_{v}\cdot \lambda)^n.
\]
Hence 
\[
z_\mathcal{O}^nE_o(\lambda) = \sum_{v\in W_0/\Stab_{W_0}\nu(\lambda)}E_o(n_{v}\cdot \lambda)^nE_o(\lambda).
\]
If $v\notin \Stab_{W_0}\nu(\lambda)$, then $E_{o}(n_v\cdot \lambda)E_{o}(\lambda) = 0$.
Hence 
\[
z_\mathcal{O}^nE_o(\lambda) = E_o(\lambda)^{n + 1}.
\]
We get the lemma.
\end{proof}

\begin{proof}[Proof of Proposition~\ref{prop:vanishing of adjoint+supsersingular}]
Let $\lambda = \lambda_P^-\in Z(W_P(1))$ be as in Proposition~\ref{prop:localization as Levi subalgebra}.
If $w\in W_0$ fixes $\nu(\lambda)$, then $w\in W_{0,P}$.
Since $\lambda\in Z(W_P(1))$, we have $n_w\cdot \lambda = \lambda$.
Namely, $\lambda$ satisfies the condition of the above lemma.
We also have $L_P(\pi) = \pi E_{o_-}(\lambda)^{-1}$.

Set $\mathcal{O} = W(1)\cdot \lambda$.
Let $x\in \pi$.
Since $\pi$ is supersingular, there exists $n\in\Z_{>0}$ such that $xz_\mathcal{O}^n = 0$.
Hence $xz_\mathcal{O}^nE_{o_-}(\lambda) = 0$.
Therefore we have $xE_{o_-}(\lambda)^{n + 1} = 0$ by the above lemma.
Hence the image of $x$ in $L_P(\pi) = \pi E_{o_-}(\lambda)^{-1}$ is zero.
We get $L_P(\pi) = 0$.

Next we prove $R_P(\pi) = n_{w_Gw_P}^{-1}\Hom_{(\mathcal{H}_{P'}^+,j_{P'}^+)}(\mathcal{H}_{P'},\pi) = 0$ where $P' = n_{w_Gw_P}\opposite{P}n_{w_Gw_P}^{-1}$.
Let $\lambda = \lambda_{P'}^+\in Z(W_{P'}(1))$ be as in Proposition~\ref{prop:localization as Levi subalgebra}.
Then again $\lambda$ satisfies the condition of the above lemma.
Take $n\in\Z_{>0}$ such that $\pi(z_\mathcal{O}^n) = 0$.
Let $\varphi\in \Hom_{(\mathcal{H}_{P'}^+,j_{P'}^+)}(\mathcal{H}_{P'},\pi)$ and $X\in \mathcal{H}_{P'}$.
Since $j_{P'}^+(E^{P'}_{o_{-,P'}}(\lambda)) = E_{o_-}(\lambda)$ by Lemma~\ref{lem:image of E by j}, with the previous lemma, we have 
\begin{align*}
\varphi(X) & = \varphi(XE^{P'}_{o_{-,P'}}(\lambda)^{-(n + 1)})j_{P'}^+(E^{P'}_{o_{-,P'}}(\lambda)^{n + 1})\\
& = \varphi(XE^{P'}_{o_{-,P'}}(\lambda)^{-(n + 1)})E_{o_-}(\lambda)^{n + 1}\\
& = \varphi(XE^{P'}_{o_{-,P'}}(\lambda)^{-(n + 1)})z_\mathcal{O}^nE_{o_-}(\lambda) = 0.
\end{align*}
We get the proposition.
\end{proof}

\subsection{Simple modules}
Assume that $C$ is an algebraically closed field of characteristic $p$.
\begin{thm}\label{thm:adjoint and simple modules}
Let $P$ be a parabolic subgroup, $\sigma$ a simple supersingular right $\mathcal{H}_P$-module and $Q$ a parabolic subgroup between $P$ and $P(\sigma)$.
For a parabolic subgroup $R$, we have
\[
	L_R(I(P,\sigma,Q)) =
	\begin{cases}
	I_R(P,\sigma,Q\cap R) & (P\subset R,\ \Delta(\sigma)\subset \Delta_Q\cup \Delta_R),\\
	0 & (\text{otherwise}).
	\end{cases}
\]
and
\[
	R_R(I(P,\sigma,Q)) = 
	\begin{cases}
	I_R(P,\sigma,Q) & (Q\subset R),\\
	0 & (\text{otherwise}).
	\end{cases}
\]
\end{thm}
\begin{proof}
By Corollary~\ref{cor:left adjoint and parabolic induction}, we have 
\[
	L_R(I(P,\sigma,Q)) = L_R(I_{P(\sigma)}(\St^{P(\sigma)}_Q(\sigma))) = I_{P(\sigma)\cap R}^R(L^{P(\sigma)}_{P(\sigma)\cap R}(\St^{P(\sigma)}_Q(\sigma)))
\]
If $\Delta_Q\cup (\Delta(\sigma)\cap \Delta_{R}) \ne \Delta(\sigma)$, then it is zero by Proposition~\ref{prop:left adj of St}.
Since $Q\subset P(\sigma)$, we have $\Delta_Q\subset \Delta(\sigma)$.
Hence $\Delta_Q = \Delta(\sigma)\cap \Delta_Q$ and, therefore we have $\Delta_Q\cup (\Delta(\sigma)\cap \Delta_R) = \Delta(\sigma)\cap (\Delta_Q\cup \Delta_R)$.
Hence $\Delta_Q\cup (\Delta(\sigma)\cap \Delta_R) = \Delta(\sigma)$ if and only if $\Delta(\sigma)\subset \Delta_Q\cup \Delta_R$.
Also from Proposition~\ref{prop:left adj of St}, if $\Delta(\sigma)\subset \Delta_Q\cup \Delta_R$, we have 
\[
	L_R(I(P,\sigma,Q))
	=
	I_{P(\sigma)\cap R}^R(\St^{P(\sigma)\cap R}_{Q\cap R}(L_{P\cap R}^P(\sigma))),
\]
here we use $Q\subset P(\sigma)$.
By Proposition~\ref{prop:vanishing of adjoint+supsersingular}, this is zero if $P\cap R \ne P$, namely $P\not\subset R$.
If $P\subset R$, then 
\[
	L_R(I(P,\sigma,Q))
	=
	I_{P(\sigma)\cap R}^R(\St^{P(\sigma)\cap R}_{Q\cap R}(\sigma))
	=
	I(P,\sigma,Q\cap R).
\]

For the functor $R$, by Proposition~\ref{prop:parabolic induction and its adjoint}, we have
\[
R_R(I(P,\sigma,Q)) = R_R(I_{P(\sigma)}(\St_Q^{P(\sigma)}(\sigma))) = I_{P(\sigma)\cap R}^R(R_{P(\sigma)\cap R}^{P(\sigma)}(\St_Q^{P(\sigma)}(\sigma))).
\]
By Proposition~\ref{prop:right adj of St}, this is zero if $\Delta_Q\ne \Delta_{Q\cap R\cap P(\sigma)}\cup \Delta_P$.
Note that we have $Q\cap R\cap P(\sigma) = Q\cap R$ since $Q\subset P(\sigma)$.
If $\Delta_Q = \Delta_{Q\cap R}\cup \Delta_P$, then we have
\[
R_R(I(P,\sigma,Q)) = I_{P(\sigma)\cap R}^R(\St_{Q\cap R}^{P(\sigma)\cap R}(R_{P\cap R}^P(\sigma))).
\]
By Proposition~\ref{prop:vanishing of adjoint+supsersingular}, this is zero if $P\cap R\ne P$.
If $P\cap R = P$, namely $P\subset R$, then we have
\[
R_R(I(P,\sigma,Q)) = I_{P(\sigma)\cap R}^R(\St_{Q\cap R}^{P(\sigma)\cap R}(\sigma)) = I_R(P,\sigma,Q\cap R).
\]
Hence we get
\[
R_R(I(P,\sigma,Q))
=
\begin{cases}
I_R(P,\sigma,Q\cap R) & (P\subset R,\Delta_Q = \Delta_{Q\cap R}\cup \Delta_P),\\
0 & (\text{otherwise}).
\end{cases}
\]
If $P\subset R$, then $\Delta_{Q\cap R}\supset \Delta_P$.
Hence $\Delta_Q = \Delta_{Q\cap R}\cup \Delta_P$ implies $\Delta_Q = \Delta_{Q\cap R}$, namely $Q\subset R$.
On the other hand, if $Q\subset R$ then $P\subset R$ and $\Delta_Q = \Delta_Q\cup \Delta_P = \Delta_{Q\cap R}\cup \Delta_P$.
We get the theorem.
\end{proof}

\end{document}